\documentclass{amsart}
\usepackage{amsfonts}

\usepackage{amsmath}

\usepackage{amssymb}
\usepackage{amsthm}
\usepackage{mathrsfs}
\usepackage{mathtools}
\usepackage{esint}
\newtheorem{theorem}{Theorem}[section]
\newtheorem{corollary}[theorem]{Corollary}
\newtheorem{lemma}[theorem]{Lemma}
\newtheorem{proposition}[theorem]{Proposition}
\theoremstyle{definition}
\newtheorem{definition}[theorem]{Definition}
\theoremstyle{remark}
\newtheorem{remark}[theorem]{Remark}
\newtheorem{example}{Example}
\numberwithin{equation}{section}

\def\XXint#1#2#3{{\setbox0=\hbox{$#1{#2#3}{int}$}  
		\vcenter{\hbox{$#2#3$}}\kern-.5\wd0}}

\usepackage{graphicx}
\usepackage{color}
\usepackage{hyperref}

\begin{document}
	\title[Stochastic $\Sigma$-convergence in Orlicz setting and Applications]{Stochastic $\Sigma$-convergence in Orlicz setting and Applications}
	
	\author{Fotso Tachago Joel}
	\address{ Fotso Tachago Joel, University of Bamenda, 	Higher Teachers Training College, Department of Mathematics
		P.O. Box 39
		Bambili, Cameroon}
	\email{fotsotachago@yahoo.fr}
	\author{Nnang Hubert}
	\address{ Nnang Hubert, University of Yaounde I,
		Higher Teachers Trainning College, Department of Mathematics
		P.O. Box 47
		Yaounde, Cameroon}
	\email{hnnang@yahoo.fr}
	\author{Tchinda Takougoum Franck}
	\address{ Tchinda Takougoum Franck, University of Bertoua, Department of Mathematics, Statistics and Computer Science, P.O. box 652, Bertoua, Cameroon}
	\email{takougoumfranckarnold@gmail.com}
	\author{Woukeng Jean Louis}
	\address{ Woukeng Jean Louis, Department of Mathematics and Computer Science,
		University of Dschang, P.O. Box 67, Dschang, Cameroon}
	\email{jwoukeng@yahoo.fr}
	
	\date{\today}
	
	\begin{abstract}
		This paper aims to extend the concept of stochastic $\Sigma$-convergence to the framework of Orlicz-Sobolev spaces in order to deals with  coupled stochastic and deterministic homogenization problems in this type of spaces. Thus, this concept is a combination of both well-known  $\Sigma$-convergence [\textit{Acta Math. Sinica, English Series} \textbf{30}(9) 1621-1654] and stochastic two-scale convergence in the mean schemes [\textit{Asympt. Anal. (2025)} \textbf{142}, 291-320]. An application to the stochastic-deterministic homogenization (in the context of ergodic $H$-supralgebra) of a class of highly oscillatory minimizations problems involving integral functionals with convex
		and nonstandard growth integrands is also given, and some concrete homogenization problems following varied structure hypothesis are deduce from this latter.
	\end{abstract}
	
	\begin{center}
		\emph{Cordially dedicated to the full professor Joseph Dongho (1975--2026)}
	\end{center}
	
	\maketitle

	
	\section{Introduction} \label{labelintro} 
	
	This paper is devoted to the stochastic-deterministic homogenization theory in the framework of Orlicz-Sobolev spaces, and more precisely to the stochastic $\Sigma$-convergence method (see, e.g. \cite{sango2}) in this type of spaces. 
	
	For the applications, we consider the following minimization  problem 
	\begin{equation}\label{ras2}
		\min \left\{ F_{\varepsilon}(v) : v \in W^{1}_{0}L^{\Phi}_{D_{x}}(Q; L^{\Phi}(\Omega)) \right\},
	\end{equation} 
	where, for each $\varepsilon>0$ the functional $F_{\varepsilon}$ is defined by 
	\begin{equation}\label{ras1}
		F_{\varepsilon}(v) = \iint_{Q\times \Omega} f\left(x, T\left(\frac{x}{\varepsilon_{1}}\right)\omega, \frac{x}{\varepsilon_{2}}, \, Dv(x,\omega) \right)dxd\mu, 
	\end{equation} 
	with $f$ a random-deterministic integrand (see, Subsection \ref{hypoproblem} and comment below) and we are interested to the  homogenization (i.e., the analysis of  asymptotic behaviour when $0<\varepsilon \rightarrow 0$) of the sequence of solutions of the problem (\ref{ras2})-(\ref{ras1}).
	
	\subsection{Our hypotheses}\label{hypoproblem}
	
	Let us clarify data in (\ref{ras2})-(\ref{ras1}):
	\begin{itemize}
		\item $\Phi : [0,+\infty[ \rightarrow [0,+\infty[ $ be a $N$-function of class   $\Delta_{2}\cap \Delta'$ (see,  (\ref{delta2}) and (\ref{deltaprime})), such that its Fenchel's conjugate $\widetilde{\Phi}$ is also a $N$-function of class $\Delta_{2}\cap \Delta'$ ;
		\item $(\Omega, \mathscr{M}, \mu)$ be a measure space with probability measure $\mu$ ;
		\item $Q$ is a bounded open set in $\mathbb{R}^{d}_{x}$ (the space $\mathbb{R}^{d}$ of variables $x = (x_{1}, \cdots , x_{d})$) and $\{ T(x) : \Omega \rightarrow \Omega, \, x \in \mathbb{R}^{d} \}$ is a fixed $d$-dimensional dynamical system on $\Omega$ such that the probability measure $\mu$ on $\Omega$ is invariant for $T$ (see, Subsection \ref{labelsubsub3sect2}) ;
		\item $D$ denotes the usual gradient operator in $Q$, i.e., $D = D_{x}= \left(\frac{\partial}{\partial x_{i}} \right)_{1\leq i\leq d}$ and $W^{1}_{0}L^{\Phi}_{D_{x}}(Q; L^{\Phi}(\Omega))$ is a  Orlicz-Sobolev's space which will be specified later (see, (\ref{c1orli1S}) and comment below) ;
		\item $\varepsilon _{1}$ and $%
		\varepsilon _{2}$ be two well separated functions of $\varepsilon $ tending
		towards zero with $\varepsilon $, that is, $0<\varepsilon _{1},\varepsilon
		_{2},\varepsilon _{2}/\varepsilon _{1}\rightarrow 0$ as $\varepsilon
		\rightarrow 0$ ;
		\item $f : \mathbb{R}^{d}\times\Omega\times\mathbb{R}^{d}\times \mathbb{R}^{d}\rightarrow\mathbb{R}$, $(x,\omega,y,\lambda) \rightarrow f(x,\omega,y,\lambda)$ is a random-deterministic integrand satisfying the following properties:
		\begin{itemize}
			\item[\textbf{(H$_{1}$)}] \textit{(Carath\'{e}odory function hypothesis)} : for all $(x,\lambda) \in \mathbb{R}^{d}\times \mathbb{R}^{d}$ and for almost all $(\omega,y)\in\Omega\times\mathbb{R}^{d}$, $f(x,\cdot,\cdot, \lambda)$ is measurable and $f(\cdot,\omega,\cdot,\cdot)$ is continuous. Moreover, there exist a continuous positive function $\varpi : \mathbb{R} \to \mathbb{R}_{+}$ with $\varpi(0)=0$, and a function $a \in L^{1}_{loc}(\mathbb{R}^{d}_{y})$ such that 
			\begin{equation*}
				|f(x,\omega,y,\lambda) - f(x',\omega,y,\lambda)| \leq \varpi(|x - x'|)[a(y) + f(x,\omega,y,\lambda)]
			\end{equation*}
			for all $x, x' \in \mathbb{R}^{d}$, $\lambda \in \mathbb{R}^{d}$ and for $d\mu\times dy$-almost all $(\omega,y) \in \Omega\times\mathbb{R}^{d}$  
			\item[\textbf{(H$_{2}$)}] \textit{(Strictly convexity hypothesis) : }$f(x,\omega,y,\cdot)$ is strictly convex for $d\mu\times dy$-almost all $(\omega,y) \in \Omega\times\mathbb{R}^{d}$ and for all $x\in \mathbb{R}^{d}$ ;
			\item[\textbf{(H$_{3}$)}] \textit{(Nonstandard growth hypothesis)} : $f$ have a nonstandard growth, that is, there are two constants $c_{1}, \, c_{2} > 0$ such that 
			\begin{equation*}
				c_{1} \Phi(|\lambda|) \leq f(x,\omega,y,\lambda) \leq c_{2}(1+ \Phi(|\lambda|))
			\end{equation*}
			for all $(x,\lambda) \in \mathbb{R}^{d}\times \mathbb{R}^{d}$ and for $d\mu\times dy$-almost all $(\omega,y) \in \Omega\times\mathbb{R}^{d}$ ;
			\item[\textbf{(H$_{4}$)}] \textit{(Abstract algebra structure hypothesis)} : $f(x,\omega,\cdot,\lambda) \in A$ for all $(x,\lambda) \in \mathbb{R}^{d}\times \mathbb{R}^{d}$ and for $\mu$-almost all $\omega \in \Omega$, where $A$ is an ergodic $H$-supralgebra of class $\mathcal{C}^{\infty}$ on $\mathbb{R}_{y}^{d}$  (see, Definitions \ref{d2.1} and \ref{d2.2}).
		\end{itemize}
	\end{itemize}
	
	\subsection{Litterature review}
	
	The functionals $F_{\varepsilon}$ defined in (\ref{ras1}) have been studied by many authors when the integrand $f$ have a particular structure of periodicity.  So, in \cite{tacha1}, J. Fotso Tachago and H. Nnang after extending the periodic two-scale convergence \cite{allair1,nguet1} to Orlicz setting, will study the periodic homogenization of functionals $F_{\epsilon}$ under the form $\int_{Q} f(\frac{x}{\epsilon}, Dv(x))\,dx$, where $v$ belongs the Orlicz-Sobolev type space $W^{1}_{0}L^{\Phi}(Q)$. In \cite{tacha3}, the same authors in order to tackle periodic multiscale problems in Orlicz setting, will extend the periodic reiterated two-scale convergence \cite{allair3} and studied the periodic homogenization of functionals $F_{\epsilon}$ under the form $\int_{Q} f(\frac{x}{\epsilon}, \frac{x}{\epsilon^{2}}, Dv(x))\,dx$ with $v \in W^{1}_{0}L^{\Phi}(Q)$.
	For others works about periodic (reiterated) homogenization in the Orlicz setting, we refer to \cite{tacha2,tacha5,tacha3,tacha4,martin}.  
	On the other hand, in \cite{franck} J. Dongho \textit{et al} after extending  the two-scale convergence in the mean method (see \cite{Bourgeat}) to Orlicz-Sobolev's spaces, will study the stochastic homogenization of functionals $F_{\varepsilon}$ when they are of type $\iint_{Q\times \Omega} f\left(T\left(\frac{x}{\varepsilon}\right)\omega,\, Dv(x,\omega) \right)dxd\mu$, where $ v \in W^{1}_{0}L^{\Phi}_{D_{x}}(Q; L^{\Phi}(\Omega))$.
	
	Note that the works cited above only take into account either the periodic aspect or the stochastic aspect in the homogenization process; which is not always consistent with most physical phenomena since most of these phenomena behave randomly in some scales, and periodically in others scales.  
	Thus, in \cite{sango}, M. Sango and J.L. Woukeng proposed a general method called stochastic two-scale convergence used to solve coupled-periodic and stochastic-homogenization problems.
	More precisely, they consider the general case, when the functionals $F_{\epsilon}$ are of type   $\iint_{Q\times\Omega} f(x,T(\frac{x}{\varepsilon_{1}})\omega, \frac{x}{\varepsilon_{2}}, Dv(x,\omega))\,dx\,d\mu$ and study  it for the test functions with values in classical Lebesgue-Sobolev's spaces, i.e., $v \in L^{p}(\Omega, W_{0}^{1,p}(Q))$.  In \cite{joe,joelf}, J. Fotso Tachago and H. Nnang use this stochastic two-scale convergence for the homogenization of Maxwell equations with linear and nonlinear periodic conductivity. In the same order of idea, J. Dongho \textit{et al} (see, \cite{tchin2}) study via the stochastic two-scale convergence method extended to Orlicz-Sobolev's spaces, the stochastic-periodic homogenization of functionals $F_{\varepsilon}$ when they are of type $\iint_{Q\times\Omega} f(x,T(\frac{x}{\varepsilon})\omega, \frac{x}{\varepsilon^{2}}, Dv(x,\omega))\,dx\,d\mu$, where $ v \in W^{1}_{0}L^{\Phi}_{D_{x}}(Q; L^{\Phi}(\Omega))$.   For others works on stochastic homogenization and periodic homogenization in classical Sobolev's spaces we refer, e.g., to \cite{abda,andre,bia,blan,cham,dal,gambi,jikov}.
	
	However, in order to deal with deterministic homogenization theory beyond the
	periodic setting, Nguetseng \cite{nguetseng2003homogenization}, following Zhikov and Krivenko \cite%
	{zhikov1983}, introduced the concept of homogenization algebras. This theory
	relies heavily on ergodic theory (but not the ergodicity!) because in
	applications, the assumption of ergodicity of the homogenization algebra
	considered is fundamental.
	 We refer to \cite{nguetseng2004homogenization,nguetseng2025homogenization,nanguet1,wou,nguetsengWoukeng} for deterministic homogenization and \cite{luka,gabri} for deterministic reiterated homogenization in classical Sobolev spaces. We also refer to \cite{nnang2014deterministic} for deterministic homogenization and \cite{tchin3} for deterministic reiterated homogenization in Orlicz setting.
	 It is important to note that there was a
	gap between the periodic homogenization theory and the stochastic
	homogenization theory, gap which was filled by Nguetseng's deterministic
	homogenization theory.
	
	 The two theories mentioned above have the specificity to be used to solve
	either stochastic homogenization problems only (for the first one) or
	deterministic homogenization problems only (for the second one).
	Unfortunately, as we know, in nature, very few phenomena behave, either just
	randomly or deterministically; most of these phenomena behave randomly in
	some scales, and deterministically in other scales.
	
	Motivated by this vision of the physical nature, M. Sango and J.L. Woukeng \cite{sango2} in the framework of Lebesgue's spaces, rely on these two
	theories and hence on their associated convergence methods (the stochastic
	two scale convergence in the mean \cite{Bourgeat} and the $\Sigma $-convergence %
	\cite{nguetseng2003homogenization}) to propose a general method of solving coupled -
	deterministic and stochastic - homogenization problems. Their method, called 
	\textit{stochastic }$\Sigma $\textit{-convergence}, combines the macroscopic
	and microscopic [random and deterministic] scales, and has therefore the
	advantage of taking both the simplicity and the efficiency of the
	macroscopic models, as well as the accuracy of the coupled
	random-deterministic microscopic models.
	
	\subsection{Motivation and objectives}
	
	 Motivated by the works of H. Nnang \cite{nnang2014deterministic} on the deterministic homogenization in Orlicz setting and by recent works of J. Dongho \textit{et al.} \cite{franck} on the stochastic homogenization in Orlicz setting, we propose in the first time, to combines these two works in order to extends the stochastic $\Sigma$-convergence method to the framework of Orlicz-Sobolev spaces.
	
	After this extension,  our second contribution in this paper is the analysis of the asymptotic behaviour (when $0<\varepsilon \rightarrow 0$) via the stochastic $\Sigma$-convergence extended to Orlicz setting, of the sequence of solutions to the problem \eqref{ras2}-\eqref{ras1} under hypotheses of Subsection \ref{hypoproblem}, in particular the \textit{ abstract algebra structure hypothesis} \textbf{(H$_{4}$)} and then several examples considered in various concrete settings are presented by way of illustration.
	
	The extension to the Orlicz setting is also motivated by the fact that the stochastic $\Sigma$-convergence method developed in \cite{sango2} have been widely adopted in stochastic-deterministic homogenization of PDEs in classical Lebesgue-Sobolev spaces neglecting materials where microstructure cannot be conveniently captured by modeling exclusively by means of those spaces.  Moreover, it is well known (see, e.g. \cite{mignon2}) that there exist problems whose solution must naturally belong not to the classical Lebesgue-Sobolev spaces but rather to the Orlicz-Sobolev spaces. 
	
		Therefore, we generalize the results obtained in \cite{sango} of two sense: firstly we pass from the periodic framework to the general deterministic framework and secondly we pass from the Lebesgue spaces to the Orlicz spaces. In the same vein, we generalize the work of \cite{tchin2} and \cite{tchin3} in the sense where, for the first, the periodicity hypothesis is replaced by the abstract algebra structure hypothesis and for the second, it is a particular case as we will see in Remark \ref{remcase}.

	\subsection{Main results}
	
	After having defined the concept of (weak and strong) stochastic $\Sigma$-convergence in Orlicz setting (see, Definitions \ref{d3.1} and \ref{d3.2}), our first novelty is concerned the generalization of the compactness results of stochastic $\Sigma$-convergence in classical Sobolev spaces to a class of Orlicz-Sobolev spaces.	Thus, considering notations in Section \ref{labelsect2} and \ref{labelsect3}, we extend the result in \cite[Theorem 5]{sango2} and \cite[Theorem 8]{sango2} to Orliz-Sobolev spaces as follows:
	
	\begin{theorem}
	Let $\Phi$ be a $N$-function of class $\Delta_{2}\cap \nabla_{2}$ and let $A$ be an $H$-supralgebra on $\mathbb{R}^{d}_{y}$. Then, any bounded sequence $(u_{\varepsilon })_{\varepsilon \in E}$ in 
		$L^{\Phi}(Q\times \Omega )$ (where $E$ is a fundamental sequence) admits a subsequence which is weakly stochastically $\Sigma $%
		-convergent in $L^{\Phi}(Q\times \Omega )$.
	\end{theorem}
	
	\begin{theorem}
		Let $\Phi \in \Delta_{2}$ be a $N$-function and $\widetilde{\Phi} \in \Delta_{2}$ its conjugate. Let $A$ be an ergodic supralgebra on $\mathbb{\mathbb{R}}_{y}^{d}$ such that $A$ is translation invariant and  each of its elements is uniformly
		continuous.
		Assume $(u_{\varepsilon })_{\varepsilon \in E}$ is a sequence in $%
		L^{\Phi}(Q\times \Omega )$ such that:
		
		\begin{itemize}
			
			\item[(i)]  $(u_{\varepsilon })_{\varepsilon \in E}$ is bounded in $L^{\Phi}\left(Q\times\Omega\right)$ and $(D_{x}u_{\epsilon})_{\epsilon\in E}$ is bounded in $L^{\Phi}\left(Q\times\Omega\right)^{d}$. 
		\end{itemize}
		
		Then there exist $u_{0} \in W^{1}L^{\Phi}_{D_{x}}(Q; I_{nv}^{\Phi}(\Omega))$, $u_{1} \in L^{1}\left(Q; W^{1}_{\#}L^{\Phi}(\Omega)\right)$, $u_{2}\in L^{1}(Q\times\Omega; W^{1}_{\#}\mathcal{X}^{\Phi}_{A})$ and a subsequence $E'$ from $E$ such that, as $E' \ni \varepsilon \rightarrow 0$, 
		
		\begin{itemize}
			
			\item[(ii)]  $u_{\varepsilon }\rightarrow u_{0}$  in $L^{\Phi}(Q\times
			\Omega )$-weak $\digamma\, \Sigma$; 
			
			\item[(iii)]  $Du_{\varepsilon }\rightarrow Du_{0}+\overline{D}_{\omega }u_{1}+%
			\overline{D}_{y}u_{2}$  in $L^{\Phi}(Q\times \Omega )^{d}$-weak $\digamma\, \Sigma $,  with $u_{1} \in L^{\Phi}(Q; W^{1}_{\#}L^{\Phi}(\Omega))$, $u_{2}\in L^{\Phi}(Q\times\Omega; W^{1}_{\#}\mathcal{X}^{\Phi}_{A})$  when $\widetilde{\Phi} \in \Delta'$ [see, (\ref{deltaprime})].  
		\end{itemize}
	\end{theorem}
	
	Our second result (main homogenization result) in this paper are about the stochastic-deterministic homogenization problem \eqref{ras2}-\eqref{ras1}. 	Thus, considering notations in Section \ref{labelsect3} and \ref{labelsect4}, we extend the results in \cite[Theorem 1.3]{tchin2}  to the general deterministic setting as follows:
	
	\begin{theorem}\label{usetheo}
		Let $\Phi \in \Delta_{2} \cap \Delta'$ be a $N$-function and let $A$ be an ergodic supralgebra on $\mathbb{\mathbb{R}}_{y}^{d}$ such that $A$ is translation invariant and  each of its elements is uniformly
		continuous.
	For each $\varepsilon > 0$, let $(u_{\varepsilon})_{\varepsilon\in E} \in W^{1}_{0}L^{\Phi}_{D_{x}}(Q; L^{\Phi}(\Omega))$ be the unique solution of (\ref{ras2}). Then, as $\varepsilon \to 0$, 
\begin{equation*}
	u_{\varepsilon} \rightarrow u_{0} \quad  \;\textup{in} \; L^{\Phi}(Q\times\Omega)-weak \,\digamma\, \Sigma ,  
\end{equation*}
and 
\begin{equation*}
	Du_{\varepsilon} \rightarrow Du_{0} + \overline{D}_{\omega}u_{1} + \overline{D}_{y}u_{2} \quad  \;\textup{in} \; L^{\Phi}(Q\times\Omega)^{d}-weak \,\digamma\, \Sigma ,  
\end{equation*}
where $\mathbf{u} = (u_{0}, u_{1}, u_{2}) \in \mathbb{F}_{0}^{1}L^{\Phi}=  W^{1}_{0}L^{\Phi}_{D_{x}}(Q; I_{nv}^{\Phi}(\Omega))\times L^{\Phi}(Q; W^{1}_{\#}L^{\Phi}(\Omega))\times L^{\Phi}(Q\times\Omega; W^{1}_{\#}\mathcal{X}^{\Phi}_{A})$ 
is the unique solution to the minimization problem 
\begin{equation*}
\inf \left\{  F(\mathbf{v})\, : \, \mathbf{v} \in \mathbb{F}_{0}^{1}L^{\Phi} \right\},
\end{equation*}
where the functional $F$ on $\mathbb{F}_{0}^{1}L^{\Phi}$ is defined by 
\begin{equation*}
	F(\mathbf{v}) = \iint_{Q\times\Omega} \pounds_{A} \int  \varrho \circ f(\cdot,\cdot, \mathbb{D}\mathbf{v})\, dydxd\mu,
\end{equation*}
with $\varrho \circ f(\cdot,\cdot, \mathbb{D}\mathbf{v})(x,\omega,y) =  \varrho (f(x,\omega, \cdot, Dv_{0} + \overline{D}_{\omega}v_{1} + \overline{D}_{y}v_{2}(x,\omega,\cdot)))(y)$ for $(x,\omega,y) \in Q \times\Omega\times \mathbb{R}^{d}$ and $\varrho$ as in \eqref{rholabel}.
 
	\end{theorem}
	
		\begin{remark}
		It should be noted that in this study, we investigate the homogenization problem \eqref{ras2}-\eqref{ras1} not  under the periodicity hypothesis as in the \cite{tchin2}, but in a general deterministic setting including the periodicity, almost periodicity, weakly almost periodicity, convergence at infinity hypotheses and more others (see Section \ref{labelsect5} for more details). 
	\end{remark}
	
	\subsection{Organization of paper}
	
	The paper is divided into sections each revolving around a specific aspect: Section \ref{labelsect2} dwells on preliminaries on  Orlicz-Sobolev spaces, dynamical systems and homogenization supralgebra.  Section \ref{labelsect3} is devoted to the compactness results of the stochastic $\Sigma$-convergence in Orlicz setting. In Section \ref{labelsect4} we study the stochastic-deterministic homogenization  of problem (\ref{ras2})-(\ref{ras1}). The periodicity hypothesis stated in \cite{tchin2} is here replaced by an \textit{abstract algebra structure hypothesis}. Finally, Section \ref{labelsect5} is concerned with a few concrete examples of homogenization problem (\ref{ras2})-(\ref{ras1}) under various concrete structure.  

	
	\section{Prelimaries on  Orlicz-Sobolev spaces, dynamical systems and homogenization supralgebra} \label{labelsect2} 
	
	
	In this section, we recall some basic results about Orlicz spaces, dynamical systems and homogenization supralgebra.
	
	\subsection{Orlicz spaces and Dynamical systems} \label{labelsub1sect2} 
	
		All definitions and results recalled here are classical and can be found in \cite{adam,franck,zand}.
		
	\subsubsection{$N$-functions}\label{labelsubsub1sect2} 
	 
	Let $\Phi : [0,+\infty) \rightarrow  [0,+\infty)$ be a $N$-function, that is, 
	$\Phi$ is continuous, convex ,
	$\Phi(t) > 0$ for $t > 0$,
	$\frac{\Phi(t)}{t} \to 0$ as $t\to 0$ and $\frac{\Phi(t)}{t} \to \infty$ as $t\to \infty$. Then
	$\Phi$ has an integral representation under form
	$
	\Phi(t) = \int_{0}^{t} \phi(\tau)\, d\tau,
	$
	where $\phi : [0, \infty) \rightarrow  [0, \infty)$ is nondecreasing, right continuous, with $\phi(0) = 0$, $\phi(t) > 0$ if $t>0$ and $\phi(t)\rightarrow \infty$ if $t\rightarrow \infty$. We denote $\widetilde{\Phi}$
	the Fenchel's conjugate (or complementary function) of the $N$-function $\Phi$, that is, $\widetilde{\Phi}(t) = \sup_{s\geq 0} \left( st - \Phi(s) \right) \; (t\geq 0)$, then $\widetilde{\Phi}$ is also a $N$-function. Given two $N$-functions $\Phi$ and $\Psi$, we say that $\Phi$ dominates $\Psi$ (denoted by $\Phi \succ \Psi$ or $\Psi \prec \Phi$) near infinity if there are $k>1$ and $t_{0}>0$ such that 
	\begin{equation*}
		\Psi(t) \leq \Phi(kt), \quad \forall t> t_{0}.
	\end{equation*}
	With this in mind, it is well known that if $\displaystyle \lim_{t\to +\infty} \frac{\Phi(t)}{\Psi(t)} = +\infty$ then $\Phi$ dominates $\Psi$ near infinity. Let us recall some important properties about $N$-functions.  
	Let $\Phi$ be a $N$-function. $\Phi$ is said to satisfy the $\Delta_{2}$-condition or $\Phi$ belongs to the class $\Delta_{2}$ at $\infty$, which is written as  $\Phi \in \Delta_{2}$, if there exist constants  $t_{0} > 0$, $k > 2$ such that  
	\begin{equation}\label{delta2}
		\Phi(2t) \leq k \Phi(t),
	\end{equation}
	for all $t \geq t_{0}$.
	$\Phi$ is said to satisfy the $\nabla_{2}$-condition or $\Phi$ belongs to the class $\nabla_{2}$ at $\infty$, which is written as  $\Phi \in \nabla_{2}$, if there exist constants $t_{0} > 0$, $c > 2$ such that  
	\begin{equation*}
		\Phi(t) \leq \dfrac{1}{2c} \Phi(ct),
	\end{equation*}
	for all $t \geq t_{0}$. As mentioned in \cite{ttchin1}, an $N$-function $\Phi$ is of class $\nabla_{2}$ if and only if it conjugate $\widetilde{\Phi}$ is of class $\Delta_{2}$. Furthermore an $N$-function
	$\Phi$ is said to satisfy the $\Delta'$-condition (or $\Phi$ belongs to the class $\Delta'$) denoted by $\Phi \in \Delta'$ if there exists $\beta > 0$ such that  
	\begin{equation}\label{deltaprime}
		\Phi(ts) \leq \beta\, \Phi(t)\Phi(s), \quad \forall t,s \geq 0.
	\end{equation}
	Let $t\rightarrow \Phi(t) = \int_{0}^{t} \phi(\tau) d\tau$ be a $N$-function, and let $\widetilde{\Phi}$ be the Fenchel's conjugate of $\Phi$. Then one has
	\begin{equation}\label{lem10} 
		\left\{ \begin{array}{l}
			\dfrac{t\,\phi(t)}{\Phi(t)} \geq 1 \quad (\textup{resp.} \, > 1 \; \textup{if} \, \phi \, \textup{is} \, \textup{strictly} \; \textup{increasing})  \\
			
			\\
			\widetilde{\Phi}(\phi(t)) \leq t\,\phi(t) \leq \Phi(2t)
		\end{array}\right.
	\end{equation}
	for all $t >0$. If $\Phi$ is of class $\Delta_{2}$. Then there are $k> 0$ and $t_{0} \geq 0$ such that 
	\begin{equation*}\label{lem11}
		\widetilde{\Phi}(\phi(t)) \leq k \, \Phi(t) \quad \textup{for} \, \textup{all} \, t \geq t_{0}.
	\end{equation*}	
	We give now some examples of $N$-functions.
	\begin{example}
		The function $t \rightarrow \frac{t^{p}}{p}$, ($p > 1$) is a $N$-function which satisfy $\Delta_{2}$-condition and $\nabla_{2}$-condition. Its Fenchel's conjugate is a $N$-function $t \rightarrow \frac{t^{q}}{q}$, where $\frac{1}{p} + \frac{1}{q} = 1$. The function $t \rightarrow t^{p}\ln(1+t)$, ($p \geq 1$) is a $N$-function that satisfies $\Delta_{2}$-condition, while the $N$-function $t \rightarrow t^{\ln t}$ and $t \rightarrow e^{t^{r}} - 1$, ($r >0$) are not of class $\Delta_{2}$. However, for $r=1$, the $N$-function $t \rightarrow e^{t} - 1$ satisfy $\Delta'$-condition.   
	\end{example}
	
	\subsubsection{Orlicz spaces on bounded open set}\label{labelsubsub2sect2} 
	
	Let $Q$ be a bounded open set in $\mathbb{R}^{d}$ (integer $d\geq 1$), and let $\Phi$ be a $N$-function. The Orlicz space $L^{\Phi}(Q)$ is defined to be the space of all measurable functions $u: Q\rightarrow \mathbb{R}$ such that 
	\begin{equation*}
		\int_{Q}^{} \Phi\left(\dfrac{|u(x)|}{\delta} \right)dx < +\infty,
	\end{equation*}
	for some $\delta=\delta(u) >0$. 
	On the space $L^{\Phi}(Q)$ we define two equivalent norms:
	\begin{itemize}
		\item[(a)] the Luxemburg norm, 
		\begin{equation*}
			\lVert u \rVert_{L^{\Phi}(Q)} = \inf \left\{ \delta>0 \; : \; \int_{Q}^{} \Phi\left(\dfrac{|u(x)|}{\delta} \right)dx  \leq 1 \right\}, \quad \forall u \in L^{\Phi}(Q),
		\end{equation*}
		\item[(b)] the Orlicz norm, 
		\begin{equation*}
			\begin{array}{rcr}
				\lVert u \rVert_{L^{(\Phi)}(Q)}& =& \sup \left\{ \left| \displaystyle  \int_{Q} u(x)v(x) dx \right| \; : \; v \in L^{\widetilde{\Phi}}(Q) \; \textup{and} \; \|v\|_{L^{\widetilde{\Phi}}(Q)} \leq 1 \right\} \\
				& &  \\
				& &  \forall u \in L^{\Phi}(Q),
			\end{array}
		\end{equation*}
	\end{itemize}
	which makes it a Banach space.  \\
	Now we will give some properties of Orlicz spaces and we will refer to \cite{adam} for more details.
	Assume that $\Phi \in \Delta_{2}$. Then:
	\begin{itemize}
		\item[\textbf{(i)}] $\mathcal{D}(Q)$ is dense in $L^{\Phi}(Q)$ ;
		\item[\textbf{(ii)}] $L^{\Phi}(Q)$ is separable and reflexive whenever $\widetilde{\Phi}\in \Delta_{2}$ ;
		\item[\textbf{(iii)}] the dual of $L^{\Phi}(Q)$ is identified with $L^{\widetilde{\Phi}}(Q)$, and the dual norm on $L^{\widetilde{\Phi}}(Q)$ is equivalent to $\lVert \cdot \rVert_{L^{\widetilde{\Phi}}(Q)}$ ;
		\item[\textbf{(iv)}] given $u \in L^{\Phi}(Q)$ and $v \in L^{\widetilde{\Phi}}(Q)$ the product $uv$ belongs to $L^{1}(Q)$ with the generalized H\"{o}lder's inequality 
		\begin{equation*}
			\left| \int_{Q} u(x)v(x)dx \right| \leq 2\, \lVert u \rVert_{L^{\Phi}(Q)} \, \lVert v \rVert_{L^{\widetilde{\Phi}}(Q)} ;
		\end{equation*}
		\item[\textbf{(v)}] given $v \in L^{\Phi}(Q)$ the linear functional $L_{v}$ on $L^{\widetilde{\Phi}}(Q)$ defined by 
		\begin{equation*}
			L_{v}(u) = \int_{Q} u(x)v(x) dx, \quad u \in L^{\widetilde{\Phi}}(Q) ;
		\end{equation*} 
		belongs to the dual $[L^{\widetilde{\Phi}}(Q)]'$ with $\lVert v \rVert_{L^{(\Phi)}(Q)} \leq \lVert L_{v} \rVert_{[L^{\widetilde{\Phi}}(Q)]'} \leq 2 \lVert v \rVert_{L^{(\Phi)}(Q)}$ 
		\item[\textbf{(vi)}] $L^{\Phi}(Q) \hookrightarrow L^{1}(Q) \hookrightarrow	L^{1}_{\textup{loc}}(Q) \hookrightarrow \mathcal{D'}(Q)$,
		\item[\textbf{(vii)}] given two $N$-functions $\Phi$ and $\Psi$, we have the continuous embedding \\ $L^{\Phi}(Q) \hookrightarrow L^{\Psi}(Q)$ if and only if $\Phi \succ \Psi$ near infinity ;
		\item[\textbf{(viii)}] the product space $L^{\Phi}(Q)^{d}= L^{\Phi}(Q)\times L^{\Phi}(Q) \times \cdots \times L^{\Phi}(Q)$, ($d$-times), is endowed with the norm 
		\begin{equation*}
			\lVert \textbf{v} \rVert_{L^{(\Phi)}(Q)^{d}} = \sum_{i=1}^{d} \lVert v_{i} \rVert_{L^{(\Phi)}(Q)}, \quad \textbf{v}=(v_{i}) \in L^{\Phi}(Q)^{d} ;
		\end{equation*}
		\item[\textbf{(ix)}] \label{property9}  If $Q_{1} \subset \mathbb{R}^{d_{1}}$ and $Q_{2} \subset \mathbb{R}^{d_{2}}$ are two bounded open sets with $d_{1}+d_{2}=d$, and if $u \in L^{\Phi}(Q_{1}\times Q_{2})$,
		then for almost all $x_{1}\in Q_{1}$, $u(x_{1}, \cdot) \in L^{\Phi}(Q_{2})$. If in addition $\widetilde{\Phi} \in \Delta'$ associate with a constant $\beta$, then the function $u$ belongs to $L^{\Phi}(Q_{1}, L^{\Phi}(Q_{2}))$, with	
		\begin{equation}\label{be1}
			\|u\|_{L^{\Phi}(Q_{1}, L^{\Phi}(Q_{2}))} \leq \iint_{Q_{1}\times Q_{2}} \Phi(|u(x_{1},x_{2})|) dx_{1}dx_{2} + \beta.
		\end{equation}
	\end{itemize}
	
	\subsubsection{Dynamical systems}\label{labelsubsub3sect2} 
	
	Let $(\Omega, \mathscr{M}, \mu)$ be a measure space with probability measure $\mu$. We define an $d$-dimensional dynamical system on $\Omega$ as a family $\{ T(x) : x \in \mathbb{R}^{d} \}$	of invertible maps,
	\begin{equation*}
		T(x) : \Omega \longrightarrow \Omega
	\end{equation*}
	such that for each $x \in \mathbb{R}^{d}$ both $T(x)$ and $T(x)^{-1}$ are measurable, and such that the following properties hold: 
	\begin{itemize}
		\item[(1)]\textit{(group property)} $T(0)=$ identity map on $\Omega$ and for all $x_{1},x_{2} \in \mathbb{R}^{d}$,
		\begin{equation*}
			T(x_{1}+x_{2})=T(x_{1})\circ T(x_{2}),
		\end{equation*}
		\item[(2)]\textit{(invariance)} for each $x \in \mathbb{R}^{d}$, the map $T(x) : \Omega \rightarrow \Omega$ is measurable on $\Omega$ and $\mu$-measure preserving, i.e. $\mu(T(x)F)=\mu(F)$, for all $F \in \mathscr{M}$,
		\item[(3)]\textit{(measurability)} for each $F \in \mathscr{M}$, the set $\{ (x,\omega) \in \mathbb{R}^{d}\times \Omega : T(x)\omega \in F \}$ is measurable with respect to the product $\sigma$-algebra $\mathscr{L}\otimes \mathscr{M}$, where $\mathscr{L}$ is the $\sigma$-algebra of Lebesgue measurable sets. 
	\end{itemize}
	If $\Omega$ is a compact topological space, by a continuous $d$-dynamical system on $\Omega$, we mean any family of mappings $\{ T(x) : \Omega \rightarrow \Omega,\; x \in \mathbb{R}^{d} \}$  satisfying the above group property $(i)$ and the following condition :
	\begin{itemize}
		\item[(4)] \textit{(continuity)} the mapping $(x, \omega) \mapsto T(x)\omega$ is continuous from $\mathbb{R}^{d}\times \Omega$ to $\Omega$.
	\end{itemize} 
	
	Let $\Phi \in \Delta_{2}$ be a $N$-function. As in \cite[Proposition 3]{franck}, a $d$-dimensional dynamical system $\{ T(x) : \Omega\rightarrow \Omega\; ; \; x \in \mathbb{R}^{d} \}$ induces an $d$-parameter group of isometries $\{ U(x) : L^{\Phi}(\Omega)\rightarrow L^{\Phi}(\Omega)\; ; \; x \in \mathbb{R}^{d} \}$ defined by
	\begin{equation}\label{group1S}
		\left(U(x)f  \right)(\omega) = f(T(x)\omega), \quad \forall\, f \in L^{\Phi}(\Omega)
	\end{equation}
	which is strongly continuous, i.e.
	\begin{equation}\label{lem38}
		\lim_{x \to 0} \|U(x)f - f\|_{L^{\Phi}(\Omega)}=0, \quad \forall\, f \in L^{\Phi}(\Omega).
	\end{equation}
	
	Now, we consider  based on (\ref{group1S})  the $d$ 1-parameter group of isometries $\{ U(t\vec{e_{i}}) : L^{\Phi}(\Omega)\rightarrow L^{\Phi}(\Omega)\; ; \; t \in \mathbb{R} \}$ where, for $i=1, \cdots, d$, $\vec{e_{i}}= (\delta_{ij})_{1\leq j\leq d}$, $\delta_{ij}$ being the Kronecker symbol. We denote by $D_{i,\Phi}$ the generator (see, e.g. \cite[Page 376]{rudin}) of $U(t\vec{e_{i}})$ and by $\textbf{D}_{i,\Phi}$ its domain. Thus, for $f \in L^{\Phi}(\Omega)$, $f$ is in $\mathbf{D}_{i,\Phi}$ if and only if the limit $D_{i,\Phi}f$ defined by 
	\begin{equation*}
		D_{i,\Phi}f = \lim_{t \to 0} \dfrac{U(t\vec{e_{i}})f - f}{t}
	\end{equation*}  
	exists strongly in $L^{\Phi}(\Omega)$, i.e., $\displaystyle \lim_{t \to 0} \frac{1}{t}\left\|U(t\vec{e_{i}})f - f \right\|_{L^{\Phi}(\Omega)} = 0$. 
	In this case, $D_{i,\Phi}f$ is called $i$-th stochastic derivative of $f$.
	\begin{remark}
		If $\Omega = Y \subset \mathbb{R}^{d}$, $d\mu = dy$ (the Lebesgue measure on $\mathbb{R}^{d}$) and that we consider the $d$-dimensional dynamical system of translation $\{ T(x) : Y\rightarrow Y;  \; x \in \mathbb{R}^{d} \}$ defined by 
		\begin{equation*}
			T(x)y = (x+y)\textup{mod\,$\mathbb{Z}^{d}$},
		\end{equation*}
		then the $i$-th stochastic derivative $D_{i,\Phi}f$ of an element $f\in L^{\Phi}(Y)$ is also the $i$-th partial derivative of $f$.
	\end{remark}
	
	For a multi-index $\alpha= (\alpha_{1}, \cdots, \alpha_{d}) \in \mathbb{N}^{d}$, one can naturally define higher order stochastic derivatives by setting:
	\begin{equation*}
		D_{\Phi}^{\alpha} = D_{1,\Phi}^{\alpha_{1}}\circ \cdots\circ D_{d,\Phi}^{\alpha_{d}} \;\;\; \textup{and} \;\;\;	D_{i,\Phi}^{\alpha_{i}} = 
		\underset{\alpha_{i}-times}{\underbrace{D_{i,\Phi}\circ \cdots\circ D_{i,\Phi}}}, \quad i=1, \cdots, d.
	\end{equation*}
	Now we need to define the stochastic analog of the smooth functions on $\mathbb{R}^{d}$.\\
	We set $\displaystyle\textbf{D}_{\Phi}(\Omega) = \bigcap_{i=1}^{d} \textbf{D}_{i,\Phi}(\Omega)$, and define 
	\begin{equation*}
		\textbf{D}_{\Phi}^{\infty}(\Omega) = \left\{ f \in L^{\Phi}(\Omega) \; : \; D^{\alpha}_{\Phi}f \in \textbf{D}_{\Phi}(\Omega), \; \forall \alpha \in \mathbb{N}^{d}  \right\}.
	\end{equation*}
	We recall that (see \cite[Page 99]{sango}), for the Lebesgue spaces $L^{p}(\Omega)$, with $1\leq p \leq \infty$, we have 
	$\displaystyle\textbf{D}_{p}(\Omega) = \bigcap_{i=1}^{d} \textbf{D}_{i,p}$ and 
	\begin{equation*}
		\textbf{D}_{p}^{\infty}(\Omega) = \left\{ f \in L^{p}(\Omega) \; : \; D^{\alpha}_{p}f \in \textbf{D}_{p}(\Omega), \; \forall \alpha \in \mathbb{N}^{d}  \right\},
	\end{equation*}
	where $D^{\alpha}_{p}$ stands for stochastic higher order derivative in $L^{p}(\Omega)$. 
	Furthermore, for $p=\infty$ it is a fact that each element of $\textbf{D}_{\infty}^{\infty}(\Omega)$ possesses stochastic derivatives of any order that are bounded. It is noted by the suggestive symbol $\mathcal{C}^{\infty}(\Omega)$ and is endowed with its natural topology defined by the family of seminorms 
	\begin{equation*}
		N_{n}(f) = \sup_{|\alpha|\leq n} \sup_{\omega\in\Omega} |D_{\infty}^{\alpha}f(\omega)|, \;\; \textup{where} \;\, f\in \mathcal{C}^{\infty}(\Omega) \;\, \textup{and} \;\, |\alpha|= \alpha_{1}+ \cdots + \alpha_{d}. 
	\end{equation*}
	Note that the space $\mathcal{C}^{\infty}(\Omega)$ is a generalization to probability sets of the usual space $\mathcal{C}^{\infty}(\mathbb{R}^{d})$ of smooth functions.  
	Let now $K \in \mathcal{C}^{\infty}_{0}(\mathbb{R}^{d})$ be a nonnegative even function such that $\int_{\mathbb{R}^{d}} K(x)dx = 1$. \\ We set $K_{\delta}(x)= \delta^{-d}K(x/\delta)$, with $\delta >0$ and then define the operator $J_{\delta}$ of $L^{\Phi}(\Omega)$ into $L^{\Phi}(\Omega)$ by 
	\begin{equation*}
		J_{\delta}f = \int_{\mathbb{R}^{d}} K_{\delta}(y) U(y)f \, dy. 
	\end{equation*}  
	By the \cite[Lemma 4]{franck}, for all $f \in L^{\Phi}(\Omega)$, the function $J_{\delta}f \in \mathcal{C}^{\infty}(\Omega)$ and we have  
	\begin{equation}\label{rem6}
		\lim_{\delta \to 0} \|J_{\delta}f - f\|_{L^{\Phi}(\Omega)} = 0.
	\end{equation}
	
	Indeed, the space $ \mathcal{C}^{\infty}(\Omega)$ is dense in $L^{\Phi}(\Omega)$, when the $N$-function $\Phi$ is of class $\Delta_{2}$.
	Now, by a stochastic distribution on $\Omega$ is meant any continuous linear mapping from $\mathcal{C}^{\infty}(\Omega)$ to the real field $\mathbb{R}$. We denote the space of stochastic distributions by $(\mathcal{C}^{\infty}(\Omega))'$.
	\\
	For any $\alpha \in \mathbb{N}^{d}$, we define the stochastic weak derivative of $f \in (\mathcal{C}^{\infty}(\Omega))'$ as follows:
	\begin{equation*}
		(D^{\alpha}f)(\phi) = (-1)^{|\alpha|}f(D^{\alpha}\phi), \quad \forall \phi \in \mathcal{C}^{\infty}(\Omega).
	\end{equation*}
	Let $\Phi \in \Delta_{2}$ be a $N$-function and $\widetilde{\Phi}$ its conjugate.
	As $\mathcal{C}^{\infty}(\Omega)$ is dense in $L^{\widetilde{\Phi}}(\Omega)$, it is immediate that $(L^{\widetilde{\Phi}}(\Omega))' = L^{\Phi}(\Omega) \subset (\mathcal{C}^{\infty}(\Omega))'$, so that one may define the stochastic weak derivative of any $f \in L^{\Phi}(\Omega)$, and it verifies the functional equation: 
	\begin{equation}\label{deri1}
		(D^{\alpha}f)(\phi) = (-1)^{|\alpha|} \int_{\Omega} f\, D^{\alpha}\phi\,d\mu, \quad \forall \phi \in \mathcal{C}^{\infty}(\Omega).
	\end{equation}	
	In particular for $f \in \textbf{D}_{i,\Phi}$, we have
	\begin{equation}\label{deri2}
		\int_{\Omega}\phi D_{i,\Phi}f \,d\mu = - \int_{\Omega} f\, D_{i,\infty}\phi\,d\mu, \quad \forall \phi \in \mathcal{C}^{\infty}(\Omega).
	\end{equation}
	So that we may identify $ D_{i,\Phi}$ with $D^{\alpha_{i}}$, where $\alpha_{i} = (\delta_{ij})_{1\leq j\leq d}$.  
	Note that, formulas (\ref{deri1}) and (\ref{deri2}) are analogous to Sango-Woukeng paper \cite{sango} and the same notation as them has been adopted in the entire subsection. 
	\\
	Conversely, if $f \in L^{\Phi}(\Omega)$ is such that there exists $f_{i} \in L^{\Phi}(\Omega)$ with \\ $(D^{\alpha_{i}}f)(\phi) = - \int_{\Omega} f_{i} \phi\ d\mu$ for all $\phi \in \mathcal{C}^{\infty}(\Omega)$, then $f \in \textbf{D}_{i,\Phi}$ and $D_{i,\Phi}f = f_{i}$. 
	\\
	Endowing $\displaystyle\textbf{D}_{\Phi}(\Omega) = \bigcap_{i=1}^{d} \textbf{D}_{i,\Phi}(\Omega)$ with the natural graph norm 
	\begin{equation*}
		\|f\|_{\textbf{D}_{\Phi}(\Omega)} = \|f\|_{L^{\Phi}(\Omega)} + \sum_{i=1}^{d} \|D_{i,\Phi}f\|_{L^{\Phi}(\Omega)}, \quad \forall f \in \textbf{D}_{\Phi}(\Omega),
	\end{equation*}
	we obtain a Banach space representing the stochastic generalization of the classical Orlicz-Sobolev space $W^{1}L^{\Phi}(\mathbb{R}^{d})$, and so, we denote it by $W^{1}L^{\Phi}(\Omega)$.  \\
	Now we recall some definitions about the invariance of functions and ergodic dynamical system.  Note that these definitions are a generalization to Orlicz-Sobolev spaces $W^{1}L^{\Phi}$ of the definitions obtained in \cite[Page 99]{sango} for the case of Lebesgue-Sobolev spaces $W^{1,p}$. \label{pageS} \\  	A function $f \in L^{\Phi}(\Omega)$ is said to be invariant for the dynamical system $T$ (relative to $\mu$) if for any $x \in \mathbb{R}^{d}$
	\begin{equation*}
		f\circ T(x) = f, \quad \mu-a.e. \;\, \textup{on} \;\, \Omega.
	\end{equation*}
	We denote  $I_{nv}^{\Phi}(\Omega)$ the set of functions in $L^{\Phi}(\Omega)$ that are invariant for $T$. 
	The set $I_{nv}^{\Phi}(\Omega)$ is a closed vector subspace of $L^{\Phi}(\Omega)$   
	and when $\Phi(t) = \frac{t^{p}}{p}$ ($p\geq 1$), then $I_{nv}^{\Phi}(\Omega)$ becomes the space $I_{nv}^{p}(\Omega)$ obtained by replacing the Orlicz space $ L^{\Phi}(\Omega)$ in the definition with the Lebesgue space  $L^{p}(\Omega)$. 
	The dynamical system $T$ is said to be ergodic if every invariant function is $\mu$-equivalent to a constant, i.e. for all $f \in I_{nv}^{\Phi}(\Omega)$, we have $f(\omega) = \textup{c}$, $c\in \mathbb{R}$, for $\mu$-a.e. $\omega\in \Omega$.  
	Let $f$ be a measurable function in $\Omega$, for a fixed $\omega\in \Omega$ the function $x\rightarrow f(T(x)\omega)$, $x\in \mathbb{R}^{d}$ is called a realization of $f$ and the mapping $(x,\omega) \rightarrow f(T(x)\omega)$ is called a stationary process. The process is said to be stationary ergodic if the dynamical system $T$ is ergodic. \\
	Thus,	for $f \in \textbf{D}_{1}^{\infty}(\Omega)$, and for $\mu$-a.e. $\omega \in \Omega$, the function $x \rightarrow f(T(x)\omega)$ is in $\mathcal{C}^{\infty}(\mathbb{R}^{d})$ and further 
	\begin{equation}\label{lem18}
		D^{\alpha}_{x} f(T(x)\omega) = (D^{\alpha}_{1}f)(T(x)\omega), \quad \textup{for\; any} \;\, \alpha \in \mathbb{N}^{d},
	\end{equation}
	where $D^{\alpha}_{x}$ is the usual higher order derivative with respect to $x$ and $D^{\alpha}_{1}$ is the stochastic higher order derivative in $L^{1}(\Omega)$.
	If the dynamical system $T$ is ergodic, then:
	\begin{itemize}
		\item[(i)] For $f\in L^{1}(\Omega)$, we have $f \in I_{nv}^{1}(\Omega)$ if and only if $D_{i,1}f = 0$ for each $1 \leq i \leq d$.
		\item[(ii)] For $f\in L^{\Phi}(\Omega)$, we have $f \in I_{nv}^{\Phi}(\Omega)$ if and only if $D_{i,\Phi}f = 0$ for each $1 \leq i \leq d$.
	\end{itemize}
	So if we endow $\mathcal{C}^{\infty}(\Omega)$ with the seminorm 
	\begin{equation*}
		\|u\|_{\# ,\Phi} = \sum_{i=1}^{d} \|D_{i,\Phi}u\|_{L^{\Phi}(\Omega)}, \quad u \in \mathcal{C}^{\infty}(\Omega),
	\end{equation*}
	we obtain a locally convex space which is generally non separated and non complete. We denote $W^{1}_{\#}L^{\Phi}(\Omega)$ the separated completion of $\mathcal{C}^{\infty}(\Omega)$ with respect to the seminorm $\|\cdot\|_{\# ,\Phi}$, and we denote by 
	\begin{equation*}
		I_{\Phi} : \mathcal{C}^{\infty}(\Omega) \hookrightarrow W^{1}_{\#}L^{\Phi}(\Omega),
	\end{equation*}
	the canonical mapping. It is to be noted that $W^{1}_{\#}L^{\Phi}(\Omega)$ is also the separated completion of $\mathcal{C}^{\infty}(\Omega)/ (I^{\Phi}_{nv}(\Omega) \cap \mathcal{C}^{\infty}(\Omega))$ with respect to the same seminorm since for $u \in \mathcal{C}^{\infty}(\Omega)$ we have $\|u\|_{\# ,\Phi} = 0$ if and only if $u \in v$, that is $u \in I^{\Phi}_{nv}(\Omega) \cap \mathcal{C}^{\infty}(\Omega)$.   
	The following property is obtained through the theory of completion of uniform spaces; see, e.g. \cite[Chap. II]{bourbaki2007topologie}. \\
	The gradient operator $D_{\omega,\Phi} = (D_{1,\Phi}, \cdots, D_{d,\Phi}) : \mathcal{C}^{\infty}(\Omega) \rightarrow L^{\Phi}(\Omega)^{d}$ extends by continuity to a unique mapping $\overline{D}_{\omega,\Phi} = (\overline{D}_{1,\Phi}, \cdots, \overline{D}_{d,\Phi}) : W^{1}_{\#}L^{\Phi}(\Omega) \rightarrow L^{\Phi}(\Omega)^{d}$ with the properties:
	\begin{itemize}
		\item[(i)] $D_{i,\Phi} = \overline{D}_{i,\Phi}\circ I_{\Phi}$, for $1\leq i\leq d$,
		\item[(ii)] $\displaystyle \|u\|_{W^{1}_{\#}L^{\Phi}(\Omega)}= \|u\|_{\# ,\Phi} = \sum_{i=1}^{d} \|\overline{D}_{i,\Phi}u\|_{L^{\Phi}(\Omega)}, \quad u \in W^{1}_{\#}L^{\Phi}(\Omega)$.
	\end{itemize}
	Moreover, the mapping $\overline{D}_{\omega,\Phi}$ is an isometric embedding of $W^{1}_{\#}L^{\Phi}(\Omega)$ into a closed subspace of $L^{\Phi}(\Omega)^{d}$ and as $L^{\Phi}(\Omega)^{d}$ is a reflexive Banach space, then we deduce that the Banach space $W^{1}_{\#}L^{\Phi}(\Omega)$ is reflexive.  \\
	By duality we define the operator $\textup{div}_{\omega,\widetilde{\Phi}} : L^{\widetilde{\Phi}}(\Omega)^{d} \rightarrow (W^{1}_{\#}L^{\Phi}(\Omega))'$ by:
	\begin{equation*}
		\langle \textup{div}_{\omega,\widetilde{\Phi}} u \, , \, v \rangle = - \langle u\, , \, \overline{D}_{\omega,\Phi}v \rangle, \quad \forall u \in L^{\widetilde{\Phi}}(\Omega)^{d} \; \textup{and} \; \forall v \in W^{1}_{\#}L^{\Phi}(\Omega).
	\end{equation*}  
	The operator $\textup{div}_{\omega,\widetilde{\Phi}}$ just defined extends the natural divergence operator defined in $\mathcal{C}^{\infty}(\Omega)$ since for all $f\in \mathcal{C}^{\infty}(\Omega)$ we have $D_{i,\Phi}f = \overline{D}_{i,\Phi}\circ I_{\Phi}(f)$. 
	We will denote $D_{\omega,\Phi}=D_{\omega}$, $\overline{D}_{\omega,\Phi}=\overline{D}_{\omega}$ and $\textup{div}_{\omega,\widetilde{\Phi}}= \textup{div}_{\omega,p'}$ ($p'= \frac{p-1}{p}$) if there is no ambiguity, and we refer to \cite{sango} for the definitions of these operators in $L^{p}$-spaces that the authors note $D_{\omega,p}$, $\overline{D}_{\omega,p}$ and $\textup{div}_{\omega,p'}$ respectively.  \\
	Now, we define the Orlicz-Sobolev space $W^{1}L^{\Phi}_{D_{x}}(Q; L^{\Phi}(\Omega))$ as follows
	\begin{equation}\label{c1orli1S}
		W^{1}L^{\Phi}_{D_{x}}(Q; L^{\Phi}(\Omega)) = \left\{  v \in L^{\Phi}(Q\times\Omega) \, : \, \frac{\partial v}{\partial x_{i}} \in L^{\Phi}(Q\times\Omega), \; 1 \leq i \leq d \right\}
	\end{equation}
	where derivatives are taken in the distributional sense on $Q$. Endowed with the norm 
	\begin{equation*}
		\lVert v \rVert_{W^{1}L^{\Phi}_{D_{x}}(Q; L^{\Phi}(\Omega))} = \lVert v \rVert_{L^{\Phi}(Q\times\Omega)} + \sum_{i=1}^{d} \lVert D_{x_{i}} v \rVert_{L^{\Phi}(Q\times\Omega)}\, ; \; v \in W^{1}L^{\Phi}_{D_{x}}(Q; L^{\Phi}(\Omega)),
	\end{equation*}
	$W^{1}L^{\Phi}_{D_{x}}(Q; L^{\Phi}(\Omega))$ is a reflexive Banach space. On the other hand, we denote by $W^{1}_{0}L^{\Phi}_{D_{x}}(Q; L^{\Phi}(\Omega))$ the set of functions in $W^{1}L^{\Phi}_{D_{x}}(Q; L^{\Phi}(\Omega))$ with zero boundary condition on $Q$. Endowed with the norm 
	\begin{equation*}
		\lVert v \rVert_{W^{1}_{0}L^{\Phi}_{D_{x}}(Q; L^{\Phi}(\Omega))} =  \sum_{i=1}^{d} \lVert D_{x_{i}} v \rVert_{L^{\Phi}(Q\times\Omega)} \;\, ; \quad v \in W^{1}_{0}L^{\Phi}_{D_{x}}(Q; L^{\Phi}(\Omega)),
	\end{equation*}
	$W^{1}_{0}L^{\Phi}_{D_{x}}(Q; L^{\Phi}(\Omega))$ is a reflexive Banach space. \\
	We end this section by the following lemma will be of great interest for the proof of compactness result of a derivatives sequence of functions.
	\begin{lemma}\cite{franck}\label{lem5}
		Let $\Phi \in \Delta_{2}$ be a  $N$-function and $\widetilde{\Phi}$ its conjugate.	Let $v \in L^{\Phi}(\Omega)^{d}$ satisfy 
		\begin{equation*}
			\int_{\Omega} v\cdot g d\mu = 0, \quad \textup{for \; all} \;\, g \in  \mathcal{V}^{\widetilde{\Phi}}_{\textup{div} }= \{ f \in \mathcal{C}^{\infty}(\Omega)^{d} : \textup{div}_{\omega,\widetilde{\Phi}}f=0 \}.
		\end{equation*}
		Then, there exists $u \in W^{1}_{\#}L^{\Phi}(\Omega)$ such that 
		\begin{equation*}
			v = \overline{D}_{\omega,\Phi}u.
		\end{equation*} 
	\end{lemma}
	
	\subsection{Orlicz spaces associated to supralgebra} \label{labelsub2sect2} 
	
	\subsubsection{Homogenization supralgebra}\label{labelsub2sub1sect2} 
	
	 This concept has just been
	defined in a more recent paper \cite{gabri}. It is more general than those
	defined in the papers \cite{nguetseng2003homogenization,zhikov1983averaging} because we do not need the algebra
	to be separable (as in \cite{nguetseng2003homogenization}), or to consist of functions that are
	uniformly continuous (as in \cite{zhikov1983averaging}). Before we go any further, we
	need to give some preliminaries. Let $\mathcal{H}=(H_{\varepsilon
	})_{\varepsilon >0}$ be the action of $\mathbb{R}_{+}^{\ast }$ (the
	multiplicative group of positive real numbers) on the numerical space $%
	\mathbb{R}^{d}$ defined as follows: 
	\begin{equation}
		H_{\varepsilon }(x)=\frac{x}{\varepsilon _{1}}\;\;(x\in \mathbb{R}^{d})
		\label{2.1}
	\end{equation}%
	where $\varepsilon _{1}$ is a positive function of $\varepsilon $ tending to
	zero with $\varepsilon $. For given $\varepsilon >0$, let 
	\begin{equation*}
		u^{\varepsilon }(x)=u(H_{\varepsilon }(x))\;\;(x\in \mathbb{R}%
		^{d}).\;\;\;\;\;\;
	\end{equation*}%
	
	A function $u\in \mathcal{B}(\mathbb{R}_{y}^{d})$ (the $\mathcal{C}^{\ast}$%
	-algebra of bounded continuous complex functions on $\mathbb{R}_{y}^{d}$)
	is said to have a mean value for $\mathcal{H}$, if there exists a complex
	number $M(u)$ such that $u^{\varepsilon }\rightarrow M(u)$ in $L^{\infty }(%
	\mathbb{R}_{x}^{d})$-weak $\ast $ as $\varepsilon \rightarrow 0$. The
	complex number $M(u)$ is called the mean value of $u$ (for $\mathcal{H}$).
	It is evident that this defines a mapping $M$ which is a positive linear
	form (on the space of functions $u\in \mathcal{B}(\mathbb{R}_{y}^{d})$ with
	mean value) attaining the value $1$ on the constant function $1$ and
	verifying the inequality $\left\vert M(u)\right\vert \leq \left\Vert
	u\right\Vert _{\infty }\equiv \sup_{y\in \mathbb{R}^{d}}\left\vert
	u(y)\right\vert $\ for all such $u$. The mapping $M$ is called the \textit{%
		mean value on} $\mathbb{R}^{N}$ \textit{for} $\mathcal{H}$. It is also a
	fact, as the characteristic function of all relatively compact set in $%
	\mathbb{R}^{d}$ lies in $L^{1}(\mathbb{R}^{d})$, that 
	\begin{equation}
		M(u)=\lim_{r\rightarrow +\infty }\frac{1}{\left\vert B_{r}\right\vert }%
		\int_{B_{r}}u(y)dy\;\;\;\;\;\;\;\;\;\;\;\;\;\;  \label{2.3}
	\end{equation}%
	where $B_{r}$ stands for the bounded open ball in $\mathbb{R}^{d}$ with
	radius $r$, and $\left\vert B_{r}\right\vert $ denotes its Lebesgue measure.
	Expression (\ref{2.3}) also works for $u\in L_{\text{loc}}^{1}(\mathbb{R}%
	^{d})$ provided that the above limit makes sense. In connection with the
	dynamical systems, we have the following Birkhoff ergodic theorem for the functions in Orlicz spaces (see \cite%
	{finet2006transfer}).
	
	\begin{theorem}[\textit{Birkhoff Ergodic Theorem}\cite{finet2006transfer}]
		\label{t2.0}Let $T$ be a dynamical system acting on the probability space $%
		(\Omega ,\mathcal{M},\mu )$. Let $f\in L^{\Phi}(\Omega )$. Then for
		almost all $\omega \in \Omega $ the realization $x\mapsto f(T(x)\omega )$
		possesses a mean value in the sense of \emph{(\ref{2.3})}. Furthermore, the
		mean value $M(f(T(\cdot )\omega ))$ is invariant and 
		\begin{equation*}
			\int_{\Omega }f(\omega )d\mu =\int_{\Omega }M(f(T(\cdot )\omega ))d\mu .
		\end{equation*}%
		Moreover if the dynamical system $T$ is ergodic, then 
		\begin{equation*}
			M(f(T(\cdot )\omega ))=\int_{\Omega }fd\mu \text{ \ for }\mu \text{-a.e. }%
			\omega \in \Omega .
		\end{equation*}
	\end{theorem}
	
	We summarize below a few basic notions and results concerning the homogenization supralgebras. We refer to \cite{gabri,sango2} for further details.
	
	\begin{definition}
		\label{d2.1} By a homogenization supralgebra  (or $H$%
			-supralgebra, in short) on $\mathbb{R}^{d}$ for $\mathcal{H}$ 
			 we mean any closed subalgebra of $\mathcal{B}(\mathbb{R}^{d})$
			which contains the constants, is closed under complex conjugation and whose
			elements possess a mean value for $\mathcal{H}$.
	\end{definition}
	
	\begin{remark}
		\label{r2.1} From the above definition we see that the concept of $H$%
		-supralgebra is more general than those of $H$-algebra \cite{reed1980methods}
			and of algebra with mean value \cite{zhikov1983averaging}. In fact any separable $H$%
		-supralgebra is an $H$-algebra while any algebra with mean
			value is an $H$-supralgebra as any uniformly continuous function is
			continuous.
	\end{remark}
	
	Let $A$ be an $H$-supralgebra on $\mathbb{R}^{d}$ (for $\mathcal{H}$). It is
	known that $A$ (endowed with the sup norm topology) is a commutative $%
	\mathcal{C}$*-algebra with identity. We denote by $\Delta (A)$ the spectrum
	of $A$ and by $\mathcal{G}$ the Gelfand transformation on $A$. We recall
	that $\Delta (A)$ (a subset of the topological dual $A^{\prime }$ of $A$) is
	the set of all nonzero multiplicative linear functionals on $A$, and $%
	\mathcal{G}$ is the mapping from $A$ to $\mathcal{C}(\Delta (A))$ such that $%
	\mathcal{G}(u)(s)=\left\langle s,u\right\rangle $ ($s\in \Delta (A)$), where 
	$\left\langle ,\right\rangle $ denotes the duality pairing between $%
	A^{\prime }$ and $A$. We endow $\Delta (A)$ with the relative weak$\ast $
	topology on $A^{\prime }$. Then using the well-known theorem of Stone (see
	e.g., either \cite{larsenbanach} or more precisely \cite[Theorem IV.6.18]{dunford1988linear})
	one can easily show that the spectrum $\Delta (A)$ is a compact topological
	space, and the Gelfand transformation $\mathcal{G}$ is an isometric
	isomorphism identifying $A$ with $\mathcal{C}(\Delta (A))$ (the continuous
	functions on $\Delta (A)$) as $\mathcal{C}$*-algebras. Next, since each
	element of $A$ possesses a mean value, this yields a map $u\mapsto M(u)$
	(denoted by $M$ and called the mean value) which is a nonnegative continuous
	linear functional on $A$ with $M(1)=1$, and so provides us with a linear
	nonnegative functional $\psi \mapsto M_{1}(\psi )=M(\mathcal{G}^{-1}(\psi ))$
	defined on $\mathcal{C}(\Delta (A))=\mathcal{G}(A)$, which is clearly
	bounded. Therefore, by the Riesz-Markov theorem, $M_{1}(\psi )$ is
	representable by integration with respect to some Radon measure $\beta $ (of
	total mass $1$) in $\Delta (A)$, called the $M$\textit{-measure} for $A$ %
	\cite{nguetseng2003homogenization}. It is evident that we have 
	\begin{equation*}
		M(u)=\int_{\Delta (A)}\mathcal{G}(u)d\beta \;\, \text{\ for }u\in A\text{.}
		\label{3.5'}
	\end{equation*}
	\label{p2.0}
	We recall that (see \cite[Proposition 2]{sango2}), if 
	the $H$-supralgebra $A$ separates the points
	of $\mathbb{R}^{d}$, then $\Delta (A)$ is the Stone-\v{C}ech
	compactification of $\mathbb{R}^{d}$.

		We give now some examples of $H$-supralgebras (more precisely of $H$-algebras).
	\begin{example}\cite{wou}\label{p2.5}  
		\begin{itemize}
			\item[(1)] \textup{We denote by $\mathcal{C}_{\text{\emph{per}}}(Y)$\  the
				$H$-algebra of $Y$-periodic continuous functions on $\mathbb{R}_{y}^{d}$\ ($Y=(-%
				\frac{1}{2},\frac{1}{2})^{d}$). Its spectrum is the $N$-dimensional
				torus $\mathbb{T}^{d}=\mathbb{R}^{d}/\mathbb{Z}^{d}$. }
			\item[(2)] \textup{We denote by $AP(\mathbb{R}_{y}^{d})$  the set of all almost periodic continuous
				functions on $\mathbb{R}_{y}^{d}$ defined as the vector space consisting of
				all functions defined on $\mathbb{R}_{y}^{d}$ that are uniformly
				approximated by finite linear combinations of the functions in the set $%
				\{\exp (2i\pi k\cdot y):k\in \mathbb{R}^{d}\}$. Then $AP(\mathbb{R}_{y}^{d})$ is an $H$-algebra. Its spectrum $\Delta
				(AP(\mathbb{R}_{y}^{N}))$\ is a compact topological group homeomorphic to
				the Bohr compactification of $\mathbb{R}^{d}$. }
			\item[(3)] \textup{ Let the $H$-algebra $\mathcal{B}_{\infty}(\mathbb{R}^{d}_{y})$, denoting the space of all function $u\in \mathcal{C}(\mathbb{R}^{d}_{y})$ with $\lim_{|y|\to +\infty} u(y)= \zeta \in \mathbb{C}$ ($\zeta$ depending on $u$), where $|y|$ denotes the euclidean norm of $y$ in $\mathbb{R}^{d}_{y}$.  }
			\item[(4)] \textup{ Let $\mathcal{B}_{\infty,\mathbb{Z}^{d}}(\mathbb{R}^{d}_{y})$ denote the closure in $\mathcal{B}(\mathbb{R}^{d}_{y})$ of the space of all finite sums $\sum_{\text{finite}} \varphi_{i}u_{i}$ with $\varphi_{i} \in \mathcal{B}_{\infty}(\mathbb{R}^{d}_{y})$, $u_{i} \in \mathcal{C}_{per}(Y)$.  Then  $\mathcal{B}_{\infty,\mathbb{Z}^{d}}(\mathbb{R}^{d}_{y})$ is an $H$-algebra.}
			\item[(5)] \textup{ We denote by $WAP(\mathbb{R}_{y}^{d})$  the set of all weakly almost periodic continuous
				functions on $\mathbb{R}_{y}^{d}$ defined as the vector space consisting of
				all functions $u$ defined on $\mathbb{R}_{y}^{d}$ such that the set of translates $\{ \tau_{b}\, : \, b \in \mathbb{R}_{y}^{d}\}$ is relatively weakly compact in $\mathcal{B}(\mathbb{R}_{y}^{d})$. Then $WAP(\mathbb{R}_{y}^{d})$ is an $H$-algebra.}
			\item[(6)] \textup{ We denote by $FS(\mathbb{R}^{d}_{y}) $ the Fourier-Stieltjes algebra on $\mathbb{R}^{d}_{y}$ is defined as the closure in $\mathcal{B}(\mathbb{R}^{d}_{y})$ of the space }
			\begin{eqnarray*}
				FS_{\ast}(\mathbb{R}^{d}_{y}) = \left\{ f: \mathbb{R}^{d}_{y} \to \mathbb{R}, \;\; f(x)= \int_{\mathbb{R}^{d}_{y}} \exp(ix\cdot y) d\nu(y) \right. \nonumber \\ 
				\textup{ for some }\;\; \nu \in \mathcal{M}_{\ast}(\mathbb{R}^{d}_{y}) \bigg\},
			\end{eqnarray*}
			\textup{	where $\mathcal{M}_{\ast}(\mathbb{R}^{d}_{y})$ denotes the space of complex valued measures $\nu$ with finite total variation: $|\nu|(\mathbb{R}^{d}_{y}) < \infty$. }
		\end{itemize}
	\end{example}

	Next, the partial derivative of index $i$ ($1\leq i\leq d$) on $\Delta (A)$
	is defined to be the mapping $\partial _{i}=\mathcal{G}\circ \partial
	/\partial y_{i}\circ \mathcal{G}^{-1}$ (usual composition) of $\mathcal{D}%
	^{1}(\Delta (A))=\{\varphi \in \mathcal{C}(\Delta (A)):\mathcal{G}%
	^{-1}(\varphi )\in A^{1}\}$ into $\mathcal{C}(\Delta (A)),$ where $%
	A^{1}=\{\psi \in \mathcal{C}^{1}(\mathbb{R}^{d}):$ $\psi ,\partial \psi
	/\partial y_{i}\in A$ ($1\leq i\leq d$)$\}$. Higher order derivatives are
	defined analogously, and one also defines the space $A^{m}$ (integers $m\geq
	1$) to be the space of all $\psi \in \mathcal{C}^{m}(\mathbb{R}_{y}^{d})$
	such that $D_{y}^{\alpha }\psi =\frac{\partial ^{\left\vert \alpha
			\right\vert }\psi }{\partial y_{1}^{\alpha _{1}}\cdot \cdot \cdot \partial
		y_{d}^{\alpha _{d}}}\in A$ for every $\alpha =(\alpha _{1},...,\alpha
	_{d})\in \mathbb{N}^{d}$ with $\left\vert \alpha \right\vert \leq m$, and we
	set $A^{\infty }=\cap _{m\geq 1}A^{m}$. At the present time, let $\mathcal{D}%
	(\Delta (A))=\{\varphi \in \mathcal{C}(\Delta (A)):$ $\mathcal{G}%
	^{-1}(\varphi )\in A^{\infty }\}.$ Endowed with a suitable locally convex
	topology, $A^{\infty }$ (resp. $\mathcal{D}(\Delta (A))$) is a Fr\'{e}chet
	space and further, $\mathcal{G}$ viewed as defined on $A^{\infty }$ is a
	topological isomorphism of $A^{\infty }$ onto $\mathcal{D}(\Delta (A))$. It
	is worth recalling that $A^{\infty }$ is the deterministic analog of the
	space $\mathcal{C}^{\infty }(\Omega )$ defined in Subsection 2.1.
	
	Analogously to the space $\mathcal{D}^{\prime }(\mathbb{R}^{d})$, we now
	define the space of distributions on $\Delta (A)$ to be the space of all
	continuous linear form on $\mathcal{D}(\Delta (A))$. We denote it by $%
	\mathcal{D}^{\prime }(\Delta (A))$ and we endow it with the strong dual
	topology. 
	
	The following  result whose proof can be found in \cite{sango2}, 
	connect the dynamical systems to the spectrum of some $H$-supralgebras.
	
	\begin{proposition}\cite{sango2} \label{t2.2}  \\
		Let $A$ be an $H$-supralgebra on $\mathbb{R}^{d}$. Suppose $A$
		is translation invariant and that each of its elements is uniformly
		continuous (thus $A$ is an algebra with mean value). Then the translations $%
		T(y):\mathbb{R}^{d}\rightarrow \mathbb{R}^{d}$, $T(y)x=x+y$, extend to a
		group of homeomorphisms $T(y):\Delta (A)\rightarrow \Delta (A)$, $y\in 
		\mathbb{R}^{d}$, which forms a continuous $d$-dimensional dynamical system
		on $\Delta (A)$ whose invariant probability measure is precisely the $M$%
		-measure $\beta $ for $A$.
	\end{proposition}

	With the above result, one may directly consider deterministic
	homogenization theory in algebras with mean value as a particular case of
	stochastic homogenization theory.
	
	\subsubsection{Orlicz spaces on spectrum of supralgebra}\label{labelsub2sub2sect2} 
	
	We recall here some notions about the  Orlicz spaces associated to a $H$%
	-supralgebra $A$. We refer to \cite{nnang2014deterministic} for further details.
	
	Let $\Phi$ be an $N$-function, and $u\in A$. We have $\Phi(|u|) \in A$, so that 
	\begin{equation*}
		\Phi(|u|)^{\varepsilon} \rightarrow \int_{\Delta (A)} \mathcal{G}(\Phi(|u|)) d\beta
	\end{equation*}
	in $L^{\infty}(\mathbb{R}^{d}_{x})$-weak $\ast$ as $\varepsilon \to 0$ (see \cite{nnang2014deterministic}).
	Moreover, we have
	\begin{equation}\label{gelfandphi1}
		\mathcal{G}(\Phi(|u|)) = \Phi(|\mathcal{G}(u)|).
	\end{equation}
	
	This being so, for $u\in \mathcal{C}(\Delta(A))$ the function $s \rightarrow \Phi(|u(s)|)$ makes sense and lies in $\mathcal{C}(\Delta(A); \mathbb{R})$. We define the Orlicz space 
	\begin{equation*}
		L^{\Phi}(\Delta(A)) = \left\{ u : \Delta(A) \rightarrow \mathbb{C} \; \textup{measurable}; \;\; \lim_{\delta \to 0^{+}} \int_{\Delta (A)} \Phi(\delta|u|) d\beta = 0 \right\}
	\end{equation*}
	which is a Banach space for the Luxemburg norm 
	\begin{equation*}
		\|u\|_{\Phi,\Delta(A)} = \inf  \left\{ \delta > 0 \; : \;\;  \int_{\Delta (A)} \Phi\left(\dfrac{|u(s)|}{\delta}\right) d\beta(s) \leq 1 \right\}.
	\end{equation*}
	Now,  if $A^{\infty }$ is dense in $A$
	(this is the case when, e.g., $A$ is translation invariant and moreover each
	element of $A$ is uniformly continuous; see \cite[Proposition 2.3]{woukeng2010homogenization} for
	the justification)
	then $L^{\Phi}(\Delta (A))$ is a subspace of $\mathcal{D%
	}^{\prime }(\Delta (A))$ (with continuous embedding), so that one may define
	the Orlicz-Sobolev spaces on $\Delta (A)$ as follows. 
	\begin{equation*}
		W^{1}L^{\Phi}(\Delta (A))=\{u\in L^{\Phi}(\Delta (A)):\text{ }\partial _{i}u\in
		L^{\Phi}(\Delta (A))\text{ (}1\leq i\leq d\text{)}\}
	\end{equation*}%
	where the derivative $\partial _{i}u$ is taken in the distribution sense on $%
	\Delta (A)$. We equip $W^{1}L^{\Phi}(\Delta (A))$ with the norm
	
	\begin{equation*}
		\begin{array}{l}
			||u||_{W^{1}L^{\Phi}} = ||u||_{\Phi,\Delta(A)} +  \sum_{i=1}^{d}||\partial _{i}u||_{\Phi,\Delta(A)}  \text{ \thinspace }\left( u\in W^{1}L^{\Phi}(\Delta (A))\right) , 
		\end{array}%
	\end{equation*}%
	which makes it a Banach space. To that space are attached some other spaces
	such as
	$$W^{1}L^{\Phi}(\Delta (A))/\mathbb{C}= \left\{ u\in W^{1}L^{\Phi}(\Delta
	(A)):\int_{\Delta (A)}ud\beta =0 \right\},$$
	provided with the seminorm
	\begin{equation*}
		\begin{array}{l}
			||u||_{W^{1}L^{\Phi}(\Delta (A))/\mathbb{C}} =   \sum_{i=1}^{d}||\partial _{i}u||_{\Phi,\Delta(A)}  \text{ \thinspace }\left( u\in W^{1}L^{\Phi}(\Delta (A))\right) , 
		\end{array}%
	\end{equation*}
	and its separated completion $%
	W^{1}_{\#}L^{\Phi}(\Delta (A))$; we refer to \cite{nnang2014deterministic} for a documented
	presentation of these spaces. We denote by $J$ the canonical mapping of $W^{1}L^{\Phi}(\Delta (A))/\mathbb{C}$ into its separated completion. Furthermore, the derivative $\partial_{i}$ viewed as a mapping of $W^{1}L^{\Phi}(\Delta (A))/\mathbb{C}$ into $L^{\Phi}(\Delta(A))$ extends to a unique continuous linear mapping, still denoted by $\partial_{i}$, of $W^{1}_{\#}L^{\Phi}(\Delta (A))$ into $L^{\Phi}(\Delta(A))$ such that $\partial_{i} J(v) = \partial_{i} v$ for $v \in W^{1}L^{\Phi}(\Delta (A))/\mathbb{C}$ and  
	\begin{equation*}
		\begin{array}{l}
			||u||_{W^{1}_{\#}L^{\Phi}} =   \sum_{i=1}^{d}||\partial _{i}u||_{\Phi,\Delta(A)}  \text{ \thinspace }\left( u\in W^{1}_{\#}L^{\Phi}(\Delta (A))\right). 
		\end{array}%
	\end{equation*}
	\begin{definition}
		An $H$-supralgebra $A$ is termed $\Phi$-total if $\mathcal{D}(\Delta(A))$ is dense in $W^{1}L^{\Phi}(\Delta (A))$.
	\end{definition}
	
	Suppose that $A$ is $\Phi$-total. Then the following assertions hold true (see, \cite{nnang2014deterministic}):
	\begin{itemize}
		\item[\textup{(i)}] $J(\mathcal{D}(\Delta(A))/\mathbb{C})$ is dense in $W^{1}_{\#}L^{\Phi}(\Delta (A))$, where $\mathcal{D}(\Delta(A))/\mathbb{C} = \mathcal{D}(\Delta(A)) \cap  W^{1}L^{\Phi}(\Delta (A))/\mathbb{C}$.
		\item[\textup{(ii)}] $\int_{\Delta (A)} \partial_{i} u\, d\beta = 0$ ($1\leq i\leq d$) for any $u \in W^{1}_{\#}L^{\Phi}(\Delta (A))$
	\end{itemize}

	\subsubsection{An appropriate representation in $\mathbb{R}^{d}$  of elements in $L^{\Phi}(\Delta(A))$ and $W^{1}L^{\Phi}(\Delta (A))$}\label{labelsub2sub3sect2} 
	
	We can define the appropriate representation of  Orlicz spaces associated to a $H$%
	-supralgebra. The notations are those of the preceding subsection.
	Let $\Xi^{\Phi} = \Xi^{\Phi}(\mathbb{R}^{d}_{y})$ be the space of all $u\in L^{\Phi}_{\textup{loc}}(\mathbb{R}^{d}_{y})$ for which the sequence $(u^{\varepsilon})_{0<\varepsilon\leq 1}$ is bounded in $L^{\Phi}_{\textup{loc}}(\mathbb{R}^{d}_{y})$ (where $u^{\varepsilon}(x) = u(H_{\varepsilon}(x))$, $x\in \mathbb{R}^{d}$, $H_{\varepsilon}$ defined in (\ref{2.1})). Provided with the norm 
	\begin{equation} \label{normchi2}
		\|u\|_{\Xi^{\Phi}} = \sup_{0< \varepsilon\leq 1}  \|u^{\varepsilon}\|_{\Phi,B_{d}}, \quad u \in \Xi^{\Phi},
	\end{equation}
	where $B_{d}$ denotes the open unit ball in $\mathbb{R}^{d}_{x}$, $\Xi^{\Phi}$ is a Banach space. We define the Banach space $\mathfrak{X}^{\Phi}_{A}$ to be the closure of $A$ in $\Xi^{\Phi}$. Furthermore, the results in \cite{nguetseng2003homogenization} together with their proofs carry over mutatis mutantis by replacing the space $\mathfrak{X}^{p}_{A}$ by $\mathfrak{X}^{\Phi}_{A}$. For any $u \in A$, one has 
	\begin{equation}
		M\left(\Phi\left(\dfrac{|u|}{\delta}\right)\right) = \int_{\Delta(A)} \mathcal{G}\left(\Phi\left(\dfrac{|u|}{\delta}\right)\right) d\beta =\int_{\Delta (A)} \Phi\left(\dfrac{|\hat{u}|}{\delta}\right) d\beta \geq 0 \label{meannew1}
	\end{equation}%
	for all $\delta>0$, where $\hat{u} = \mathcal{G}(u)$ (see, (\ref{gelfandphi1})). Naturally, (\ref{meannew1}) holds when $u\in \mathfrak{X}^{\Phi}_{A}$. So, instead of the $\Xi^{\Phi}$-norm, we endow $\mathfrak{X}^{\Phi}_{A}$ with the seminorm 
	\begin{equation*}\label{normchi1}
		\|u\|_{\Phi,A} = \inf \left\{ \delta>0 \; : \; 	M\left(\Phi\left(\dfrac{|u|}{\delta}\right)\right) \leq 1 \right\}, \quad u \in \mathfrak{X}^{\Phi}_{A},
	\end{equation*}	
	which, by (\ref{meannew1}) yields that 
	\begin{equation*}
		\|\mathcal{G}(u)\|_{\Phi,\Delta(A)} = \|u\|_{\Phi,A} \quad  u \in \mathfrak{X}^{\Phi}_{A}.	
	\end{equation*}
	Hence, endowed with the seminorm $\|\cdot\|_{\Phi,A}$, $\mathfrak{X}^{\Phi}_{A}$ is complete,
	verifying $\mathfrak{X}^{\Phi}_{A}\subset \mathfrak{X}^{\Psi}_{A}$
	when $\Phi\succ \Psi$.
	
	We recall that the spaces $\mathfrak{X}^{\Phi}_{A}$ are not in general Fr\'{e}chet spaces since they are not separated in
	general. The following properties are worth noticing (see, \cite{sango2} for the Lebesgue spaces and \cite{nnang2014deterministic} for the last property):
	\begin{itemize}
		\item[(\textbf{1)}] The Gelfand transformation $\mathcal{G}:A\rightarrow 
		\mathcal{C}(\Delta (A))$ extends by continuity to a unique continuous linear
		mapping, still denoted by $\mathcal{G}$, of $\mathfrak{X}^{\Phi}_{A}$ into $L^{\Phi}(\Delta
		(A))$. Furthermore if $u\in \mathfrak{X}^{\Phi,\infty}_{A}=\mathfrak{X}^{\Phi}_{A}\cap L^{\infty }(\mathbb{R}_{y}^{d})$
		then $\mathcal{G}(u)\in L^{\infty }(\Delta (A))$ and $\left\Vert \mathcal{G}%
		(u)\right\Vert _{L^{\infty }(\Delta (A))}\leq \left\Vert u\right\Vert
		_{L^{\infty }(\mathbb{R}_{y}^{d})}$.
		
		\item[(\textbf{2)}] The mean value $M$ viewed as defined on $A$, extends by
		continuity to a positive continuous linear form (still denoted by $M$) on $%
		\mathfrak{X}^{\Phi}_{A}$ satisfying $M(u)=\int_{\Delta (A)}\mathcal{G}(u)d\beta $ ($u\in
		\mathfrak{X}^{\Phi}_{A}$). Furthermore, $M(\tau _{a}u)=M(u)$ for each $u\in \mathfrak{X}^{\Phi}_{A}$ and
		all $a\in \mathbb{R}^{d}$, where $\tau _{a}u(y)=u(y-a)$ for almost all $y\in 
		\mathbb{R}^{N}$.
		
		\item[(\textbf{3)}] Let $\Phi$ be an $N$-function and $\widetilde{\Phi}$ its complementary. The topological dual of $\mathfrak{X}^{\Phi}_{A}$ coincides with $\mathfrak{X}^{\widetilde{\Phi}}_{A}$. Hence, the usual multiplication $A\times
		A\rightarrow A$; $(u,v)\mapsto uv$, extends by continuity to a bilinear form 
		$\mathfrak{X}^{\Phi}_{A}\times \mathfrak{X}^{\widetilde{\Phi}}_{A}\rightarrow L^{1}(\Delta(A))$ with 
		\begin{equation*}
			\int_{\Delta (A)} |\mathcal{G}(u) \mathcal{G}(v)| d\beta \leq \left\| u\right\|_{\Phi,A}\left\| v\right\|_{\widetilde{\Phi},A}%
			\text{\ for }(u,v)\in \mathfrak{X}^{\Phi}_{A}\times \mathfrak{X}^{\widetilde{\Phi}}_{A}.
		\end{equation*}
		
		\item[(\textbf{4)}] \label{p2.3} Assume each
		element of $A$ is uniformly continuous and moreover $A$ is translation
		invariant (i.e. $\tau _{a}u=u(\cdot +a)\in A$ for all $u\in A$ and all $a\in 
		\mathbb{R}^{d}$). Then $A^{\infty }$ is dense in $\mathfrak{X}^{\Phi}_{A}$.
	\end{itemize}
	
	
	
	Now, let $u\in \mathfrak{X}^{\Phi}_{A}$, $\|u\|_{\Phi,A} = 0$ if and only if $\mathcal{G}(u) = 0$ (see, e.g., in \cite{nnang2014deterministic}). Unfortunately, the mapping $\mathcal{G}$ (defined on $%
	\mathfrak{X}^{\Phi}_{A}$) is not in general injective. So, we denote the space of (class of) functions by  
	\begin{equation*}
		\mathcal{X}_{A}^{\Phi} = \mathfrak{X}^{\Phi}_{A}/Ker\mathcal{G},\;\;\;\;\;\;\;\;\;
	\end{equation*}%
	where $Ker\mathcal{G}$ is the kernel of $\mathcal{G}$.
	Endowed with the norm 
	\begin{equation}\label{rholabel}
		\left\| \varrho u\right\| _{\mathcal{X}_{A}^{\Phi}}=\left\| u\right\|
		_{\Phi,A}, \quad \;\;(u\in \mathfrak{X}^{\Phi}_{A}),\;\;
	\end{equation}%
	where $\varrho$ is the canonical mapping of $\mathfrak{X}^{\Phi}_{A}$ onto $\mathcal{X}_{A}^{\Phi}$, $\mathcal{X}_{A}^{\Phi}$ is a Banach space. Moreover, according with \cite[Theorem 2.8]{nnang2014deterministic}, The mapping $\mathcal{G}:\mathfrak{X}^{\Phi}_{A}\rightarrow L^{\Phi}(\Delta (A))$
	induces an isometric isomorphism $\mathcal{G}_{1}$ of $\mathcal{X}_{A}^{\Phi}$
	onto $L^{\Phi}(\Delta (A))$ defined by 
	\begin{equation} \label{isomorphismeG1}
		\mathcal{G}_{1}(\varrho u) = \mathcal{G}(u), \quad u \in \mathfrak{X}^{\Phi}_{A}.
	\end{equation}
	So, for any $u \in L^{\Phi}(\Delta (A))$ there exists a unique $u_{0} \in \mathcal{X}_{A}^{\Phi}$ such that $\mathcal{G}_{1}(u_{0}) = u$; that is, $u_{0}$ is the representation in $\mathbb{R}^{d}_{y}$ of $u$.
	
	\begin{remark}
		For $u\in A$ and for $\delta>0$, we have 
		\begin{equation*}
			M\left(\Phi\left(\dfrac{|u|}{\delta}\right)\right) = \lim_{\varepsilon \to 0} \dfrac{1}{|B_{d}|} \int_{B_{d}} \Phi\left(\dfrac{|u(\frac{x}{\varepsilon})|}{\delta}\right)	dx,
		\end{equation*}
		where $|K| = \int 1_{K} dx$. By changing the variable $x= \varepsilon y$ we are led to 
		\begin{equation*}
			M\left(\Phi\left(\dfrac{|u|}{\delta}\right)\right) = \lim_{\varepsilon \to 0} \dfrac{1}{|\varepsilon^{-1} B_{d}|} \int_{\varepsilon^{-1} B_{d}} \Phi\left(\dfrac{|u(y)|}{\delta}\right)	dy,
		\end{equation*}
		for every $\delta>0$, which turns to be exactly the identity observed in \cite{gabri}.
	\end{remark}
	Noting that 
	\begin{equation*}
		\lim_{\varepsilon \to 0}  \int_{B_{d}} \Phi\left(\dfrac{|u(\frac{x}{\varepsilon})|}{\delta}\right)	dx \leq \sup_{0< \varepsilon\leq 1} \int_{B_{d}} \Phi\left(\dfrac{|u(\frac{x}{\varepsilon})|}{\delta}\right)	dx, 
	\end{equation*}
	we get 
	\begin{equation*}
		\left\{ \delta>0 \, : \; \sup_{0< \varepsilon\leq 1} \dfrac{1}{|B_{d}|}\int_{B_{d}} \Phi\left(\dfrac{|u(\frac{x}{\varepsilon})|}{\delta}\right)	dx \leq 1 \right\} \subset \left\{ \delta>0 \, : \;  M\left(\Phi\left(\dfrac{|u|}{\delta}\right)\right) \leq 1	 \right\}.
	\end{equation*}
	Therefore, (\ref{normchi2}) implies $\|u\|_{\Phi,A} \leq |B_{d}|^{-1} \|u\|_{\Xi^{\Phi}}$ for all $u\in A$.
	
	\par We extend the mean value on $\mathcal{X}^{\Phi}_{A}$ by 
	\begin{equation}\label{m1m}
		M_{1}(\varrho u) = M(u), \quad (u\in \mathfrak{X}^{\Phi}_{A})	
	\end{equation}
	and then we have 
	\begin{equation*}
		M_{1}(\varrho u) = \lim_{\varepsilon \to 0} \dfrac{1}{|\varepsilon^{-1} B_{d}|} \int_{\varepsilon^{-1} B_{d}} u(y) dy \quad \textup{for \, any} \;\, u \in \mathfrak{X}^{\Phi}_{A}.
	\end{equation*}
	
	Moreover, the following results hold true (see, \cite{nnang2014deterministic}):	
	\begin{itemize}  \label{c2.1}
		\item[(i)] The spaces $\mathcal{X}^{\Phi}_{A}$ is reflexive;
		
		\item[(ii)] The topological dual of the space $\mathcal{X}^{\Phi}_{A}$ coincides with the space $\mathcal{X}^{\widetilde{\Phi}}_{A}$, the duality being given by 
		\begin{equation*}
			\begin{array}{l}
				\left\langle \varrho u, \varrho v\right\rangle_{\mathcal{X}^{\Phi}_{A},\mathcal{X}^{\widetilde{\Phi}}_{A}}=M(uv)=\int_{\Delta (A)}\mathcal{G}%
				_{1}(\varrho u) \mathcal{G}_{1}(\varrho v) d\beta \\ 
				\text{for }u\in \mathfrak{X}^{\Phi}_{A} \text{ and }v\in \mathfrak{X}^{\widetilde{\Phi}}_{A}\text{.}%
			\end{array}%
		\end{equation*}
	\end{itemize}
	
	\begin{remark}
		\label{r2.1'}The space $\mathcal{X}^{\Phi}_{A}$ is the separated
			completion of $\mathfrak{X}^{\Phi}_{A}$ and the canonical mapping of $\mathfrak{X}^{\Phi}_{A}$%
		 into $\mathcal{X}^{\Phi}_{A}$ is just the canonical surjection
			of $\mathfrak{X}^{\Phi}_{A}$ onto $\mathcal{X}^{\Phi}_{A}$; see once more %
			\cite[Chap. II, Sect. 3, no 7]{bourbaki2007topologie} for the theory of completion.
	\end{remark}
	\begin{definition}
		By the formal integral representation of $u\in \mathfrak{X}^{\Phi}_{A}$, the complex number 
		\begin{equation*}
			\pounds_{A} \int_{\mathbb{R}^{d}} u(y) dy \equiv  \pounds_{A} \int u(y) dy	
		\end{equation*}
		defined by 
		\begin{equation*}
			\pounds_{A} \int_{\mathbb{R}^{d}} u(y) dy \equiv  \pounds_{A} \int u(y) dy = M_{1}(\varrho u) = M(u).	
		\end{equation*}
	\end{definition}
	
		We recall now some notions about ergodic $H$-supralgebra which we will use in the next section.
		
	\begin{definition}
		\label{d2.2} We say that an $H$-supralgebra $A$ on $\mathbb{R}^{d}_{y}$ 
		is ergodic if for any $u\in A$, 
		\begin{equation*}
			\lim_{r\rightarrow +\infty }\frac{1}{\left| B_{r}\right| }%
			\int_{B_{r}}u(x+y)dx=M(u)\text{ uniformly with respect to }y.  
		\end{equation*}
	\end{definition}
	
	The following result whose proof can be found in \cite[Lemma 2.29]{nnang2014deterministic} give us
	an equivalent property for the ergodic $H$-supralgebras.
	
	\begin{lemma}\cite{nnang2014deterministic}
		\label{p2.4}An $H$-supralgebra $A$\ on $\mathbb{R}^{d}_{y}$\ is ergodic\ if and
		only if 
		\begin{equation*}
			\lim_{r\rightarrow +\infty }\left\| \frac{1}{\left| B_{r}\right| }%
			\int_{B_{r}}u(\cdot +y)dy-M(u)\right\| _{\Phi,A}=0\text{\ for all }u\in \mathfrak{X}^{\Phi}_{A}
		\end{equation*}	
	\end{lemma}
	
	\begin{remark}\label{algegood}
		\textup{As mentioned in \cite{gabri}, each $H$-supralgebra of Example \ref{p2.5} is ergodic, of class $\mathcal{C}^{\infty}$,  translation
			invariant, and moreover each of their elements is uniformly continuous}
	\end{remark}
	
	Our goal now is to define another spaces attached to $\mathcal{X}_{A}^{\Phi}$.
	Let $u\in \mathcal{D}^{\prime }(\Delta (A))$, and let $\alpha \in \mathbb{N}%
	^{d}$. We know that $\partial ^{\alpha }u\in \mathcal{D}^{\prime }(\Delta
	(A))$ exists and is defined by 
	\begin{equation}
		\left\langle \partial ^{\alpha }u,\varphi \right\rangle =(-1)^{\left| \alpha
			\right| }\left\langle u,\partial ^{\alpha }\varphi \right\rangle \text{\ for
			any }\varphi \in \mathcal{D}(\Delta (A))\text{.}  \label{2.7'}
	\end{equation}%
	This leads to the following definition.
	
	\begin{definition}\label{formderiva1}
		\label{d2.3} The formal derivative of order $\alpha \in \mathbb{N}^{d}$
		is defined to be the formal operator on $\mathcal{X}_{A}^{\Phi}$
			given by %
		\begin{equation*}
			\overline{D}_{y}^{\alpha }=\mathcal{G}_{1}^{-1}\circ \partial ^{\alpha
			}\circ \mathcal{G}_{1}  \label{2.2'}
		\end{equation*}%
		where $\partial ^{\alpha }$ is defined above. In particular,
			for $\alpha =(\delta _{ij})_{1\leq j\leq d}$ with $1\leq i\leq d$%
		, $\overline{D}_{y}^{\alpha }$ is denoted by $\overline{%
			\partial }/\partial y_{i}$ and is called the \textit{formal partial derivative of
		index} $i$.
	\end{definition}

	Due to (\ref{isomorphismeG1}), we put $(\mathcal{G}_{1})^{-1}(W^{1}L^{\Phi}(\Delta(A))) = W^{1}\mathcal{X}^{\Phi}_{A}$, that is, 
	\begin{equation*}
		W^{1}\mathcal{X}^{\Phi}_{A} = \left\{ u \in  \mathcal{X}^{\Phi}_{A}\, : \; \partial_{i} \mathcal{G}_{1}(u) \in L^{\Phi}(\Delta(A)), \; 1\leq i \leq d \right\}.
	\end{equation*}
	Since for any $u \in W^{1}\mathcal{X}^{\Phi}_{A}$, $\partial_{i} \mathcal{G}_{1}(u) \in L^{\Phi}(\Delta(A))$ ($1\leq i \leq d$) and, once more taking into account (\ref{isomorphismeG1}), there exists a unique $u_{i} \in \mathcal{X}^{\Phi}_{A}$ such that $\partial_{i} \mathcal{G}_{1}(u) = \mathcal{G}_{1}(u_{i})$.
	
	Hence, the characterization of $W^{1}\mathcal{X}^{\Phi}_{A}$ is 
	\begin{equation*}\label{orlsobogen1}
		W^{1}\mathcal{X}^{\Phi}_{A} = \left\{ u \in  \mathcal{X}^{\Phi}_{A}\, : \; \dfrac{\overline{\partial}u}{\partial y_{i}} \in \mathcal{X}^{\Phi}_{A}, \; 1\leq i \leq d \right\}.
	\end{equation*}
	In order to simplify the representation, from now on we will use the same letter $u$ (if there is no danger of confusion) to denote the equivalence class $\varrho u \in \mathcal{X}^{\Phi}_{A}$ and its representative $u \in \mathfrak{X}^{\Phi}_{A}$. This being so, first given $u$ in $W^{1}\mathfrak{X}^{\Phi}_{A} = \left\{ u \in \mathfrak{X}^{\Phi}_{A} \; : \; \frac{\partial u}{\partial y_{i}} \in \mathfrak{X}^{\Phi}_{A} \;\;  (1\leq i \leq d) \right\}$, we have 
	\begin{equation*}
		\mathcal{G}_{1}\left(\varrho\left(\dfrac{\partial u}{\partial y_{i}}\right)\right) = \mathcal{G}\left(\dfrac{\partial u}{\partial y_{i}}\right) = \partial_{i} \mathcal{G}(u)  = \partial_{i} \mathcal{G}_{1}(\varrho(u)) \; \substack{\textup{Definition \ref{formderiva1}}  \\ =} \; \mathcal{G}_{1}\left(\dfrac{\overline{\partial} \varrho(u)}{\partial y_{i}}\right);
	\end{equation*}
	so that 
	\begin{equation}\label{partial1}
		\dfrac{\overline{\partial} }{\partial y_{i}} \circ \varrho = \varrho \circ \dfrac{\partial}{\partial y_{i}} \quad \textup{on} \;\, W^{1}\mathfrak{X}^{\Phi}_{A}, \quad 1 \leq i \leq d.
	\end{equation}
	Second, we have	
	\begin{equation*}
		\pounds_{A}\int \dfrac{\overline{\partial} u}{\partial y_{i}} dy = 0 \quad \textup{for \; all} \;\, u \in  W^{1}\mathfrak{X}^{\Phi}_{A}, \quad 1 \leq i \leq d. 
	\end{equation*}
	Third, we endow $W^{1}\mathcal{X}^{\Phi}_{A}$ with the norm
	\begin{equation*}
		\|u\|_{W^{1}\mathcal{X}^{\Phi}_{A}} = \|u\|_{\Phi, A} + \sum_{i=1}^{d} \left\|\dfrac{\overline{\partial} u}{\partial y_{i}} \right\|_{_{\Phi, A}}, \quad u \in  W^{1}\mathcal{X}^{\Phi}_{A},
	\end{equation*}
	which makes it a Banach space. Besides, the canonical mapping $\mathcal{G}_{1}$ (see (\ref{isomorphismeG1})) is an isometric isomorphism of $W^{1}\mathcal{X}^{\Phi}_{A}$ such that 
	\begin{equation*}
		\mathcal{G}_{1}(u_{0}) = u \quad \textup{and} \quad \mathcal{G}_{1}\left(\dfrac{\overline{\partial} u}{\partial y_{i}}\right) = \partial_{i}u, \quad 1\leq i \leq d.	
	\end{equation*} 
	We will be concerned with the space 
	\begin{equation*}
		W^{1}\mathcal{X}^{\Phi}_{A}/\mathbb{C} = \left\{ u \in W^{1}\mathcal{X}^{\Phi}_{A} \; :\; \pounds_{A}\int u dy = 0 \right\}
	\end{equation*}
	equipped with the seminorm
	\begin{equation*}
		\|u\|_{W^{1}\mathcal{X}^{\Phi}_{A}/\mathbb{C}} =  \sum_{i=1}^{d} \left\|\dfrac{\overline{\partial} u}{\partial y_{i}} \right\|_{_{\Phi, A}}, \quad u \in  W^{1}\mathcal{X}^{\Phi}_{A}/\mathbb{C},
	\end{equation*}
	which makes it a locally convex topological space in general nonseparated and noncomplete. We denote by $W^{1}_{\#}\mathcal{X}^{\Phi}_{A}$  the separated completion of $W^{1}\mathcal{X}^{\Phi}_{A}/\mathbb{C}$ (for the above seminorm), and by $J_{1}$ the canonical mapping from $W^{1}\mathcal{X}^{\Phi}_{A}/\mathbb{C}$ onto $W^{1}_{\#}\mathcal{X}^{\Phi}_{A}$. Therefore, there is a unique isometric isomorphism $\overline{\mathcal{G}}_{1}$ of $W^{1}_{\#}\mathcal{X}^{\Phi}_{A}$ onto $W^{1}_{\#}L^{\Phi}(\Delta(A))$; otherwise the restriction $\frac{\overline{\partial}}{\partial y_{i}} : W^{1}\mathcal{X}^{\Phi}_{A}/\mathbb{C} \to \mathcal{X}^{\Phi}_{A}$ ($1\leq i \leq d$) extends by continuity to a continuous linear mapping, still denoted by $\frac{\overline{\partial}}{\partial y_{i}}$, from $W^{1}_{\#}\mathcal{X}^{\Phi}_{A}$ into $\mathcal{X}^{\Phi}_{A}$ with 
	\begin{equation*}
		\dfrac{\overline{\partial}}{\partial y_{i}}\circ J_{1} = \dfrac{\overline{\partial}}{\partial y_{i}} \quad \textup{in} \;\,  W^{1}\mathcal{X}^{\Phi}_{A}/\mathbb{C}, \;\, 1\leq i \leq d	
	\end{equation*} 
	\begin{equation*}
		\overline{\mathcal{G}}_{1}\circ J_{1} = J_{1} \circ \mathcal{G}_{1} \quad \textup{in} \;\,  W^{1}\mathcal{X}^{\Phi}_{A}/\mathbb{C},
	\end{equation*} 
	\begin{equation*}
		\partial_{i}\circ J \circ \mathcal{G}_{1}  = \mathcal{G}_{1}\circ \dfrac{\overline{\partial}}{\partial y_{i}} \quad \textup{in} \;\,  W^{1}\mathcal{X}^{\Phi}_{A}/\mathbb{C}, \;\; \textup{or}
	\end{equation*} 
	\begin{equation*}
		\partial_{i}\circ  \overline{\mathcal{G}}_{1}  = \mathcal{G}_{1}\circ \dfrac{\overline{\partial}}{\partial y_{i}} \quad \textup{in} \;\,  W^{1}_{\#}\mathcal{X}^{\Phi}_{A},
	\end{equation*}
	\begin{equation*}
		\|J_{1}u\|_{W^{1}_{\#}\mathcal{X}^{\Phi}_{A}}  = \|u\|_{W^{1}\mathcal{X}^{\Phi}_{A}/\mathbb{C}}, \quad  u \in   W^{1}\mathcal{X}^{\Phi}_{A}/\mathbb{C}, \quad \textup{and}
	\end{equation*}
	\begin{equation*}
		\|u\|_{W^{1}_{\#}\mathcal{X}^{\Phi}_{A}} =  \sum_{i=1}^{d} \left\|\dfrac{\overline{\partial} u}{\partial y_{i}} \right\|_{_{\Phi, A}}, \quad u \in  W^{1}_{\#}\mathcal{X}^{\Phi}_{A}.
	\end{equation*}
	Furthermore, as $J_{1}(W^{1}\mathcal{X}^{\Phi}_{A}/\mathbb{C})$ is
	dense in $W^{1}_{\#}\mathcal{X}^{\Phi}_{A}$ (this is classical), it follows that if $%
	A^{\infty }$ is dense in $A$ (this is the case when $A$ is an algebra with
	mean value), then $(J_{1}\circ \varrho )(A^{\infty }/\mathbb{C})$ is dense
	in $W^{1}_{\#}\mathcal{X}^{\Phi}_{A}$, where $A^{\infty }/\mathbb{C}=\{u\in A^{\infty
	}:M(u)=0\}$.

	\bigskip We return for a while to the framework of the preceding subsection
	and assume that the hypotheses in Proposition (\ref{t2.2})\ are satisfied. Let $%
	\{T(y):y\in \mathbb{R}^{d}\}$ be the dynamical system constructed in
	Proposition (\ref{t2.2}). We know by the results of Subsection 2.1 that $T(y)$ induces a $%
	d $-parameter group of isometries $U(y):L^{\Phi}(\Delta (A))\rightarrow
	L^{\Phi}(\Delta (A))$. By the properties of $\mathcal{G}_{1}$, this also yields
	a $d$-parameter group of isometries $\mathcal{G}_{1}^{-1}\circ U(y)\circ 
	\mathcal{G}_{1}:W^{1}\mathcal{X}^{\Phi}_{A}\rightarrow W^{1}\mathcal{X}^{\Phi}_{A}$. We
	denote by $D_{i,\Phi}$ the generators of $\mathcal{G}_{1}^{-1}\circ U(y)\circ 
	\mathcal{G}_{1}$. Now, let $u\in A^{1}$; we have $\partial _{i}\mathcal{G}%
	(u)=\mathcal{G}(\frac{\partial u}{\partial y_{i}})=\mathcal{G}_{1}(\varrho (%
	\frac{\partial u}{\partial y_{i}}))$, so that $\varrho (\frac{\partial u}{%
		\partial y_{i}})=\frac{\overline{\partial }}{\partial y_{i}}(\varrho (u))$
	by the preceding remark (see, (\ref{partial1})). But since $\frac{\partial u}{\partial y_{i}}$ is
	the derivative along the direction $e_{i}=(\delta _{ij})_{1\leq i\leq d}$ of
	the dynamical system induced by the translations in $\mathbb{R}^{d}$, it is
	immediate that 
	\begin{equation*}
		D_{i,\Phi}(\varrho (u))=\frac{\overline{\partial }}{\partial y_{i}}(\varrho
		(u)),
	\end{equation*}%
	so that 
	\begin{equation*}
		D_{i,\Phi}=\frac{\overline{\partial }}{\partial y_{i}}.\;\;\;\;\;\;\;\;\;
		\label{2.12}
	\end{equation*}%
	The above equality is crucial in the process of viewing homogenization in
	algebras with mean value as a special case of stochastic homogenization.
	Indeed in the case when $\Omega =\Delta (A)$, it allows to just replace $%
	\mathcal{C}^{\infty }(\Omega )$ by the space $\mathcal{G}_{1}(\varrho
	(A^{\infty }))=\mathcal{G}(A^{\infty })=\mathcal{D}(\Delta (A))$ which plays
	exactly the same role since firstly, it is dense in $L^{\Phi}(\Delta (A))$  and secondly, for all $u\in \mathcal{D}(\Delta
	(A))\equiv \mathcal{C}^{\infty }(\Delta (A))$ we have $u\in L^{\infty
	}(\Delta (A))$ and $\partial ^{\alpha }u\in L^{\infty }(\Delta (A))$ for all 
	$\alpha \in \mathbb{N}^{d}$.
	

	
	\section{stochastic $\Sigma$-convergence in Orlicz spaces} \label{labelsect3} 
	
	In this section we extend the concept of stochastic $\Sigma $-convergence to Orlicz setting. This is a combination of two convergence methods in this type of space : the stochastic  two-scale convergence in the mean in Orlicz space (see, 
	 \cite[J. Dongho \textit{et al.}]{franck}) and the $\Sigma $-convergence in Orlicz space (see, \cite[H. Nnang]{nnang2014deterministic}%
	). 
	
	In all that follows, $Q$ is an open subset of $\mathbb{\mathbb{R}}^{d}$
	and $A$ is an $H$-supralgebra on $\mathbb{\mathbb{R}}_{y}^{d}$. We use the
	letter $\mathcal{G}$ to denote the Gelfand transformation on $A$. Points in $%
	\Delta (A)$ are denoted by $s$. We still denote by $M$ the mean value on $%
	\mathbb{\mathbb{R}}^{d}$ for the action $\mathcal{H}$ (see Section 2). The
	compact space $\Delta (A)$ is equipped with the $M$-measure $\beta $ for $A$%
	. Next, let $(\Omega ,\mathcal{M},\mu )$ denote a probability space and let $%
	\{T(y):y\in \mathbb{\mathbb{R}}^{d}\}$ denote a $d$-dimensional dynamical
	system acting on the probability space $(\Omega ,\mathcal{M},\mu )$. Points
	in $\Omega $ are denoted by $\omega $. Assume that $\Phi$ is an $N$-function of $\Delta_{2}\cap\Delta^{\prime}$-class. Finally, let $\varepsilon _{1}$ and $%
	\varepsilon _{2}$ be two well separated functions of $\varepsilon $ tending
	towards zero with $\varepsilon $, that is, $0<\varepsilon _{1},\varepsilon
	_{2},\varepsilon _{2}/\varepsilon _{1}\rightarrow 0$ as $\varepsilon
	\rightarrow 0$, and such that the functions $x\mapsto x/\varepsilon _{1}$
	and $x\mapsto x/\varepsilon _{2}$ define two actions of $\mathbb{R}%
	_{+}^{\ast }$ on $\mathbb{\mathbb{R}}^{d}$.
	
	\subsection{(Weak and Strong) stochastic $\Sigma$-convergence}
	
	In this subsection, we start by define the concept of weak  stochastic $\Sigma$-convergence and we give some of her properties. Furthermore, we state and prove the compactness theorem. We end this subsection by the concept of 	strong stochastic $\Sigma$-convergence.
	\par 	Setting 
	\begin{eqnarray*}
		L^{\Phi}(Q\times\Omega; \mathcal{X}^{\Phi}_{A}) = \bigg\{ u : Q\times\Omega\times\mathbb{R}_{y}^{d} \rightarrow \mathbb{C} \; \textup{measurable}; \;\;  \\ \lim_{\delta \to 0^{+}} \iint_{Q\times\Omega} \pounds_{A}~\int \Phi(\delta|u|) dy\,dxd\mu  = 0 \bigg\},
	\end{eqnarray*}
	we define for the $H$-supralgebra $A$, an Orlicz space endowed with the Luxemburg norm
	\begin{equation*}
		\|u\|_{L^{\Phi}(Q\times\Omega; \mathcal{X}^{\Phi}_{A})} = \inf \left\{\delta>0 \; : \; \iint_{Q\times\Omega} \pounds_{A}\int \Phi\left(\dfrac{|u|}{\delta}\right) dy\, dxd\mu  \leq 1 \right\}.
	\end{equation*}
	
Let $\varepsilon>0$. For $f \in L_{l o c}^{1}\left(Q \times \Omega  \times \mathbb{R}_{y}^{d}\right)$, we set
\begin{equation*}
	f^{\varepsilon}(x, \omega)= f\left( x,T\left( \frac{x}{%
		\varepsilon _{1}}\right) \omega ,\frac{x}{\varepsilon _{2}}\right) \quad((x, \omega) \in Q \times \Omega ) \label{eq:3.1}
\end{equation*}
whenever the right-hand side has a meaning. This is the case when $f\in \mathcal{C}_{0}^{\infty }(Q)\otimes \mathcal{C}%
^{\infty }(\Omega )\otimes A$ or $f\in \mathcal{K}(Q;\mathcal{C}^{\infty }(\Omega ;A))$ (see, e.g., \cite[Definition 5]{sango2} and comment below).

	\begin{definition}
		\label{d3.1} A bounded sequence $(u_{\varepsilon })_{\varepsilon >0}$%
		 in $L^{\Phi}(Q\times \Omega )$  is
		said to weakly stochastically $\Sigma $-converge  in $L^{\Phi}(Q\times
		\Omega )$  to some $u_{0}\in L^{\Phi}(Q\times \Omega ;\mathcal{X}^{\Phi}_{A}) \equiv \mathcal{G}_{1}^{-1}(L^{\Phi}(Q\times\Omega\times\Delta(A)))$ if as $\varepsilon \rightarrow 0$, we have  
		\begin{equation}
			\iint_{Q\times \Omega }u_{\varepsilon }(x,\omega ) f^{\varepsilon}(x, \omega) dxd\mu
			\rightarrow \iint_{Q\times \Omega} \pounds_{A}\int u_{0}(x,\omega,y) f(x,\omega,y)dy\, dxd\mu  \;\;\;\;\;\;\;\;  \label{3.1}
		\end{equation}%
		for every $f\in \mathcal{C}_{0}^{\infty }(Q)\otimes \mathcal{C}%
		^{\infty }(\Omega )\otimes A$
		
		We express this by writing $%
		u_{\varepsilon }\rightarrow u_{0}$ in $L^{\Phi}(Q\times \Omega
		)$-weak $\digamma\,\Sigma $, where $\digamma$ stands for \textit{digamma} symbol in Latex.
	\end{definition}
	\begin{remark}
	We recall that $\mathcal{C}_{0}^{\infty }(Q)\otimes \mathcal{C}^{\infty
	}(\Omega )\otimes A$ is the space of functions of the form 
	\begin{equation*}
		f(x,\omega ,y)=\sum_{\text{finite}}\varphi _{i}(x)\psi _{i}(\omega )g_{i}(y)%
		\text{,\ \ }(x,\omega ,y)\in Q\times \Omega \times \mathbb{R}^{d}\text{,}
	\end{equation*}%
	with $\varphi _{i}\in \mathcal{C}_{0}^{\infty }(Q)$, $\psi _{i}\in \mathcal{C%
	}^{\infty }(\Omega )$ and $g_{i}\in A$. Such functions are dense in $%
	\mathcal{C}_{0}^{\infty }(Q)\otimes L^{\widetilde{\Phi}}(\Omega )\otimes A$, (since $\mathcal{C}^{\infty }(\Omega
	) $ is dense in $L^{\widetilde{\Phi}}(\Omega )$) and hence in $\mathcal{K}%
	(Q;L^{\widetilde{\Phi}}(\Omega ))\otimes A$, ($\mathcal{K}(Q;L^{\widetilde{\Phi}}(\Omega ))$ being the space of continuous functions of $Q$ into $%
	L^{\widetilde{\Phi}}(\Omega )$ with compact support containing in $Q$). As $\mathcal{K}%
	(Q;L^{\widetilde{\Phi}}(\Omega ))$ is dense in $L^{\widetilde{\Phi}}(Q\times \Omega )$ and $L^{\widetilde{\Phi}}(Q\times
	\Omega )\otimes A$ is dense in $L^{\widetilde{\Phi}}(Q\times \Omega ;A)$, the
	uniqueness of the stochastic $\Sigma $-limit is ensured.
\end{remark}
		
	Before continuing our study, we need to make a comparison between the weak
	stochastic $\Sigma $-convergence and other existing convergence methods
	closed to it. For that, we must first state these convergence schemes:
	
	\textbf{(1)} A sequence $(u_{\varepsilon })_{\varepsilon >0}\subset L^{\Phi}(Q)$ 
		is said to weakly $\Sigma $-converge\ in $%
		L^{\Phi}(Q)$\ to some $v_{0}\in L^{\Phi}(Q;\mathcal{X}^{\Phi}_{A})$ (in the sense of \cite[Definition 2.13]{nnang2014deterministic}) \ if as $E\ni
		\varepsilon \rightarrow 0$, we\emph{\ }have\emph{\ } 
		\begin{equation}
			\int_{Q}u_{\varepsilon }(x)f\left( x,\frac{x}{\varepsilon _{2}}\right)
			dx \rightarrow \int_{Q} \pounds_{A}\int v_{0}(x,y)f%
			(x,y) dydx \label{3.1'}
		\end{equation}%
		for every $f\in L^{\widetilde{\Phi}}(Q;A)$. We denote it by $%
		u_{\varepsilon }\rightarrow v_{0}$\textit{\ in }$L^{\Phi}\left( Q \right) $\textit{-weak} $\Sigma$.
		
		\textbf{(2)} A sequence $(u_{\varepsilon})_{\varepsilon >0}\subset
		L^{\Phi}(Q\times \Omega )$ is said to\emph{\ }%
		stochastically two-scale converge\ in the mean\ to some\emph{\ }$v_{0}\in
		L^{\Phi}(Q\times \Omega )$\ if as\emph{\ }$\varepsilon \rightarrow 0$, we have 
		\begin{equation}
			\iint_{Q\times \Omega }u_{\varepsilon }(x,\omega )f\left( x,T\left( \frac{x}{%
				\varepsilon _{1}}\right) \omega \right) dxd\mu \rightarrow \iint_{Q\times
				\Omega }v_{0}(x,\omega )f(x,\omega )dxd\mu  \label{3.2'}
		\end{equation}%
		for all admissible functions (in the sense of \cite[Section 3]{franck}) $f\in
		L^{\widetilde{\Phi}}\left( Q\times \Omega \right) $. We denote it by $%
		u_{\varepsilon }\rightarrow v_{0}$\textit{\ stoch. in }$L^{\Phi}\left( Q\times
		\Omega \right) $\textit{-weak.}
	
	\begin{remark}
		\label{r3.0} The weak stochastic $\Sigma $-convergence method
		generalizes the above two convergence methods. Indeed, it is very important
		to note that both the above definitions (\ref{3.1'}) and (\ref{3.2'}) imply
		the boundedness of the sequence $u_{\varepsilon }$ either in $%
		L^{\Phi}(Q)$ or in $L^{\Phi}(Q\times \Omega )$, accordingly. With
		this in mind, we see that if in (\ref{3.1}) we take $f\in \mathcal{C}%
		_{0}^{\infty }(Q)\otimes \mathcal{C}^{\infty }(\Omega )$, that is $f$ 
		is constant with respect to $y\in \mathbb{R}^{d}$, and next
		using the density of the latter space in $L^{\widetilde{\Phi}}(Q\times \Omega )$%
		, then (\ref{3.1}) reads as (\ref{3.2'}) with $v_{0}(x,\omega
		)=\pounds_{A}\int u_{0}(x,\omega,y) dy $ by choosing
		in $L^{\widetilde{\Phi}}(Q\times \Omega )$ admissible functions. If
		besides we take in (\ref{3.1}) $f\in \mathcal{C}_{0}^{\infty }(Q)\otimes A$%
		, that is $f$ not depending upon the random variable $\omega $%
		and further if we choose $u_{\varepsilon }$ not depending on 
		$\omega $, then using the density of $\mathcal{C}_{0}^{\infty
		}(Q)\otimes A$ in $L^{\widetilde{\Phi}}(Q;A)$ we readily get (\ref%
		{3.1'}) with $ v_{0}(x,y) =\int_{\Omega } u_{0}(x,\omega
		,y)d\mu $.
	\end{remark}
	
	\begin{remark}
		In Definition \ref{d3.1}, if we consider  $A = \mathcal{C}_{\textup{per}}(Y)$ be the classical $H$-algebra of $Y$-periodic continuous complex functions on $\mathbb{R}_{y}^{d}$ where $Y= (0,1)^{d}$ (the open unit cube in $\mathbb{R}_{y}^{d}$), then we find the concept of \textit{stochastic two-scale convergence} in Orlicz setting defined in \cite[Definition 2]{tchin2}.
	\end{remark}
	
	As in \cite[Proposition 6]{sango2}, the following result is easily verified (see, Remark \ref{r3.0} above, taking in particular $f\in \mathcal{C}%
	_{0}^{\infty }(Q)\otimes (\mathcal{C}^{\infty }(\Omega )\cap I_{nv}^{\widetilde{\Phi}}\left( \Omega \right)$).
	
	\begin{proposition}
		\label{p2}Let $(u_{\varepsilon })_{\varepsilon >0}$ be a sequence in $%
		L^{\Phi}\left( Q\times \Omega \right) $ and let $u_{0}\in L^{\Phi}(Q\times \Omega ;\mathcal{X}^{\Phi}_{A})$. If $u_{\varepsilon }\rightarrow u_{0}$
		 in $L^{\Phi}\left( Q\times \Omega \right) $-weak $\digamma\,\Sigma $, then $%
		(u_{\varepsilon })_{\varepsilon >0}$ stochastically two-scale converges in
		the mean towards $v_{0}(x,\omega )=\pounds_{A}\int u_{0}(x,\omega
		,y)dy $ and 
		\begin{equation*}
			\int_{\Omega }u_{\varepsilon }\left( \cdot ,\omega \right) \psi (\omega
			)d\mu \rightarrow \int_{\Omega} \pounds_{A}\int u_{0}\left(
			\cdot ,\omega ,y\right)  \psi (\omega ) dyd\mu  \text{ in }L^{1}(Q)\text{%
				-weak }\forall \psi \in I_{nv}^{\widetilde{\Phi}}\left( \Omega \right) .
		\end{equation*}
	\end{proposition}
	
	The next results provide us with a few examples of sequences that weakly
	stochastically $\Sigma $-converge in the sense of Definition \ref{d3.1}.
	
	\begin{proposition}
		\label{p3.3}Let $f\in \mathcal{K}(Q;\mathcal{C}^{\infty }(\Omega ;A))$.
		Then, as $\varepsilon \rightarrow 0$, 
		\begin{equation}
			\iint_{Q\times \Omega }f\left( x,T\left( \frac{x}{\varepsilon _{1}}\right)
			\omega ,\frac{x}{\varepsilon _{2}}\right) dxd\mu \rightarrow \iint_{Q\times
				\Omega } \pounds_{A}\int f\left( x,\omega ,y\right)dy\, dxd\mu.
			\label{Eqn1}
		\end{equation}
	\end{proposition}
	
	\begin{proof}
		Since $\mathcal{C}_{0}^{\infty }(Q)\otimes \mathcal{C}^{\infty }(\Omega
		)\otimes A$ is dense in $\mathcal{K}(Q;\mathcal{C}^{\infty }(\Omega ;A))$ we
		first check (\ref{Eqn1}) for $f$ in $\mathcal{C}_{0}^{\infty }(Q)\otimes 
		\mathcal{C}^{\infty }(\Omega )\otimes A$. However, it is sufficient to do it
		for $f$ under the form $f(x,\omega ,y)=\varphi (x)\psi (\omega )g(y)$ with $%
		\varphi \in \mathcal{C}_{0}^{\infty }(Q)$, $\psi \in \mathcal{C}^{\infty
		}(\Omega )$ and $g\in A$. But for such a $f$ we have 
		\begin{eqnarray*}
			\iint_{Q\times \Omega }f\left( x,T\left( \frac{x}{\varepsilon _{1}}\right)
			\omega ,\frac{x}{\varepsilon _{2}}\right) dxd\mu &=&\int_{Q}\left(
			\int_{\Omega }\psi \left( T\left( \frac{x}{\varepsilon _{1}}\right) \omega
			\right) d\mu \right) \varphi (x)g\left( \frac{x}{\varepsilon _{2}}\right) dx
			\\
			&=&\int_{Q}\left( \int_{\Omega }\psi (\omega )d\mu \right) \varphi
			(x)g\left( \frac{x}{\varepsilon _{2}}\right) dx \\
			&=&\left( \int_{\Omega }\psi (\omega )d\mu \right) \int_{Q}\varphi
			(x)g\left( \frac{x}{\varepsilon _{2}}\right) dx
		\end{eqnarray*}%
		where the second equality above is due to the Fubini's theorem and to the
		fact that the measure $\mu $ is invariant under the maps $T(y)$. But, as $%
		\varepsilon \rightarrow 0$, we have the following well-known convergence
		result (see, e.g., \cite[Proposition 2.15]{nnang2014deterministic}): 
		\begin{equation*}
			\int_{Q}\varphi (x)g\left( \frac{x}{\varepsilon _{2}}\right) dx\rightarrow
			\int_{Q} \varphi (x) \pounds_{A}\int g(y)dy\, dx \;\, \text{ as }%
			\varepsilon \rightarrow 0.
		\end{equation*}%
		Hence the sequence 
		\begin{equation*}
			\iint_{Q\times \Omega }f\left( x,T\left( \frac{x}{\varepsilon _{1}}\right)
			\omega ,\frac{x}{\varepsilon _{2}}\right) dxd\mu \rightarrow \iint_{Q\times
				\Omega } \pounds_{A}\int f\left( x,\omega ,y\right) dy\, dxd\mu.
		\end{equation*}
		
		Now, let $f\in \mathcal{K}(Q;\mathcal{C}^{\infty }(\Omega ;A))$ and let $%
		\eta >0$ be arbitrarily fixed. Let $K\subset Q$ be a compact set such that
		supp$f\subset K$. By a density argument we choose $\phi $ in $\mathcal{C}%
		_{0}^{\infty }(Q)\otimes \mathcal{C}^{\infty }(\Omega )\otimes A$ with supp$%
		\phi \subset K$, such that $\left\Vert f-\phi \right\Vert _{\infty }\leq
		\eta /(3\left\vert K\right\vert )$, $\left\vert K\right\vert $ denoting the
		Lebesgue volume of $K$. By the decomposition 
		\begin{equation*}
			\begin{array}{l}
				\displaystyle	\iint_{Q\times \Omega }f^{\varepsilon }dxd\mu -\iint_{Q\times \Omega} \pounds_{A}\int f dy dxd\mu  = \displaystyle \iint_{Q\times \Omega }(f^{\varepsilon
				}-\phi ^{\varepsilon })dxd\mu \\ 
				\ \ \ \ \ \displaystyle +\iint_{Q\times \Omega }\phi ^{\varepsilon }dxd\mu -\iint_{Q\times \Omega} \pounds_{A}\int \phi  dy dxd\mu 
				+\iint_{Q\times \Omega
				} \pounds_{A}\int (\phi - f) dy dxd\mu ,%
			\end{array}%
		\end{equation*}%
		it follows readily that there exists $\varepsilon _{0}>0$ such that 
		\begin{equation*}
			\left\vert \iint_{Q\times \Omega }f^{\varepsilon }dxd\mu -\iint_{Q\times \Omega} \pounds_{A}\int f dy dxd\mu  \right\vert \leq \eta 
			\text{\ for }0<\varepsilon \leq \varepsilon _{0}\text{.}
		\end{equation*}%
		This completes the proof.
	\end{proof}
	
	As a result, we have the following corollaries.
	
	\begin{corollary}
		\label{c3.1}Let $u\in \mathcal{K}(Q;\mathcal{C}^{\infty }(\Omega ;A))$. Then, as $\varepsilon \rightarrow 0$,
		
		\begin{itemize}
			\item[(i)] $u^{\varepsilon }\rightarrow \varrho (u)$  in $%
			L^{\Phi}(Q\times \Omega )$-weak $\digamma\,\Sigma $, where $\varrho $ denote the
			canonical mapping of $\mathfrak{X}^{\Phi}_{A}$ into $\mathcal{X}^{\Phi}_{A}$, and the
			function $\varrho (u)$ is defined by $\varrho (u)\equiv \varrho
			(u(x,\omega,\cdot ))$ for a.e. $(x,\omega )\in Q\times \Omega $;
			
			\item[(ii)] $\left\| u^{\varepsilon }\right\| _{L^{\Phi}(Q\times \Omega
				)}\rightarrow \left\| \varrho (u)\right\| _{L^{\Phi}(Q\times \Omega ; \mathcal{X}^{\Phi}_{A})}$.
		\end{itemize}
	\end{corollary}
	
	\begin{proof}
		(i) For each $f\in \mathcal{C}_{0}^{\infty }(Q)\otimes \mathcal{C}^{\infty
		}(\Omega )\otimes A$ we have $uf\in \mathcal{K}(Q;\mathcal{C}^{\infty
		}(\Omega ;A))$, hence part (i) follows readily by Proposition \ref{p3.3}.
		For (ii), 	We just need to show its for each $u \in \mathcal{C}^{\infty}_{0}(Q)\otimes\mathcal{C}^{\infty}(\Omega)\otimes A$. For that, it is sufficient to do it for $f$ under the form $u(x, \omega,y) = \phi(x)\psi(\omega)g(y)$ with $\phi \in \mathcal{C}^{\infty}_{0}(Q)$, $\psi \in \mathcal{C}^{\infty}(\Omega)$ and $g \in A$. \\
		Given the fact that the measure $\mu$ is invariant under the dynamical system $T$ and by \cite[Proposition 2.15]{nnang2014deterministic} taking account (\ref{meannew1}), we have 
		\begin{equation*}
			\begin{array}{rcl}
				\displaystyle 	\lim_{\varepsilon\to 0} \| u^{\varepsilon} \|_{L^{\Phi}(Q\times \Omega
					)}  &= & \displaystyle  \lim_{\varepsilon\to 0} \inf \left\{\delta >0 \; ; \;\, \iint_{Q\times \Omega} \Phi\left(\dfrac{\phi(x)\psi(T\left(\frac{x}{\varepsilon_{1}}\right)\omega)g\left(\frac{x}{\varepsilon_{2}}\right) }{\delta} \right)dxd\mu \leq 1 \right\} \\
				&= & \displaystyle  \lim_{\varepsilon\to 0} \inf \left\{\delta >0 \; ; \;\, \iint_{Q\times \Omega} \Phi\left(\dfrac{\phi(x)g\left(\frac{x}{\varepsilon_{2}}\right)\psi(\omega) }{\delta} \right)dxd\mu_{T} \leq 1 \right\} \\
				& =& \inf \left\{\delta >0 \; ; \;\, \displaystyle \lim_{\varepsilon\to 0}  \displaystyle \iint_{Q\times \Omega} \Phi\left(\dfrac{\phi(x)g\left(\frac{x}{\varepsilon_{2}}\right)\psi(\omega) }{\delta} \right)dxd\mu \leq 1 \right\}  \\
				& = & \inf \left\{\delta >0 \; ; \;\,   \displaystyle \iiint_{Q\times \Omega\times \Delta(A)} \Phi\left(\dfrac{\phi(x)\hat{g}(s)\psi(\omega) }{\delta} \right)dxd\mu d\beta \leq 1 \right\}  \\
				& = & \inf \left\{\delta >0 \; ; \;\,   \displaystyle \iint_{Q\times \Omega} \int_{\Delta(A)} \mathcal{G}\left(\Phi\left(\dfrac{\phi(x)g(s)\psi(\omega) }{\delta} \right)\right)dxd\mu d\beta \leq 1 \right\}  \\
				& = & \inf \left\{\delta >0 \; ; \;\,   \displaystyle \iint_{Q\times \Omega} \pounds_{A}\int \Phi\left(\dfrac{|u|}{\delta}\right) dy dx d\mu \leq 1 \right\}  \\
				& = & \| \varrho(u) \|_{L^{\Phi}(Q\times \Omega ; \mathcal{X}^{\Phi}_{A})}.
			\end{array}
		\end{equation*} 
	\end{proof}

	Let now define the notion of \textit{admissible function} for Orlicz spaces, based on \cite[Definition 5]{sango2}.
	
	\begin{definition}
		\label{d3.1'} A function $u\in L^{\Phi}(Q\times \Omega ;\mathfrak{X}_{A}^{\Phi})$ 
			 is said to be admissible if the trace function $(x,\omega
		)\mapsto u(x,T(x/\varepsilon _{1})\omega ,x/\varepsilon _{2})$
			(denoted by $u^{\varepsilon }$), from $Q\times \Omega $ to $%
		\mathbb{C}$, is well-defined as an element of $L^{\Phi}(Q\times \Omega )$ 
		 and satisfies the following convergence result: 
		\begin{equation*}
			\iint_{Q\times \Omega } \Phi\left(\dfrac{\left\vert u^{\varepsilon }\right\vert}{\delta}\right) dxd\mu
			\rightarrow \iint_{Q\times \Omega} \pounds_{A}\int \Phi\left(\dfrac{\left\vert u%
				\right\vert}{\delta}\right) dy dxd\mu  \text{\ \emph{as} }\varepsilon \rightarrow 0\text{,%
			}  \label{adm}
		\end{equation*}
		for all $\delta>0$.
	\end{definition}
	\begin{remark}
	One can verify that any function in each of the following spaces is
	admissible: $\mathcal{K}(Q;L^{\Phi}(\Omega ;A))$ (the space of continuous
	functions $f:\mathbb{R}^{d}\rightarrow L^{\Phi}(\Omega ;A)$ with compact
	support contained in $Q$), $\mathcal{C}(\overline{Q}%
	;L^{\infty }(\Omega ;A))$ (for any bounded domain $Q$ in $\mathbb{R}^{d}$) and  $L^{\Phi}(Q)\otimes L^{\Phi}(\Omega) \otimes \mathfrak{X}_{A}^{\Phi}$ (when $\Phi$ is of $\Delta^{\prime}$-class). 
		\end{remark}
		
	\begin{proposition}
		\label{p3.4}Let $(u_{\varepsilon })_{\varepsilon \in E}\subset L^{\Phi}(Q\times
		\Omega )$  be a sequence which is weakly stochastically $%
		\Sigma $-convergent in $L^{\Phi}(Q\times \Omega )$ to some $u_{0}\in
		L^{\phi}(Q\times \Omega ;\mathcal{X}_{A}^{\Phi})$. Then as $E\ni \varepsilon
		\rightarrow 0$ we have \emph{(\ref{3.1})} (in Definition \emph{\ref{d3.1}})
		for any admissible function $f\in \mathcal{K}(Q;L^{\widetilde{\Phi}}(\Omega
		;\mathfrak{X}_{A}^{\widetilde{\Phi},\infty }))$ where $\mathfrak{X}_{A}^{\widetilde{\Phi},\infty }=\mathfrak{X}_{A}^{\widetilde{\Phi} }\cap L^{\infty }(\mathbb{R}^{d})$.
	\end{proposition}
	
	\begin{proof}
		The space $\mathcal{K}(Q)\otimes \mathcal{C}^{\infty }(\Omega )\otimes
		\mathfrak{X}_{A}^{\widetilde{\Phi},\infty }$ is dense in $\mathcal{K}(Q;L^{\widetilde{\Phi}}(\Omega ;\mathfrak{X}_{A}^{\widetilde{\Phi},\infty }))$. Indeed $\mathcal{C}^{\infty
		}(\Omega )\otimes \mathfrak{X}_{A}^{\widetilde{\Phi},\infty }$ is dense in $L^{\widetilde{\Phi}}(\Omega ;\mathfrak{X}_{A}^{\widetilde{\Phi},\infty })$, so that by \cite[p. 46]{bourbaki1996}, our
		claim is justified. With this in mind, we firstly check (\ref{3.1}) for $%
		f\in \mathcal{K}(Q)\otimes \mathcal{C}^{\infty }(\Omega )\otimes
		\mathfrak{X}_{A}^{\widetilde{\Phi},\infty }$. It suffices to verify this for $f$ under the
		form 
		\begin{eqnarray*}
			f(x,\omega ,y) &=&\varphi (x)\psi (\omega )v(y)\text{\ \ }(x\in Q,\omega \in
			\Omega ,y\in \mathbb{R}^{N})\text{\ with} \\
			&  &	\varphi \in \mathcal{K}(Q),\psi \in \mathcal{C}^{\infty }(\Omega )\text{
				and }v\in \mathfrak{X}_{A}^{\widetilde{\Phi},\infty }\text{.}
		\end{eqnarray*}%
		Let $f$ be as above. Let $\delta >0$ be freely fixed, and let $w\in A$ be
		such that $\left\Vert v-w\right\Vert _{\widetilde{\Phi},A}\leq \delta $ (where we
		have used here the density of $A$ in $\mathfrak{X}_{A}^{\widetilde{\Phi}}$). Set 
		\begin{equation*}
			g(x,\omega ,y)=\varphi (x)\psi (\omega )w(y)\text{\ \ }(x\in Q,\omega \in
			\Omega ,y\in \mathbb{R}^{N}),
		\end{equation*}%
		which gives a function $g\in \mathcal{K}(Q)\otimes \mathcal{C}^{\infty
		}(\Omega )\otimes A$. We have 
		\begin{eqnarray*}
			&&\iint_{Q\times \Omega }u_{\varepsilon }f^{\varepsilon }dxd\mu
			-\iint_{Q\times \Omega} \pounds_{A}\int u_{0} f dy dxd\mu \\
			&=&\iint_{Q\times \Omega }u_{\varepsilon }\varphi (x)\psi (T(x/\varepsilon
			_{1})\omega )[v^{\varepsilon }(x/\varepsilon _{2})-w(x/\varepsilon
			_{2})]dxd\mu \\
			&&+\iint_{Q\times \Omega }u_{\varepsilon }g^{\varepsilon }dxd\mu
			-\iint_{Q\times \Omega} \pounds_{A}\int u_{0} g dy dxd\mu  \\
			&&+\iint_{Q\times \Omega } \pounds_{A} \int u_{0}\varphi \psi (%
			w-v) dydxd\mu  \\
			&=&(I)+(II)+(III).
		\end{eqnarray*}%
		As far as $(I)$ is concerned, we have 
		\begin{equation*}
			\left\vert (I)\right\vert \leq \left\Vert \varphi \right\Vert _{\infty
			}\left\Vert \psi \right\Vert _{\infty }\left\Vert u_{\varepsilon
			}\right\Vert _{L^{\Phi}(Q\times \Omega )}\left( \inf\left\{ \lambda>0 \;:\; \int_{K} \widetilde{\Phi}\left(\dfrac{\left\vert
				v^{\varepsilon }-w^{\varepsilon }\right\vert}{\lambda}\right) dx \leq 1   \right\} \right)
		\end{equation*}%
		where $K$ is a compact subset of $\mathbb{R}^{N}$ containing the support of $%
		\varphi $. But $v$ and $w$ possess mean value, so that, as $\varepsilon
		\rightarrow 0$, 
		\begin{equation*}
			\int_{K} \widetilde{\Phi}\left(\dfrac{\left\vert
				v^{\varepsilon }-w^{\varepsilon }\right\vert}{\lambda}\right) dx \rightarrow M\left(\widetilde{\Phi}\left(\dfrac{\left\vert
				v-w\right\vert}{\lambda}\right)\right)\left\vert
			K\right\vert \text{ since } \widetilde{\Phi}\left(\dfrac{\left\vert
				v-w\right\vert}{\lambda}\right)\in
			\mathfrak{X}_{A}^{1} := B^{1}_{A}\text{ (see \cite{sango2})}
		\end{equation*}%
	with	$\left\vert K\right\vert $ denoting the Lebesgue measure of $K$ and . In view of
		the equality $\left\Vert v-w\right\Vert _{\widetilde{\Phi},A}= \inf \left\{\lambda>0\, : \,  M\left(\widetilde{\Phi}\left(\dfrac{\left\vert
			v-w\right\vert}{\lambda}\right)\right) \leq 1 \right\}$, we have $\lim_{E\ni
			\varepsilon \rightarrow 0}\left\vert (I)\right\vert \leq c\delta $ where $c$
		is a positive constant independent of $\delta $. For $(III)$, we have 
		\begin{eqnarray*}
			&&\left\vert \iint_{Q\times \Omega } \pounds_{A} \int u_{0}\varphi \psi (%
			w-v) dydxd\mu \right\vert \\
			&\leq &\left\Vert u_{0}\right\Vert_{L^{\Phi}(Q\times \Omega ;
				\mathcal{X}_{A}^{\Phi})}\left\Vert \varphi \right\Vert _{\infty }\left\Vert \psi
			\right\Vert _{\infty } \inf \left\{\lambda>0\, : \, \pounds_{A} \int  \widetilde{\Phi}\left(\dfrac{\left\vert
				w-v\right\vert}{\lambda}\right) dy \leq 1 \right\} \\
			&=&c\left\Vert v-w\right\Vert _{\widetilde{\Phi},A}   \\
			&\leq & c\delta
		\end{eqnarray*}%
		where $c=\left\Vert u_{0}\right\Vert _{L^{\Phi}(Q\times \Omega ;
			\mathcal{X}_{A}^{\Phi})} \left\Vert\varphi \right\Vert _{\infty }\left\Vert \psi
		\right\Vert _{\infty }$. Next, since 
		\begin{equation*}
			\iint_{Q\times \Omega }u_{\varepsilon }g^{\varepsilon }dxd\mu \rightarrow
			\iint_{Q\times \Omega} \pounds_{A}\int u_{0}g dy dxd\mu
		\end{equation*}%
		it follows that 
		\begin{equation*}
			\lim_{E\ni \varepsilon \rightarrow 0}\left\vert \iint_{Q\times \Omega
			}u_{\varepsilon }\varphi (x)\psi (T(x/\varepsilon _{1})\omega
			)w(x/\varepsilon _{2})dxd\mu -\iint_{Q\times \Omega}%
		\pounds_{A}\int u_{0}\varphi \psi w dydxd\mu  \right\vert \leq c\delta
		\end{equation*}%
		where $c>0$ is independent of $\delta $, hence (\ref{3.1}) follows with the
		above taken $f$, since $\delta $ is arbitrary. In view of the density of $%
		\mathcal{K}(Q)\otimes \mathcal{C}^{\infty }(\Omega )\otimes \mathfrak{X}_{A}^{\widetilde{\Phi},\infty }$ in $\mathcal{K}(Q;L^{\widetilde{\Phi}}(\Omega ;\mathfrak{X}_{A}^{\widetilde{\Phi},\infty }))$ the result follows by repeating the same way of proceeding as
		done above.
	\end{proof}
	
	The next result is a mere consequence of the preceding result. 
	
	\begin{corollary}
		\label{c3.3}Let $u\in \mathcal{K}(Q;L^{\infty }(\Omega ;\mathfrak{X}_{A}^{\Phi,\infty }))$ 
		be an admissible function in the sense of Definition \emph{%
			\ref{d3.1'}}. Then the sequence $(u^{\varepsilon })_{\varepsilon >0}$ is
		weakly stochastically $\Sigma $-convergent in $L^{\Phi}(Q\times \Omega )$ to $%
		\varrho (u)$.
	\end{corollary}
	
	We define now  the concept of strong stochastic $\Sigma$-convergence in Orlicz setting.
	
	\begin{definition}
		\label{d3.2} A sequence $(u_{\varepsilon })_{\varepsilon >0}\subset
		L^{\Phi}(Q\times \Omega )$ is said to %
		strongly stochastically $\Sigma $-converge in $L^{\Phi}(Q\times \Omega
		) $ to some $u_{0}\in L^{\Phi}(Q\times \Omega ;\mathcal{X}_{A}^{\Phi})$ 
		 if  the following condition is verified: %
		\begin{equation}\label{c4eq4}
			\begin{array}{l}
				\textup{Given}\; \eta>0\; \textup{and}\; \varphi \in L^{\Phi}(Q\times\Omega; A)\; \textup{verifying}\; \|u_{0}- \varphi\|_{L^{\Phi}(Q\times \Omega ;\mathcal{X}_{A}^{\Phi})} \leq \dfrac{\eta}{2}\;  \\
				\textup{there exists}\; \rho>0\; \textup{such that} 	
				\|u_{\varepsilon}-\varphi^{\varepsilon}\|_{\Phi,Q\times\Omega} \leq \eta\;\; \textup{for\, all}\; 0<\varepsilon\leq \rho.
			\end{array}
		\end{equation}
		
		We denote this by $u_{\varepsilon }\rightarrow u_{0}$ 
			in $L^{\Phi}(Q\times \Omega )$-strong $\digamma\,\Sigma $.
	\end{definition}
	
	\begin{remark}
		\label{r3.1}
		\begin{itemize}
			\item[(1)]  By the above definition, the uniqueness of the limit
			of such a sequence is ensured.
			\item[(2)]  Note that in (\ref{c4eq4}), we have used the same letter $\varphi$  to denote the equivalence class $\varrho \varphi$ and its representative $\varphi$.
			\item[(3)]  By the Corollary \ref{c3.1} it is
			immediate that for any $u\in \mathcal{K}(Q;\mathcal{C}^{\infty }(\Omega
			;A)) $, the sequence $(u^{\varepsilon })_{\varepsilon >0}$ is
			strongly stochastically $\Sigma $-convergent to $\varrho (u)$.
			\item[(4)] 	We may the weak/strong stochastic $\Sigma $-convergence in $L^{\Phi}(Q\times\Omega)$ (see Definitions \ref{d3.1} and \ref{d3.2}) as a generalization of usual  weak/strong stochastic $\Sigma $-convergence in $L^{p}(Q\times\Omega)$ (see \cite[Definition 4 and 6]{sango2} and ) when $\Phi(t)=\frac{t^{p}}{p}$ ($t\geq 0$). 
		\end{itemize}
	\end{remark}
	
	As in \cite[Proposition 10]{sango2} the strong stochastic $\Sigma $-convergence in Orlicz space is a generalization of the
	classical strong convergence in this space as one can easily see in the following result.
	
	\begin{proposition}
		\label{p3.5}Let $(u_{\varepsilon })_{\varepsilon >0}\subset L^{\Phi}(Q\times
		\Omega )$  be a strongly convergent sequence in $%
		L^{\Phi}(Q\times \Omega )$ to some $u_{0}\in L^{\Phi}(Q\times \Omega )$. Then $%
		(u_{\varepsilon })_{\varepsilon \in E}$ strongly stochastically $\Sigma $%
		-converges in $L^{\Phi}(Q\times \Omega )$ towards $u_{0}$.
	\end{proposition}
	
	In the next subsection we give some compactness results on the weak stochastic-$\Sigma$ convergence in the Orlicz-Sobolev spaces. 
	
	\subsection{Stochastic-$\Sigma$ compactness results in Orlicz spaces and stochastic $\Sigma$-reflexivity}
	
	Throughout the paper the letter $E$ will denote any ordinary sequence $%
	E=(\varepsilon _{n})$ (integers $n\geq 0$) with $0<\varepsilon _{n}\leq 1$
	and $\varepsilon _{n}\rightarrow 0$ as $n\rightarrow \infty $. Such a
	sequence will be termed a \textit{fundamental sequence}. 
	
	\subsubsection{Compactness Theorem in Orlicz space}
	
	The following result is the point of departure of all the compactness
	results involved in this paper.
	
	\begin{theorem}
		\label{t3.1} Let $\Phi$ be a $N$-function of class $\Delta_{2}\cap \nabla_{2}$ and let $A$ be an $H$-supralgebra on $\mathbb{R}^{d}_{y}$.
		Then, any bounded sequence $(u_{\varepsilon })_{\varepsilon \in E}$ in 
		$L^{\Phi}(Q\times \Omega )$ (where $E$ is a fundamental sequence) admits a subsequence which is weakly stochastically $\Sigma $%
		-convergent in $L^{\Phi}(Q\times \Omega )$.
	\end{theorem}
	
	
	
	\begin{proof} 
		Let $Y=L^{\widetilde{\Phi}}(Q\times \Omega ; \mathcal{X}^{\widetilde{\Phi}}_{A})$, $X=\mathcal{C}%
		_{0}^{\infty }(Q)\otimes \mathcal{C}^{\infty }(\Omega )\otimes A$. Let us define the mapping $L_{\varepsilon }$ by 
		\begin{equation*}
			L_{\varepsilon }(f)=\int_{Q\times \Omega }u_{\varepsilon
			}f^{\varepsilon }dxd\mu \text{\ \ (}f \in \mathcal{C}_{0}^{\infty
			}(Q)\otimes \mathcal{C}^{\infty }(\Omega )\otimes A \text{).}
		\end{equation*}%
		where $f^{\varepsilon }(x,\omega )=f(x,T(x/\varepsilon _{1})\omega
		,x/\varepsilon _{2})$ for $(x,\omega )\in Q\times \Omega $. Since $(u_{\varepsilon})_{\varepsilon\in E}$ is bounded in $L^{\Phi}(Q\times\Omega)$, then 
		\begin{equation*}
			\underset{\varepsilon }{\lim \sup }\left\vert L_{\varepsilon }(f%
			)\right\vert \leq c\left\Vert \varrho(f)\right\Vert _{L^{\widetilde{\Phi}}(Q\times \Omega ; \mathcal{X}^{\widetilde{\Phi}}_{A})}\text{\ for all } f \in X.
		\end{equation*}%
		Indeed one has the inequality $\left\vert L_{\varepsilon }(f)\right\vert
		\leq c\left\Vert f^{\varepsilon }\right\Vert _{L^{\widetilde{\Phi}}(Q\times
			\Omega )}$ and thus, as $\varepsilon \rightarrow 0$, $\left\Vert
		f^{\varepsilon }\right\Vert _{L^{\widetilde{\Phi}}(Q\times \Omega )}\rightarrow
		\left\Vert \varrho(f)\right\Vert _{L^{\widetilde{\Phi}}(Q\times \Omega ; \mathcal{X}^{\widetilde{\Phi}}_{A})}$ (see Corollary \ref{c3.1}). We therefore apply \cite[Proposition 3.2]{wou} with the above notation to get the existence of a subsequence $%
		E^{\prime }$ of $E$ and of a unique $v_{0}\in \left(L^{\widetilde{\Phi}}(Q\times \Omega ; \mathcal{X}^{\widetilde{\Phi}}_{A})\right)^{\prime}\equiv L^{\Phi}(Q\times \Omega ; \mathcal{X}^{\Phi}_{A})$ (by reflexivity) such that 
		\begin{equation*}
		\lim_{E' \ni\varepsilon \to 0}	\iint_{Q\times \Omega }u_{\varepsilon }f^{\varepsilon }dxd\mu =
			\iint_{Q\times \Omega} \pounds_{A}\int  v_{0}f dydxd\mu 
		\end{equation*}%
		\ for all $f\in X$, and so the result
		follows.
	\end{proof}

	
	\subsubsection{Stochastic $\Sigma$-reflexivity in Orlicz-Sobolev spaces}
	
	In this subsection, we start by define the concepts of weak stochastic $\Sigma$-convergence and stochastic $\Sigma$-reflexivity (for $H$-supralgebra) in Orlicz-Sobolev spaces.
	
	We define the Banach space 
	\begin{equation*}
		W^{1}L^{\Phi}_{D_{x}}(Q; L^{\Phi}(\Omega;\mathcal{X}_{A}^{\Phi})) = \left\{ u \in L^{\Phi}(Q\times\Omega;\mathcal{X}_{A}^{\Phi})\, : \; \dfrac{\partial u}{\partial x_{i}} \in L^{\Phi}(Q\times\Omega;\mathcal{X}_{A}^{\Phi}), \, 1 \leq i \leq d \right\}
	\end{equation*}
	equiped with the norm 
	\begin{equation*}
		\|u\|_{W^{1}L^{\Phi}_{D_{x}}(Q; L^{\Phi}(\Omega;\mathcal{X}_{A}^{\Phi}))} = \|u\|_{L^{\Phi}(Q\times\Omega;\mathcal{X}_{A}^{\Phi})} + \sum_{i=1}^{d} \left\|\dfrac{\partial u}{\partial x_{i}} \right\|_{L^{\Phi}(Q\times\Omega;\mathcal{X}_{A}^{\Phi})}.
	\end{equation*}
	\begin{definition}\label{sigmaSobolev}
		A sequence $(u_{\varepsilon})_{\varepsilon\in E}\subset 	W^{1}L^{\Phi}_{D_{x}}(Q; L^{\Phi}(\Omega))$ is said to be weakly stochastically $\Sigma$-convergent (for $A$) in $	W^{1}L^{\Phi}_{D_{x}}(Q; L^{\Phi}(\Omega))$ to some $u_{0} \in 	W^{1}L^{\Phi}_{D_{x}}(Q; L^{\Phi}(\Omega;\mathcal{X}_{A}^{\Phi}))$ if there exist two  functions $u_{1} \in L^{1}\left(Q; W^{1}_{\#}L^{\Phi}(\Omega)\right)$ and $u_{2}\in L^{1}(Q\times\Omega; W^{1}_{\#}\mathcal{X}_{A}^{\Phi})$ such that as $E \ni \varepsilon \to 0$, we have 
		\begin{itemize}
			\item[(i)] $u_{\varepsilon} \rightarrow u_{0}$ in $L^{\Phi}(Q\times\Omega)$-weak $\digamma\,\Sigma$,
			\item [(ii)] $\frac{\partial u_{\varepsilon}}{\partial x_{i}} \rightarrow \frac{\partial u_{0}}{\partial x_{i}} + \overline{D}_{i,\Phi}u_{1} + \frac{\overline{\partial} u_{2}}{\partial y_{i}}$ in $L^{\Phi}(Q\times\Omega)$-weak $\digamma\,\Sigma$, $1 \leq i \leq d$.
		\end{itemize}
		We then express by: $u_{\varepsilon} \rightarrow u_{0}$ in $W^{1}L^{\Phi}_{D_{x}}(Q; L^{\Phi}(\Omega))$-weak $\digamma\,\Sigma$, and we refer to $u_{0}$ as the weak stochastic $\Sigma$-limit in $W^{1}L^{\Phi}_{D_{x}}(Q; L^{\Phi}(\Omega))$ of $(u_{\varepsilon})_{\varepsilon\in E}$. The functions $u_{1}$ and $u_{2}$ are the stochastic corrector and deterministic corrector for $(u_{\varepsilon})_{\varepsilon\in E}$, respectively. 
	\end{definition}
	
	
	\begin{definition}
		Given a bounded open set $\Omega$ in $\mathbb{R}^{d}_{x}$, the Orlicz-Sobolev space $W^{1}L^{\Phi}_{D_{x}}(Q; L^{\Phi}(\Omega))$ is said to be stochastically $\Sigma$-reflexive (for $A$) if the following holds: Given a sequence $(u_{\varepsilon})_{\varepsilon\in E}$ in $W^{1}L^{\Phi}_{D_{x}}(Q; L^{\Phi}(\Omega)$ such that $(u_{\varepsilon})$ and $(\partial_{x_{i}}u_{\varepsilon})$ are bounded in $L^{\Phi}(Q\times\Omega)$, $1\leq i\leq d$, a subsequence $E^{\prime}$ can be extracted such that the sequence $(u_{\varepsilon})_{\varepsilon\in E^{\prime}}$ is weakly stochastically $\Sigma$-convergent (for $A$) in $W^{1}L^{\Phi}_{D_{x}}(Q; L^{\Phi}(\Omega))$.
	\end{definition}

 Next, we define the concept of Proper $H$-algebras and we study the particular cases of periodic $H$-algebras and $H$-algebras with mean value. 

	\subsection{Stochastic-properness of some $H$-algebras}

In all that follows, $\Phi$ and $\widetilde{\Phi}$ are of $\Delta_{2}$-class.
	
	\begin{definition}
		The $H$-algebra $A$ (of class $\mathcal{C}^{\infty}$) on $\mathbb{R}^{d}_{y}$ is said to be stochastic $\Phi$-proper (resp. stochastic $\Phi$-pseudoproper) if the following two conditions are satisfied: 
		\begin{itemize}
			\item[(i)] $A$ is $\Phi$-total.
			\item[(ii)] For each bounded open set $Q$ in $\mathbb{R}^{d}_{x}$, the Orlicz-Sobolev space $W^{1}L^{\Phi}_{D_{x}}(Q; L^{\Phi}(\Omega))$ (resp. $W^{1}_{0}L^{\Phi}_{D_{x}}(Q; L^{\Phi}(\Omega))$) is stochastically $\Sigma$-reflexive (for $A$).
		\end{itemize}
	\end{definition}

We discuss now about the two cases : the periodic $H$-algebras and the ergodic $H$-algebras. Let us begin by the first case.
	
	\subsubsection{Case of periodic $H$-algebras}
	
	Let $A = \mathcal{C}_{\textup{per}}(Y)$ be the classical $H$-algebra of $Y$-periodic continuous complex functions on $\mathbb{R}_{y}^{d}$ where $Y= (0,1)^{d}$ (the open unit cube in $\mathbb{R}_{y}^{d}$). It is known that $\mathcal{C}_{\textup{per}}(Y)$ is of class $\mathcal{C}^{\infty}$, see \cite{nguetseng2003homogenization}. We set 
	
	\begin{eqnarray*}
		L^{\Phi}_{\textup{per}}(Y) = \left\{v \in L^{\Phi}_{\textup{loc}}(\mathbb{R}_{y}^{d}) \; : \; v \; \textup{is} \; Y\textup{-periodic}  \right\},  \\
		W^{1}L^{\Phi}_{\textup{per}}(Y) = \left\{v \in W^{1}L^{\Phi}_{\textup{loc}}(\mathbb{R}_{y}^{d}) \; : \; v \; \textup{is} \; Y\textup{-periodic}  \right\}, \\
		W^{1}_{\#}L^{\Phi}_{\textup{per}}(Y) = \left\{v \in W^{1}L^{\Phi}_{\textup{per}}(Y) \; : \; \int_{Y} v dy = 0  \right\}.
	\end{eqnarray*}
	
	$L^{\Phi}_{\textup{per}}(Y)$ is a Banach space under the $L^{\Phi}$-norm (denoted $\|\cdot\|_{\Phi,Y}$). Also each of spaces $W^{1}L^{\Phi}_{\textup{per}}(Y)$ and $W^{1}_{\#}L^{\Phi}_{\textup{per}}(Y)$ is a Banach space under norms 
	\begin{equation*}
		\|u\|_{W^{1}L^{\Phi}(Y)} = \|u\|_{\Phi,Y} + \sum_{i=1}^{d} \left\|\dfrac{\partial u}{\partial y_{i}} \right\|_{\Phi,Y}, \quad u \in W^{1}L^{\Phi}_{\textup{per}}(Y)
	\end{equation*}
	and 
	\begin{equation*}
		\|u\|_{W^{1}_{\#}L^{\Phi}(Y)} =  \sum_{i=1}^{d} \left\|\dfrac{\partial u}{\partial y_{i}} \right\|_{\Phi,Y}, \quad u \in W^{1}_{\#}L^{\Phi}_{\textup{per}}(Y),
	\end{equation*}
	respectively. Then $\mathcal{C}_{\textup{per}}(Y)$ is dense in $L^{\Phi}_{\textup{per}}(Y)$, and $\mathcal{C}^{\infty}_{\textup{per}}(Y)= \mathcal{C}^{\infty}(\mathbb{R}_{y}^{d})\cap \mathcal{C}_{\textup{per}}(Y)$ is dense in $W^{1}L^{\Phi}_{\textup{per}}(Y)$ (see, \cite{nnang2014deterministic}). We have $\pounds_{\mathcal{C}_{\textup{per}}(Y)}\int_{Y} u(y)dy = \int_{Y} u(y)dy$ ($u\in \mathfrak{X}^{\Phi}_{\mathcal{C}_{\textup{per}}(Y)} \equiv L^{\Phi}_{\textup{per}}(Y)$).
	
	\begin{proposition}
		The periodic $H$-algebra $\mathcal{C}_{\textup{per}}(Y)$ is stochastic $\Phi$-proper.
	\end{proposition}
	\begin{proof}
		See \cite[Theorem 10]{tchin2}.	
	\end{proof}
	

	\subsubsection{Case of ergodic $H$-algebras}

	Now we assume in the sequel that the $H$-supralgebra $A$ is translation
	invariant and moreover each of its elements is uniformly continuous, that
	is, $A$ is an algebra with mean value. The next result requires some
	preliminaries. Let $a\in \mathbb{\mathbb{R}}^{N}$. Since $A$ is translation
	invariant, the translation operator $\tau _{a}:A\rightarrow A$ extends by
	continuity to a unique translation operator still denoted by $\tau
	_{a}:\mathfrak{X}_{A}^{\Phi}\rightarrow \mathfrak{X}_{A}^{\Phi}$. Indeed $\tau _{a}$
	is bijective and $\left\| \tau _{a}u\right\| _{\Phi,A}=\left\| u\right\| _{\Phi,A}$
	since $M(\Phi(\left| \tau _{a}u\right|))=M(\tau _{a}\Phi(\left| u\right|))=M(\Phi(\left| u\right|))$ for all $u\in A$. Besides, as each element of 
	$A$ is uniformly continuous, the group of unitary operators $\{\tau
	_{a}:a\in \mathbb{\mathbb{R}}^{d}\}$ thus defined is strongly continuous,
	i.e. $\tau _{a}u\rightarrow u$ in $\mathfrak{X}_{A}^{\Phi}$ as $\left| a\right|
	\rightarrow 0$ for all $u\in \mathfrak{X}_{A}^{\Phi}$. Moreover 
	\begin{equation}
		M(\tau _{a}u)=M(u)\text{ for all }u\in \mathfrak{X}_{A}^{\Phi}\text{ and any }a\in \mathbb{%
			\mathbb{R}}^{d}\text{.}  \label{5.2'}
	\end{equation}%
	Arguing as above we see that the group $\{\tau _{a}\}_{a\in \mathbb{\mathbb{R%
		}}^{d}}$ yields a family of mappings still denoted by $\{\tau _{a}\}_{a\in 
		\mathbb{\mathbb{R}}^{d}}$ (each of them sending $L^{\Phi}(\Omega ;\mathfrak{X}_{A}^{\Phi})$
	into itself) verifying 
	\begin{equation*}
		\tau _{a}u\left( \omega ,y\right) =\tau _{a}u(\omega ,\cdot )(y)=u\left(
		\omega ,y+a\right) \text{ for a.e. }(\omega ,y)\in \Omega \times \mathbb{R}%
		^{d}\text{ and for }u\in L^{\Phi}(\Omega ;\mathfrak{X}_{A}^{\Phi}).
	\end{equation*}%
	With this in mind, we begin with the following preliminary results.
	
	
	\begin{lemma}
		\label{l5.1}Assume the $H$-supralgebra $A$ is an algebra with mean value on $%
		\mathbb{R}_{y}^{d}$, i.e., it is translation invariant and each of its
		elements is uniformly continuous. Let $(u_{\varepsilon })_{\varepsilon \in
			E} $ be a sequence in $L^{\Phi}(Q\times \Omega )$  which weakly
		stochastically $\Sigma $-converges towards $u_{0}\in L^{\Phi}(Q\times \Omega ;%
		\mathcal{X}_{A}^{\Phi})$. Let the sequence $(v_{\varepsilon })_{\varepsilon \in
			E}$ be defined by 
		\begin{equation*}
			v_{\varepsilon }(x,\omega )=\int_{B_{r}}u_{\varepsilon }(x+\varepsilon
			_{2}\rho ,\omega )d\rho \text{\ \ }((x,\omega )\in Q\times \Omega ).
		\end{equation*}%
		Then, as $E\ni \varepsilon \rightarrow 0$, 
		\begin{equation}
			v_{\varepsilon }\rightarrow v_{0}\text{ in }L^{\Phi}(Q\times \Omega )%
			\text{-weak } \digamma\,\Sigma  \label{5.3'}
		\end{equation}%
		where $v_{0}$ is defined by $v_{0}(x,\omega ,y)=\int_{B_{r}}u_{0}(x,\omega
		,y+\rho )d\rho $\ for $(x,\omega ,y)\in Q\times \Omega \times \mathbb{%
			\mathbb{R}}^{d}$.
	\end{lemma}
	
	\begin{remark}
		\label{r5.2}\emph{Assume Lemma \ref{l5.1} holds. Then as }$E\ni \varepsilon
		\rightarrow 0$\emph{, }%
		\begin{equation}
			\frac{1}{\left\vert B_{\varepsilon _{2}r}\right\vert }\int_{B_{\varepsilon
					_{2}r}}u_{\varepsilon }(x+y,\omega )dy\rightarrow \frac{1}{\left\vert
				B_{r}\right\vert }v_{0}\text{\  \emph{in} }L^{\Phi}(Q\times \Omega
			)\text{\emph{-weak }} \digamma\,\Sigma \text{\emph{.}}  \label{5.5'}
		\end{equation}%
		\emph{The above convergence result will be of particular interest in the
			Compactness theorem for the ergodic $H$-supralgebra.}
	\end{remark}
	
	\begin{proof}[\textit{Proof of Lemma} \ref{l5.1}]
		Let $\varphi \in \mathcal{C}_{0}^{\infty }(Q)$, $f\in \mathcal{C}^{\infty
		}(\Omega )$ and $g\in A$. One has 
		\begin{eqnarray*}
			&&\int_{Q\times \Omega }\left( \int_{B_{r}}u_{\varepsilon }(x+\varepsilon
			_{2}\rho ,\omega )d\rho \right) \varphi (x)f\left( T\left( \frac{x}{%
				\varepsilon _{1}}\right) \omega \right) g\left( \frac{x}{\varepsilon _{2}}%
			\right) dxd\mu \\
			&=&\int_{B_{r}}\left( \int_{Q\times \Omega }u_{\varepsilon }(x+\varepsilon
			_{2}\rho ,\omega )\varphi (x)f\left( T\left( \frac{x}{\varepsilon _{1}}%
			\right) \omega \right) g\left( \frac{x}{\varepsilon _{2}}\right) dxd\mu
			\right) d\rho .
		\end{eqnarray*}%
		In view of the Lebesgue dominated convergence theorem, (\ref{5.3'}) will be
		checked as soon as we show that for each fixed $\rho \in \mathbb{\mathbb{R}}%
		^{d}$, 
		\begin{eqnarray*}
			&&\int_{Q\times \Omega }u_{\varepsilon }(x+\varepsilon _{2}\rho ,\omega
			)\varphi (x)f\left( T\left( \frac{x}{\varepsilon _{1}}\right) \omega \right)
			g\left( \frac{x}{\varepsilon _{2}}\right) dxd\mu \\
			&\rightarrow &\int_{Q\times \Omega } \pounds_{A}\int \tau _{-\rho
				}u_{0}(x,\omega ,y)\varphi (x)f(\omega ) g(y) dy dxd\mu  \;\; \text{
				when }E\ni \varepsilon \rightarrow 0\text{.}
		\end{eqnarray*}%
		First, let us beginning by noticing that since $\mathcal{G}_{1}$ is a
		bounded linear operator of $\mathcal{X}_{A}^{1}= \mathcal{B}_{A}^{1}$ into $L^{1}(\Delta (A))$ (see, \cite{sango2}),
		we have 
		\begin{equation*}
			\mathcal{G}_{1}\left( \int_{Br}u_{0}(x,\omega ,\cdot +\rho )d\rho \right)
			=\int_{B_{r}}\mathcal{G}_{1}(u_{0}(x,\omega ,\cdot +\rho ))d\rho
		\end{equation*}%
		where $u_{0}$ is as above. So let $a\in \mathbb{\mathbb{R}}^{N}$ and let $%
		\varphi $, $f$ and $g$ be as above. One has 
		\begin{equation*}
			\begin{array}{l}
			\displaystyle 	\int_{Q\times \Omega }u_{\varepsilon }(x-\varepsilon _{2}a,\omega )\varphi
				(x)f\left( T\left( \frac{x}{\varepsilon _{1}}\right) \omega \right) g\left( 
				\frac{x}{\varepsilon _{2}}\right) dxd\mu  \\ 
			\displaystyle 	\;\;=\int_{\left( Q-\varepsilon _{2}a\right) \times \Omega }u_{\varepsilon
				}(x,\omega )\varphi (x+\varepsilon _{2}a)f\left( T\left( \frac{x}{%
					\varepsilon _{1}}+\frac{\varepsilon _{2}}{\varepsilon _{1}}a\right) \omega
				\right) g\left( \frac{x}{\varepsilon _{2}}+a\right) dxd\mu \\ 
			\displaystyle 	\;\;=\int_{Q\times \Omega }u_{\varepsilon }(x,\omega )\varphi (x+\varepsilon
				_{2}a)f\left( T\left( \frac{x}{\varepsilon _{1}}+\frac{\varepsilon _{2}}{%
					\varepsilon _{1}}a\right) \omega \right) g\left( \frac{x}{\varepsilon _{2}}%
				+a\right) dxd\mu \\ 
			\displaystyle 	\;\;\;\;\;-\int_{\left( Q\backslash (Q-\varepsilon _{2}a)\right) \times
					\Omega }u_{\varepsilon }(x,\omega )\varphi (x+\varepsilon _{2}a)f\left(
				T\left( \frac{x}{\varepsilon _{1}}+\frac{\varepsilon _{2}}{\varepsilon _{1}}%
				a\right) \omega \right) g\left( \frac{x}{\varepsilon _{2}}+a\right) dxd\mu
				\\ 
			\displaystyle 	\;\;\;\;\;\;\;+\int_{\left( (Q-\varepsilon _{2}a)\backslash Q\right) \times
					\Omega }u_{\varepsilon }(x,\omega )\varphi (x+\varepsilon _{2}a)f\left(
				T\left( \frac{x}{\varepsilon _{1}}+\frac{\varepsilon _{2}}{\varepsilon _{1}}%
				a\right) \omega \right) g\left( \frac{x}{\varepsilon _{2}}+a\right) dxd\mu
				\\ 
				=(I)-(II)+(III).%
			\end{array}%
		\end{equation*}%
		As for $(I)$ we have 
		\begin{eqnarray*}
			(I) &=&\int_{Q\times \Omega }u_{\varepsilon }(x,\omega )\varphi (x)f\left(
			T\left( \frac{x}{\varepsilon _{1}}+\frac{\varepsilon _{2}}{\varepsilon _{1}}%
			a\right) \omega \right) (\tau _{-a}g)\left( \frac{x}{\varepsilon _{2}}%
			\right) dxd\mu \\
			&&+\int_{Q\times \Omega }u_{\varepsilon }(x,\omega )[\varphi (x+\varepsilon
			_{2}a)-\varphi (x)]f\left( T\left( \frac{x}{\varepsilon _{1}}+\frac{%
				\varepsilon _{2}}{\varepsilon _{1}}a\right) \omega \right) (\tau
			_{-a}g)\left( \frac{x}{\varepsilon _{2}}\right) dxd\mu \\
			\;\;\;\;\;\; &=&(I_{1})+(I_{2}).
		\end{eqnarray*}%
		But 
		\begin{eqnarray*}
			(I_{1}) &=&\int_{Q\times \Omega }u_{\varepsilon }(x,\omega )\varphi
			(x)f\left( T\left( \frac{x}{\varepsilon _{1}}\right) \omega \right) (\tau
			_{-a}g)\left( \frac{x}{\varepsilon _{2}}\right) dxd\mu \\
			&&+\int_{Q\times \Omega }u_{\varepsilon }(x,\omega )\varphi (x)(\tau
			_{-a}g)\left( \frac{x}{\varepsilon _{2}}\right) \left[ f\left( T\left( \frac{%
				x}{\varepsilon _{1}}+\frac{\varepsilon _{2}}{\varepsilon _{1}}a\right)
			\omega \right) -f\left( T\left( \frac{x}{\varepsilon _{1}}\right) \omega
			\right) \right] dxd\mu \\
			&=&(I_{1}^{\prime })+(I_{2}^{\prime }).
		\end{eqnarray*}%
		The $H$-supralgebra $A$ being translation invariant, we have $\tau _{-a}g\in
		A$ and so, 
		\begin{equation*}
			(I_{1}^{\prime })\rightarrow \int_{Q\times \Omega } \pounds_{A}\int 
				u_{0}(x,\omega ,y)\varphi (x)f(\omega ) \tau _{-a}g(y) dy
			dxd\mu  \;\; \text{\ as }E\ni \varepsilon \rightarrow 0.
		\end{equation*}%
		But 
		\begin{eqnarray*}
			\pounds_{A} \int u_{0}(x,\omega ,y) \tau _{-a}g(y) dy
			&=&M(u_{0}(x,\omega ,\cdot )(\tau _{-a}g)) \\
			&=&M(\tau _{-a}[\tau _{a}u_{0}(x,\omega ,\cdot )g]) \\
			&=&M(\tau _{a}u_{0}(x,\omega ,\cdot )g)\text{\ (see (\ref{5.2'}))} \\
			&=&\pounds_{A}\int \tau _{a}u_{0}(x,\omega ,y) g%
			(y)dy.
		\end{eqnarray*}%
		Note that here we have identified $u_{0}(x,\omega ,\cdot )\in \mathcal{X}%
		_{A}^{\Phi}$ with its representative still denoted by $u_{0}(x,\omega ,\cdot
		)\in \mathfrak{X}_{A}^{\Phi}$ so that $M_{1}(u_{0}(x,\omega ,\cdot ))=M(u_{0}(x,\omega
		,\cdot ))$, $u_{0}(x,\omega ,\cdot )$ on the left-hand side of the above
		equality being an equivalence class whereas $u_{0}(x,\omega ,\cdot )$ on the
		right-hand side is one of its representative. For $(I_{2}^{\prime })$, we
		have 
		\begin{equation*}
			\begin{array}{l}
				|(I'_{2})|  \leq  2 \|u_{\epsilon}\|_{\Phi,Q\times\Omega} \times \|\varphi\|_{\infty} \times \|g\|_{\infty} \times \\ \left(  \inf\left\{ \delta>0 :  \, \displaystyle \iint_{Q\times \Omega} \widetilde{\Phi}\left( \dfrac{ \left| f\left( T\left( \frac{%
						x}{\varepsilon _{1}}+\frac{\varepsilon _{2}}{\varepsilon _{1}}a\right)
					\omega \right) -f\left( T\left( \frac{x}{\varepsilon _{1}}\right) \omega
					\right) \right| }{\delta}  \right) dx d\mu \leq 1  \right\} \right).
			\end{array}
		\end{equation*}	
		But 
		\begin{equation*}
			\begin{array}{l}
				\inf\left\{ \delta>0 :  \, \displaystyle \iint_{Q\times \Omega} \widetilde{\Phi}\left( \dfrac{ \left| f\left( T\left( \frac{%
						x}{\varepsilon _{1}}+\frac{\varepsilon _{2}}{\varepsilon _{1}}a\right)
					\omega \right) -f\left( T\left( \frac{x}{\varepsilon _{1}}\right) \omega
					\right) \right| }{\delta}  \right) dx d\mu \leq 1  \right\}  \\
				=  \inf\left\{ \delta>0 :  \, \displaystyle \iint_{Q\times \Omega} \widetilde{\Phi}\left( \dfrac{ \left| \left(U\left(\frac{%
						x}{\varepsilon _{1}}+\frac{\varepsilon _{2}}{\varepsilon _{1}}a\right)f \right)(\omega) - \left(U\left(\frac{x}{\varepsilon_{1}}\right)f \right)(\omega) \right| }{\delta}  \right) dx d\mu \leq 1  \right\}. 
			\end{array}
		\end{equation*}
		Since the group $U(x)$ is strongly continuous in $L^{\widetilde{\Phi}}(\Omega)$ [see (\ref{lem38})], we get immediately that 
		\begin{equation*}
			\begin{array}{c}
				\inf\left\{ \delta>0 :  \, \displaystyle \iint_{Q\times \Omega} \widetilde{\Phi}\left( \dfrac{ \left|f\left( T\left( \frac{%
						x}{\varepsilon _{1}}+\frac{\varepsilon _{2}}{\varepsilon _{1}}a\right)
					\omega \right) -f\left( T\left( \frac{x}{\varepsilon _{1}}\right) \omega
					\right)  \right| }{\delta}  \right) dx d\mu \leq 1  \right\}  \rightarrow 0, \quad \textup{as} \; \epsilon \to 0.
			\end{array}
		\end{equation*}
		Thus $(I_{2}^{\prime })\rightarrow 0$ as $E\ni \varepsilon \rightarrow 0$.
		Finally since the sequence $(u_{\varepsilon })_{\varepsilon \in E}$ is
		bounded in $L^{\Phi}(Q\times \Omega )$, this sequence is uniformly
		integrable in $L^{1}(Q\times \Omega )$, so that from the inequality 
		\begin{eqnarray*}
			&&\int_{\left( (Q-\varepsilon _{2}a)\Delta Q\right) \times \Omega
			}\left\vert u_{\varepsilon }(x,\omega )\right\vert \left\vert \varphi
			(x+\varepsilon _{2}a)\right\vert \left\vert f\left( T\left( \frac{x}{%
				\varepsilon _{1}}+\frac{\varepsilon _{2}}{\varepsilon _{1}}a\right) \omega
			\right) \right\vert \left\vert g\left( \frac{x}{\varepsilon _{2}}+a\right)
			\right\vert dxd\mu \\
			&\leq &\left\Vert \varphi \right\Vert _{\infty }\left\Vert f\right\Vert
			_{\infty }\left\Vert g\right\Vert _{\infty }\int_{\left( (Q-\varepsilon
				_{2}a)\Delta Q\right) \times \Omega }\left\vert u_{\varepsilon }(x,\omega
			)\right\vert dxd\mu ,
		\end{eqnarray*}%
		we see that $(II)$ and $(III)$ go towards $0$ as $E\ni \varepsilon
		\rightarrow 0$; here the symbol $\Delta $ between the sets $(Q-\varepsilon
		_{2}a)$ and $Q$ denotes the symmetric difference between these two sets.
		Hence the lemma.
	\end{proof}
	
	We are now able to state and prove the most important compactness result of
	the paper. It will be of capital interest in the next sections.
	
	\begin{theorem}
		\label{t3.3} Let $\Phi \in \Delta_{2}$ be a $N$-function and $\widetilde{\Phi} \in \Delta_{2}$ its conjugate and $Q$ being an open subset of $\mathbb{\mathbb{R}}^{d}$. Let $A$ be an ergodic supralgebra on $\mathbb{\mathbb{R}}_{y}^{d}$ such that $A$ is translation invariant and  each of its elements is uniformly
		continuous.
		Assume $(u_{\varepsilon })_{\varepsilon \in E}$ is a sequence in $%
		L^{\Phi}(Q\times \Omega )$ such that:
		
		\begin{itemize}
			
			\item[(i)] $(u_{\varepsilon })_{\varepsilon \in E}$ is bounded in $L^{\Phi}\left(Q\times\Omega\right)$ and $(D_{x}u_{\epsilon})_{\epsilon\in E}$ is bounded in $L^{\Phi}\left(Q\times\Omega\right)^{d}$.
		\end{itemize}
		
		Then there exist $u_{0} \in W^{1}L^{\Phi}_{D_{x}}(Q; I_{nv}^{\Phi}(\Omega))$, $u_{1} \in L^{1}\left(Q; W^{1}_{\#}L^{\Phi}(\Omega)\right)$, $u_{2}\in L^{1}(Q\times\Omega; W^{1}_{\#}\mathcal{X}^{\Phi}_{A})$ and a subsequence $E'$ from $E$ such that, as $E' \ni \epsilon \rightarrow 0$, 
		
		\begin{itemize}
			
			\item[(ii)] $u_{\varepsilon }\rightarrow u_{0}$  in $L^{\Phi}(Q\times
			\Omega )$-weak $\digamma\,\Sigma$;
			
			\item[(iii)] $Du_{\varepsilon }\rightarrow Du_{0}+\overline{D}_{\omega }u_{1}+%
			\overline{D}_{y}u_{2}$  in $L^{\Phi}(Q\times \Omega )^{d}$-weak $\digamma\,\Sigma $,  with $u_{1} \in L^{\Phi}(Q; W^{1}_{\#}L^{\Phi}(\Omega))$, $u_{2}\in L^{\Phi}(Q\times\Omega; W^{1}_{\#}\mathcal{X}^{\Phi}_{A})$  when $\widetilde{\Phi} \in \Delta'$.  
		\end{itemize}
	\end{theorem}
	
	\begin{proof}
		By Theorem \ref{t3.1}, there exist a subsequence $E^{\prime }$ from $E$, a
		function $u_{0}\in L^{\Phi}(Q\times \Omega ;\mathcal{X}_{A}^{\Phi})$ and a vector
		function $\mathbf{v}=(v_{i})_{1\leq i\leq d}\in L^{\Phi}(Q\times \Omega ;%
		\mathcal{X}_{A}^{\Phi})^{d}$ such that, as $E^{\prime }\ni \varepsilon
		\rightarrow 0$, we have $u_{\varepsilon }\rightarrow u_{0}$  in $%
		L^{\Phi}(Q\times \Omega )$-weak $\digamma\,\Sigma $ and $Du_{\varepsilon }\rightarrow 
		\mathbf{v}$  in $L^{\Phi}(Q\times \Omega )^{d}$-weak $\digamma\,\Sigma $.
		
		\noindent At this level, the proof consists of two parts. We must check
		that:\ Part (\textbf{I}) ($a$) $u_{0}$ does not depend upon $y$, that is $%
		\overline{D}_{y}u_{0}=0$, ($b$) $u_{0}(x,\cdot )\in I_{nv}^{\Phi}(\Omega )$,
		that is $D_{\omega }u_{0}(x,\cdot )=0$ or equivalently $\int_{\Omega
		}u_{0}(x,\cdot )D_{i,p}\varphi d\mu =0\;\forall \varphi \in \mathcal{C}%
		^{\infty }(\Omega )$ and ($c$) $u_{0}\in W^{1}L^{\Phi}_{D_{x}}(Q; I_{nv}^{\Phi}(\Omega))$;
		Part (\textbf{II}) There exist two functions  $u_{1} \in L^{1}\left(Q; W^{1}_{\#}L^{\Phi}(\Omega)\right)$  and $u_{2}\in L^{1}(Q\times\Omega; W^{1}_{\#}\mathcal{X}^{\Phi}_{A})$  such that $\mathbf{v}=Du_{0}+\overline{D}_{\omega }u_{1}+%
		\overline{D}_{y}u_{2}$. The fact that $u_{1} \in L^{\Phi}(Q; W^{1}_{\#}L^{\Phi}(\Omega))$, $u_{2}\in L^{\Phi}(Q\times\Omega; W^{1}_{\#}\mathcal{X}^{\Phi}_{A})$  if $\widetilde{\Phi} \in \Delta'$   will follow from \cite[Theorem 12]{franck} and \cite[Remark 2]{tacha2}.
		
		Let us first prove (\textbf{I}). ($a$) Let $\Phi _{\varepsilon }(x,\omega
		)=\varepsilon _{2}\varphi (x)f(T(x/\varepsilon _{1})\omega )g(x/\varepsilon
		_{2})$ for $(x,\omega )\in Q\times \Omega $, where $\varphi \in \mathcal{C}%
		_{0}^{\infty }(Q)$, $f\in \mathcal{C}^{\infty }(\Omega )$ and $g\in
		A^{\infty }$. Then 
		\begin{eqnarray*}
			\int_{Q\times \Omega }\frac{\partial u_{\varepsilon }}{\partial x_{i}}\Phi
			_{\varepsilon }dxd\mu &=&-\int_{Q\times \Omega }\varepsilon
			_{2}u_{\varepsilon }f^{\varepsilon }g^{\varepsilon }\frac{\partial \varphi }{%
				\partial x_{i}}dxd\mu -\int_{Q\times \Omega }u_{\varepsilon }\varphi
			f^{\varepsilon }\left( D_{y_{i}}g\right) ^{\varepsilon }dxd\mu \\
			&&-\int_{Q\times \Omega }\frac{\varepsilon _{2}}{\varepsilon _{1}}%
			u_{\varepsilon }\varphi g^{\varepsilon }(D_{i,\omega }f)^{\varepsilon }dxd\mu
		\end{eqnarray*}%
		where $D_{y_{i}}g=\partial g/\partial y_{i}$. Letting $E^{\prime }\ni
		\varepsilon \rightarrow 0$ we get 
		\begin{equation*}
			\iint_{Q\times \Omega} \pounds_{A} \int u_{0}\varphi 
				D_{y_{i}}gf dydxd\mu  = 0\text{,}
		\end{equation*}%
		hence $\pounds_{A}\int u_{0}(x,\omega ,\cdot )D_{y_{i}}g dy = 0$ for all $g\in A^{\infty }$ and all $1\leq i\leq d$, which means
		that $u_{0}$ does not depend on $y$ since the $H$-supralgebra $A$ is ergodic.
		
		($b$) Let $\Phi _{\varepsilon }(x,\omega )=\varepsilon _{1}\varphi
		(x)f(T(x/\varepsilon _{1})\omega )$ for $(x,\omega )\in Q\times \Omega $
		where $\varphi \in \mathcal{C}_{0}^{\infty }(Q)$ and $f\in \mathcal{C}%
		^{\infty }(\Omega )$. Then proceeding as above we get $\int_{\Omega
		}u_{0}(x,\cdot )D_{i,\omega }fd\mu =0$ for all $1\leq i\leq d$ and $f\in 
		\mathcal{C}^{\infty }(\Omega )$, which is equivalent to say that $%
		u_{0}(x,\cdot )\in I_{nv}^{\Phi}(\Omega )$ for a.e. $x\in Q$.
		
		($c$)  Hypothesis (i) implies that the sequence $(u_{\epsilon})_{\epsilon\in E'}$ is bounded in $W^{1}L^{\Phi}_{D_{x}}(Q ; L^{\Phi}(\Omega))$, which yields the existence of a subsequence of $E'$ not relabeled and of a function $u \in W^{1}L^{\Phi}_{D_{x}}(Q ; L^{\Phi}(\Omega))$ such that $u_{\epsilon} \rightarrow u$ in $W^{1}L^{\Phi}_{D_{x}}(Q ; L^{\Phi}(\Omega))$-weak as $E' \ni \epsilon \rightarrow 0$. In particular $\displaystyle \int_{\Omega} u_{\epsilon}(\cdot,\omega)\psi(\omega)d\mu \rightarrow \int_{\Omega} u(\cdot,\omega)\psi(\omega)d\mu$ in $L^{1}(Q)$-weak for all $\psi \in I_{nv}^{\widetilde{\Phi}}(\Omega)$. Therefore, using \cite{Bourgeat} (see in particular [Lemma 3.6] therein, replacing the Lebesgue spaces by Orlicz spaces), we get at once $u_{0} \in W^{1}L^{\Phi}_{D_{x}}(Q ; L^{\Phi}(\Omega))$ so that $u_{0} \in W^{1}L^{\Phi}_{D_{x}}(Q ; I_{nv}^{\Phi}(\Omega))$.
		
		
		It	remains to check (\textbf{II}) here above.
		We begin by deriving the
		existence of $u_{2}\in L^{\Phi}(Q\times\Omega; W^{1}_{\#}\mathcal{X}^{\Phi}_{A})$. For
		that purpose, let $r>0$ be freely fixed. Let $B_{\varepsilon _{2}r}$ denote
		the open ball in $\mathbb{\mathbb{R}}^{d}$ centered at the origin and of
		radius $\varepsilon _{2}r$. By the equalities 
		\begin{eqnarray*}
			&&\frac{1}{\varepsilon _{2}}\left( u_{\varepsilon }(x,\omega )-\frac{1}{%
				\left\vert B_{\varepsilon _{2}r}\right\vert }\int_{B_{\varepsilon
					_{2}r}}u_{\varepsilon }(x+\rho ,\omega )d\rho \right) \\
			&=&\frac{1}{\varepsilon _{2}}\frac{1}{\left\vert B_{\varepsilon
					_{2}r}\right\vert }\int_{B_{\varepsilon _{2}r}}\left( u_{\varepsilon
			}(x,\omega )-u_{\varepsilon }(x+\rho ,\omega )\right) d\rho \\
			&=&\frac{1}{\varepsilon _{2}}\frac{1}{\left\vert B_{r}\right\vert }%
			\int_{B_{r}}\left( u_{\varepsilon }(x,\omega )-u_{\varepsilon
			}(x+\varepsilon _{2}\rho ,\omega )\right) d\rho \\
			&=&-\frac{1}{\left\vert B_{r}\right\vert }\int_{B_{r}}d\rho
			\int_{0}^{1}Du_{\varepsilon }(x+t\varepsilon _{2}\rho ,\omega )\cdot \rho dt
		\end{eqnarray*}%
		where the dot denotes the usual Euclidean inner product in $\mathbb{\mathbb{R%
		}}^{d}$, we deduce from the boundedness of $(u_{\varepsilon })_{\varepsilon
			\in E^{\prime }}$ in $W^{1}L^{\Phi}_{D_{x}}(Q; L^{\Phi}(\Omega))$ and Jensen's inequality, that the sequence $%
		(z_{\varepsilon }^{r})_{\varepsilon \in E^{\prime }}$ defined by 
		\begin{equation*}
			z_{\varepsilon }^{r}(x,\omega )=\frac{1}{\varepsilon _{2}}\left(
			u_{\varepsilon }(x,\omega )-\frac{1}{\left\vert B_{\varepsilon
					_{2}r}\right\vert }\int_{B_{\varepsilon _{2}r}}u_{\varepsilon }(x+\rho
			,\omega )d\rho \right) \;((x,\omega )\in Q\times \Omega ,\varepsilon \in
			E^{\prime })
		\end{equation*}%
		is bounded in $L^{\Phi}(Q\times \Omega )$. It is important to note that in
		general the function $z_{\varepsilon }^{r}$ is well defined since $%
		u_{\varepsilon }$ and $Du_{\varepsilon }$ can be naturally extended off $Q$
		as elements of $L^{\Phi}(\Omega ;L_{\text{loc}}^{\Phi}(\mathbb{R}^{d}))$ and $%
		L^{\Phi}(\Omega ;L_{\text{loc}}^{\Phi}(\mathbb{R}^{d})^{d})$, respectively. Once
		more, by virtue of Theorem \ref{t3.1} we find that there exist a subsequence
		from $E^{\prime }$ (not relabeled) and a function $z_{r}$ in $L^{\Phi}(Q\times
		\Omega ;\mathcal{X}_{A}^{\Phi})$ such that, as $E^{\prime }\ni \varepsilon
		\rightarrow 0$ 
		\begin{equation}
			z_{\varepsilon }^{r}\rightarrow z_{r}\text{ in }L^{\Phi}(Q\times \Omega
			)\text{-weak }\digamma\,\Sigma \text{.}  \label{5.8'}
		\end{equation}%
		As $(z_{\varepsilon }^{r})_{\varepsilon \in E^{\prime }}$ is bounded in $%
		L^{\Phi}(Q\times \Omega )$ we have (since $\varepsilon _{2}$, $\varepsilon
		_{2}/\varepsilon _{1}\rightarrow 0$ as $E^{\prime }\ni \varepsilon
		\rightarrow 0$) that 
		\begin{equation}
			\varepsilon _{2}z_{\varepsilon }^{r}\rightarrow 0\text{\ in }L^{\Phi}(Q\times
			\Omega )\text{ and }\frac{\varepsilon _{2}}{\varepsilon _{1}}z_{\varepsilon
			}^{r}\rightarrow 0\text{ in }L^{\Phi}(Q\times \Omega )\text{ as }E^{\prime }\ni
			\varepsilon \rightarrow 0.  \label{5.9'}
		\end{equation}%
		Now, for $\varphi \in \mathcal{C}_{0}^{\infty }(Q)$, $f\in \mathcal{C}%
		^{\infty }(\Omega )$ and $g\in A^{\infty }$ we have 
		\begin{equation}
			\begin{array}{l}
				\displaystyle 	\int_{Q\times \Omega }\left( \frac{\partial u_{\varepsilon }}{\partial x_{i}}%
				(x,\omega )-\frac{1}{\left\vert B_{\varepsilon _{2}r}\right\vert }%
				\int_{B_{\varepsilon _{2}r}}\frac{\partial u_{\varepsilon }}{\partial x_{i}}%
				(x+\rho ,\omega )d\rho \right) \varphi (x)f\left( T\left( \frac{x}{%
					\varepsilon _{1}}\right) \omega \right) g\left(\frac{x}{\varepsilon _{2}}\right)dxd\mu
				\\ 
				\displaystyle 	\;\;=-\int_{Q\times \Omega }\varepsilon _{2}z_{\varepsilon }^{r}(x,\omega
				)f\left( T\left( \frac{x}{\varepsilon _{1}}\right) \omega \right) g\left(\frac{x}{%
					\varepsilon _{2}}\right)\frac{\partial \varphi }{\partial x_{i}}(x)dxd\mu \\ 
			\displaystyle 	 \;\;\;\;\;	-\int_{Q\times \Omega }\frac{\varepsilon _{2}}{\varepsilon _{1}}%
				z_{\varepsilon }^{r}(x,\omega )\varphi (x)g(\frac{x}{\varepsilon _{2}}%
				)\left( D_{i,\omega }f\right) (T\left( \frac{x}{\varepsilon _{1}}\right)
				\omega )dxd\mu \\ 
			\displaystyle 	\;\;\;\;\;-\int_{Q\times \Omega }z_{\varepsilon }^{r}(x,\omega )\varphi
				(x)f\left( T\left( \frac{x}{\varepsilon _{1}}\right) \omega \right) \frac{%
					\partial g}{\partial y_{i}}\left( \frac{x}{\varepsilon _{2}}\right) dxd\mu .%
			\end{array}
			\label{5.10'}
		\end{equation}%
		Passing to the limit in (\ref{5.10'}) (as $E^{\prime }\ni \varepsilon
		\rightarrow 0$) using conjointly (\ref{5.8'}), (\ref{5.9'}) and Remark \ref%
		{r5.2} (see (\ref{5.5'}) therein) one gets 
		\begin{equation*}
			\begin{array}{l}
				\displaystyle 	\iint_{Q\times \Omega} \pounds_{A}\int \left( v_{i}(x,\omega
				,\cdot )-\frac{1}{\left\vert B_{r}\right\vert }\int_{B_{r}}v_{i}(x,\omega
				,\cdot +\rho )d\rho \right) (s)\varphi (x)f(\omega ) g dydxd\mu
				 \\ 
				\displaystyle 	\;\;\;=-\iint_{Q\times \Omega} \pounds_{A}\int z_{r}(x,\omega
				,y)\varphi (x)f(\omega )\partial _{i} g(y) dydxd\mu  ,%
			\end{array}%
		\end{equation*}%
		the derivative $\partial _{i}$ in front of $g$ being the partial
		derivative of index $i$ with respect to $A$ 
		(see also (\ref{2.7'}) therein). Therefore, because of
		the arbitrariness of $\varphi $, $f$ and $g$, we are led to  
		\begin{equation*}
			\frac{\overline{\partial }z_{r}}{\partial y_{i}}(x,\omega ,\cdot
			)=v_{i}(x,\omega ,\cdot )-\frac{1}{\left\vert B_{r}\right\vert }%
			\int_{B_{r}}v_{i}(x,\omega ,\cdot +\rho )d\rho \text{\ a.e. in }\mathbb{R}%
			_{y}^{d}\text{\ for }(x,\omega )\in Q\times \Omega
		\end{equation*}%
		$1\leq i\leq d$. Set $f_{r}(x,\omega
		,y)=z_{r}(x,\omega ,y)-M_{y}(z_{r}(x,\omega ,\cdot ))$ where here, $%
		z_{r}(x,\omega ,\cdot )\in \mathcal{X}_{A}^{\Phi}$ is viewed as its
		representative in $\mathfrak{X}_{A}^{\Phi}$ and $M_{y}=M$ standing here for the mean value
		with respect to $y$  (see in particular
		property (\textbf{2}) and equality (\ref{m1m}) in Subsection \ref{labelsub2sub3sect2}). Then $%
		M_{y}(f_{r})=0$ and moreover $\overline{D}_{y}f_{r}=\overline{D}_{y}z_{r}$
		so that $f_{r}(x,\omega,\cdot)\in \mathcal{X}_{A}^{\Phi}$ with $%
		\overline{\partial }f_{r}(x,\omega,\cdot)/\partial y_{i}\in \mathcal{X}_{A}^{\Phi}$ for a.e. $(x,\omega) \in Q\times\Omega$, that is, 
		\begin{equation*}
			f_{r}(x,\omega,\cdot)\in W^{1}\mathcal{X}_{A}^{\Phi}/\mathbb{C}\text{.}%
			\;\;\;\;\;\;\;\;\;\;\;\;\;\;\;\;\;\;\;\;
		\end{equation*}%
		So let $g_{r}=J_{1}\circ f_{r}$, where $J_{1}$ denotes the canonical mapping
		of $W^{1}\mathcal{X}_{A}^{\Phi}/\mathbb{C}$ into its separated completion $%
		W^{1}_{\#}\mathcal{X}_{A}^{\Phi}$. Then $g_{r}(x,\omega,\cdot)\in W^{1}_{\#}\mathcal{X}_{A}^{\Phi}$ for a.e. $(x,\omega) \in Q\times\Omega$
		and moreover 
		\begin{equation*}
			\frac{\overline{\partial }g_{r}}{\partial y_{i}}(x,\omega ,\cdot
			)=v_{i}(x,\omega ,\cdot )-\frac{1}{\left\vert B_{r}\right\vert }%
			\int_{B_{r}}v_{i}(x,\omega ,\cdot +\rho )d\rho \;\ \ (1\leq i\leq d)
		\end{equation*}%
		since $\frac{\overline{\partial }g_{r}}{\partial y_{i}}(x,\omega ,\cdot )=%
		\frac{\overline{\partial }f_{r}}{\partial y_{i}}(x,\omega ,\cdot )=\frac{%
			\overline{\partial }z_{r}}{\partial y_{i}}(x,\omega ,\cdot )$. Now, we also
		view $v_{i}(x,\omega ,\cdot )$ as its representative in $\mathfrak{X}_{A}^{\Phi}$. Taking
		this into account, we have 
		\begin{equation}
			\begin{array}{l}
				\left\Vert g_{r}(x,\omega ,\cdot )-g_{r^{\prime }}(x,\omega ,\cdot
				)\right\Vert _{W^{1}_{\#}\mathcal{X}_{A}^{\Phi}} \\ 
				\leq \left\Vert \overline{D}_{y}g_{r}(x,\omega ,\cdot )-\mathbf{v}(x,\omega
				,\cdot )+M_{y}(\mathbf{v}(x,\omega ,\cdot ))\right\Vert _{\Phi,A} \\ 
				+\left\Vert \overline{D}_{y}g_{r^{\prime }}(x,\omega ,\cdot )-\mathbf{v}%
				(x,\omega ,\cdot )+M_{y}(\mathbf{v}(x,\omega ,\cdot ))\right\Vert _{\Phi,A}.%
			\end{array}
			\label{5.11'}
		\end{equation}%
		But 
		\begin{equation*}
			\begin{array}{l}
				\left\Vert \overline{D}_{y}g_{r}(x,\omega ,\cdot )-\mathbf{v}(x,\omega
				,\cdot )+M_{y}(\mathbf{v}(x,\omega ,\cdot ))\right\Vert _{\Phi,A} \\ 
				\;\;\;=\left\Vert \frac{1}{\left\vert B_{r}\right\vert }\int_{B_{r}}\mathbf{v%
				}(x,\omega ,\cdot +\rho )d\rho -M_{y}(\mathbf{v}(x,\omega ,\cdot
				))\right\Vert _{\Phi,A}.%
			\end{array}%
		\end{equation*}%
		Therefore, since the algebra $A$ is ergodic, the right-hand side (and hence
		the left-hand side) of (\ref{5.11'}) goes to zero when $r,r^{\prime
		}\rightarrow +\infty $. Thus, the sequence $(g_{r}(x,\omega ,\cdot ))_{r>0}$
		is a Cauchy sequence in the Banach space $W^{1}_{\#}\mathcal{X}_{A}^{\Phi}$, whence
		the existence of a unique $u_{2}(x,\omega ,\cdot )\in W^{1}_{\#}\mathcal{X}_{A}^{\Phi} $ such that 
		\begin{equation*}
			g_{r}(x,\omega ,\cdot )\rightarrow u_{2}(x,\omega ,\cdot )\text{\ in }%
			W^{1}_{\#}\mathcal{X}_{A}^{\Phi} \text{\ as }r\rightarrow +\infty ,
		\end{equation*}%
		that is 
		\begin{equation*}
			\overline{D}_{y}g_{r}(x,\omega ,\cdot )\rightarrow \overline{D}%
			_{y}u_{2}(x,\omega ,\cdot )\text{\ in }(\mathcal{X}_{A}^{\Phi})^{d}\text{\ as }%
			r\rightarrow +\infty .
		\end{equation*}%
		Once again the ergodicity of $A$ and the uniqueness of the limit leads at
		once to 
		\begin{equation*}
			\overline{D}_{y}u_{2}(x,\omega ,\cdot )=\mathbf{v}(x,\omega ,\cdot )-M_{y}(%
			\mathbf{v}(x,\omega ,\cdot ))\text{\ a.e. in }\mathbb{R}^{d}\text{\ and for
				a.e. }(x,\omega )\in Q\times \Omega \text{.}
		\end{equation*}%
		We deduce the existence of a function $u_{2}:Q\times \Omega \rightarrow 
		W^{1}_{\#}\mathcal{X}_{A}^{\Phi}$, $(x,\omega )\mapsto u_{2}(x,\omega ,\cdot )$, lying in 
		$L^{1}(Q\times \Omega ;W^{1}_{\#}\mathcal{X}_{A}^{\Phi})$ such that 
		\begin{equation}
			\mathbf{v}-M(\mathbf{v})=\overline{D}_{y}u_{2}.\;\;\;\;  \label{Equ1}
		\end{equation}%
		
		Let us finally derive the existence of $u_{1}$. Let $\Phi _{\varepsilon
		}(x,\omega )=\varphi (x)\Psi (T(x/\varepsilon _{1})\omega )$ ($(x,\omega
		)\in Q\times \Omega )$) with $\varphi \in \mathcal{C}_{0}^{\infty }(Q)$ and $%
		\Psi =(\psi _{j})_{1\leq j\leq N}\in \mathcal{V}^{\widetilde{\Phi}}_{\text{div}}$ (i.e. $\text{%
			div}_{\omega ,\widetilde{\Phi}}\Psi =0$). Clearly 
		\begin{equation*}
			\sum_{j=1}^{d}\int_{Q\times \Omega }\frac{\partial u_{\varepsilon }}{%
				\partial x_{j}}\varphi \psi _{j}^{\varepsilon }dxd\mu
			=-\sum_{j=1}^{d}\int_{Q\times \Omega }u_{\varepsilon }\psi _{j}^{\varepsilon
			}\frac{\partial \varphi }{\partial x_{j}}dxd\mu
		\end{equation*}%
		where $\psi _{j}^{\varepsilon }(x,\omega )=\psi _{j}(T(x/\varepsilon
		_{1})\omega )$. Passing to the limit when $E^{\prime }\ni \varepsilon
		\rightarrow 0$ yields 
		\begin{equation*}
			\sum_{j=1}^{d}\iint_{Q\times \Omega} \pounds_{A} \int v_{j}\varphi
			\psi _{j} dydxd\mu  =-\sum_{j=1}^{d}\iint_{Q\times \Omega} \pounds_{A}\int u_{0}\psi _{j}\frac{\partial \varphi }{\partial x_{j}} dydxd\mu 
			\text{,}
		\end{equation*}%
		or equivalently, 
		\begin{equation*}
			\iint_{Q\times \Omega} \pounds_{A} \int \left( v
			-Du_{0}\right) \cdot \Psi \varphi  dydxd\mu =0\text{,}
		\end{equation*}%
		and so, as $\varphi $ is arbitrarily fixed in $\mathcal{C}_{0}^{\infty }(Q)$%
		, 
		\begin{equation*}
			\int_{\Omega} \pounds_{A} \int \left( \mathbf{v}(x,\omega
			,y)-Du_{0}(x,\omega )\right) \cdot \Psi (\omega )d\mu dy =0\;\;\forall
			\Psi \in \mathcal{V}^{\widetilde{\Phi}}_{\text{div}}.
		\end{equation*}%
		This is also equivalent to 
		\begin{equation*}
			\int_{\Omega }\left( M(\mathbf{v})-Du_{0}\right) \cdot \Psi d\mu =0\;\;\text{%
				for all }\Psi \in \mathcal{V}^{\widetilde{\Phi}}_{\text{div}}\text{.}
		\end{equation*}%
		Therefore, the Lemma \ref{lem5} provides us with a function $%
		u_{1}(x,\cdot )\in W^{1}_{\#}L^{\Phi}(\Omega)$ such that 
		\begin{equation}
			M(\mathbf{v})-Du_{0}=\overline{D}_{\omega }u_{1}(x,\cdot ).  \label{Equ2}
		\end{equation}%
		Putting (\ref{Equ1}) and (\ref{Equ2}) together leads at once at $\mathbf{v}%
		=Du_{0}+\overline{D}_{\omega }u_{1}+\overline{D}_{y}u_{2}$, where the
		function $u_{1}:x\mapsto u_{1}(x,\cdot )$ lies in $L^{1}\left(Q; W^{1}_{\#}L^{\Phi}(\Omega)\right)$ (see, e.g., \cite[Theorem 12]{franck}). This completes the proof.
	\end{proof}
	
	\begin{remark}\label{remX}
		Under the hypothesis of Theorem \ref{t3.3}, if the sequence $(u_{\varepsilon})_{\varepsilon\in E} \subset  L^{\Phi}(Q\times\Omega)$ is such that $u_{\varepsilon}(\cdot,\omega) \in X$ for all $\varepsilon\in E$ and for $\mu$-a.e. $\omega\in \Omega$, where $X$ is a norm closed convex subset of $W^{1}L^{\Phi}(Q)$, then as in \cite[Theorem 3.7(b)]{Bourgeat} the stochastic $\Sigma$-limit $u_{0}$ is such that $u_{0}(\cdot,\omega) \in X$ for $\mu$-a.e. $\omega\in \Omega$.
	\end{remark}
	
	
	\begin{corollary}
		Let the hypothesis be as in Theorem \ref{t3.3}. By Remark \ref{remX}, it follows that $A$ is stochastic $\Phi$-pseudoproper.
	\end{corollary}
	
	In the next section, we apply the results obtained in this section (in particular Theorem \ref{t3.3})  to the homogenization of problems (\ref{ras2})-(\ref{ras1}). 
	
	\section{Application to the homogenization of an integral functional with convex and nonstandard growth random integrand} \label{labelsect4} 
	
	\subsection{Setting of the problem}
	
	Let $(\Omega, \mathscr{M}, \mu)$ be a measure space with probability measure $\mu$ and $\{ T(x): \Omega\to \Omega, \, x \in \mathbb{R}^{d} \}$ be a fixed $d$-dimensional dynamical system on $\Omega$ which is invariant for $T$. Let $\Phi \in \Delta_{2}$ be a $N$-function and $\widetilde{\Phi} \in \Delta_{2}$ its complementary and let $A$ be an ergodic supralgebra on $\mathbb{\mathbb{R}}_{y}^{d}$ such that $A$ is translation invariant and  each of its elements is uniformly
	continuous. \\
	Let $(x,\omega,y,\lambda) \rightarrow f(x,\omega,y,\lambda)$ be a function from $\mathbb{R}^{d}\times\Omega\times\mathbb{R}^{d}\times \mathbb{R}^{d}$ into $\mathbb{R}$ defined as in Section \ref{labelintro} and satisfying hypotheses \textbf{(H$_{1}$)}-\textbf{(H$_{4}$)} of Subsection \ref{hypoproblem}.
	Let $Q$ be a bounded open set in $\mathbb{R}^{d}_{x}$ and let $\varepsilon _{1}$, $%
	\varepsilon _{2}$ be two well separated functions of $\varepsilon $ tending
	towards zero with $\varepsilon $.
	
	We consider the minimization problem
	\begin{equation}\label{lem26}
		\min \left\{ F_{\varepsilon}(v) : v \in W^{1}_{0}L^{\Phi}_{D_{x}}(Q; L^{\Phi}(\Omega)) \right\},
	\end{equation}
	where $F_{\varepsilon}$ is the integral functional defined by 
	\begin{eqnarray}\label{lem13}
		F_{\varepsilon}(v) = \iint_{Q\times \Omega} f\left(x, T\left(\frac{x}{\varepsilon_{1}}\right)\omega, \frac{x}{\varepsilon_{2}}, \, Dv(x,\omega) \right)dxd\mu, \\
		 \quad v \in W^{1}_{0}L^{\Phi}_{D_{x}}(Q; L^{\Phi}(\Omega)). \nonumber
	\end{eqnarray} 
	
		Note firstly that, hypothesis \textbf{(H$_{1}$)}-\textbf{(H$_{2}$)} yields that (see, \cite[Corollary 4.1]{tchin2} and \cite[Theorem 4.1]{tchin2}) problem \eqref{lem13} makes sense and there exists (for each $\varepsilon>0$) a unique $u_{\varepsilon} \in W^{1}_{0}L^{\Phi}_{D_{x}}(Q; L^{\Phi}(\Omega))$ that realizes the infimum of $F_{\varepsilon}$ on $W^{1}_{0}L^{\Phi}_{D_{x}}(Q; L^{\Phi}(\Omega))$, i.e., 
	\begin{equation*}
		F_{\varepsilon}(u_{\epsilon}) = \min \left\{ F_{\varepsilon}(v) : v \in W^{1}_{0}L^{\Phi}_{D_{x}}(Q; L^{\Phi}(\Omega)) \right\}.
	\end{equation*}
	
	Our objective in this section amounts to find, under the assumptions \textbf{(H$_{1}$)}-\textbf{(H$_{4}$)} a homogenized functional $F$ such that the sequence of minimizers $u_{\varepsilon}$ converges to a limit $\mathbf{u}$, which is precisely the minimizer of $F$. 
	
	
	
	\begin{remark}\label{remcase}
		Taking into account Proposition \ref{t2.2}, if $A$ is translation invariant and each of their elements are uniformly continuous then stochastic-deterministic homogenization problem \eqref{lem13} can be reduces to a entirely deterministic reiterated homogenization problem in the framework of Orlicz-Sobolev spaces (see \cite{tchin3}). 
	\end{remark}
	
		\begin{remark}
		Suppose that $\Phi(t) = \frac{t^{p}}{p}$ ($p>1, \;t\geq 0$), then $\Phi \in \Delta_{2}\cap\Delta'$ (with $\widetilde{\Phi}(t)=\frac{t^{q}}{q}$, $q=\frac{p}{p-1}$) and one has  $W^{1}_{0}L^{\Phi}(Q; L^{\Phi}(\Omega))\equiv W^{1,p}_{0}(Q; L^{p}(\Omega))=L^{p}(\Omega; W^{1,p}_{0}(Q))$.  Therefore, stochastic-deterministic homogenization problem (\ref{lem26}) can be rewritten in the classical Sobolev's spaces as
			\begin{equation}\label{minclass}
				\min \left\{ \int_{Q\times\Omega} f\left(x,T\left(\frac{x}{\varepsilon_{1}}\right)\omega,\frac{x}{\varepsilon_{2}}, Dv(x,\omega) \right) dx \; : \; v \in  L^{p}(\Omega; W^{1,p}_{0}(Q)) \right\}.
			\end{equation}
		Moreover,  if $A$ is translation invariant and each of their elements are uniformly continuous, then taking into account Proposition \ref{t2.2}, homogenization problem (\ref{minclass}) becomes a detreministic reiterated homogenization problem and whose a particular case can be found in \cite{wou}.
	\end{remark}
	
	\subsection{Preliminary results} 
	
	Let $\mathbf{v}  \in  \mathcal{C}(\overline{Q} ; L^{\infty}(\Omega; \mathcal{B}(\mathbb{R}^{d}_{y})^{d}))$. As in \cite[Sect. 5]{sango2}, the function  $(x, \omega) \rightarrow f\left( x,T\left(\frac{x}{\varepsilon_{1}}\right)\omega,\frac{x}{\varepsilon_{2}}, \mathbf{v}\left( x,T\left(\frac{x}{\varepsilon_{1}}\right)\omega,\frac{x}{\varepsilon_{2}}\right)\right)$ of $Q\times\Omega$ into $\mathbb{C}$, denoted by $f^{\epsilon}(\cdot, \cdot, \mathbf{v}^{\epsilon})$ is well defined as an element of $L^{\infty}(Q\times \Omega)$. Moreover for any  $\mathbf{v}  \in  [\mathcal{C}^{\infty}_{0}(Q)\otimes \mathcal{C}^{\infty}(\Omega)\otimes A]^{d}$, the function  $f(\cdot,\cdot, \mathbf{v}) : (x,\omega,y) \mapsto f(x,\omega,y, \mathbf{v}(x,\omega,y))$ lies in $\mathcal{C}(\overline{Q}; L^{\infty}(\Omega; A))$.
	
	We can now state and prove the first preliminary result.
	
	\begin{proposition}\label{lem23}
		Suppose that \textbf{(H$_{1}$)}-\textbf{(H$_{4}$)} hold. For every $\mathbf{v}  \in  [\mathcal{C}^{\infty}_{0}(Q)\otimes \mathcal{C}^{\infty}(\Omega)\otimes A]^{d}$ we have the following convergence result: 
		\begin{equation}\label{lem17}
			f^{\varepsilon }(\cdot , \cdot, \mathbf{v} ^{\varepsilon }) \longrightarrow \varrho \circ
			f(\cdot , \cdot, \mathbf{v} )\text{ in }L^{\Phi}(Q\times \Omega )\text{%
				-weak } \digamma\,\Sigma \text{ as }\varepsilon \rightarrow 0\text{ },
		\end{equation}%
		$\varrho $ being the canonical mapping from $\mathfrak{X}_{A}^{\Phi}$ into $%
		\mathcal{X}_{A}^{\Phi}$ and $(\varrho \circ f(\cdot , \cdot, \mathbf{v}
		))(x,\omega ,y)=\varrho (f(x,\omega ,\cdot ,\mathbf{v} (x,\omega ,\cdot )))(y)$
		for $(x,\omega ,y)\in Q\times \Omega \times \mathbb{R}^{d}$.
		
		Furthermore, the mapping  $\mathbf{v} \rightarrow f(\cdot,\cdot, \mathbf{v})$ of $[\mathcal{C}^{\infty}_{0}(Q)\otimes \mathcal{C}^{\infty}(\Omega)\otimes A]^{d}$ into $L^{1}(Q\times \Omega ; \mathfrak{X}^{\Phi}_{A})$ extends by continuity to a mapping, still denoted by $\mathbf{v} \rightarrow f(\cdot,\cdot, \mathbf{v})$, of
		$L^{\Phi}(Q\times \Omega ; \mathfrak{X}^{\Phi}_{A})^{d}$ into $L^{1}(Q\times \Omega ; \mathfrak{X}^{\Phi}_{A})$ with the property 
		\begin{equation}\label{f3}
			\begin{array}{l}
				\parallel f(\cdot, \mathbf{v}) - f(\cdot, \mathbf{w}) \parallel_{L^{1}(Q\times \Omega ; \mathfrak{X}^{\Phi}_{A})} \\
				\leq c\, \left( \|1\|_{L^{\widetilde{\Phi}}(Q\times \Omega ; \mathfrak{X}^{\widetilde{\Phi}}_{A})} + \parallel \phi( 1+|\mathbf{v}|+ |\mathbf{w}|)\parallel_{L^{\widetilde{\Phi}}(Q\times \Omega ; \mathfrak{X}^{\widetilde{\Phi}}_{A})} \right)\parallel \mathbf{v}-\mathbf{w}  \parallel_{L^{\Phi}(Q\times \Omega ; \mathfrak{X}^{\Phi}_{A})^{d}}
			\end{array}
		\end{equation}
		for all $\mathbf{v}, \mathbf{w} \in L^{\Phi}(Q\times \Omega ; \mathfrak{X}^{\Phi}_{A})^{d}$.
	\end{proposition}
	
	\begin{proof}
		Let	 $\mathbf{v}  \in  [\mathcal{C}^{\infty}_{0}(Q)\otimes \mathcal{C}^{\infty}(\Omega)\otimes A]^{d}$. Since the function  $f(\cdot,\cdot, \mathbf{v}) : (x,\omega,y) \mapsto f(x,\omega,y, \mathbf{v}(x,\omega,y))$ lies in $\mathcal{C}(\overline{Q}; L^{\infty}(\Omega; A))$ the convergence result (\ref{lem17}) is a consequence of Corollary \ref{c3.3}. 
		
		On the other hand, since $\Phi \in \Delta_{2}$, using the \cite[Proposition 12]{tchin2}, there is a constant $c= c(c_{2},Q,\Omega,\Phi)$ such that (\ref{f3}) holds for all $\mathbf{v}, \mathbf{w} \in [\mathcal{C}^{\infty}_{0}(Q)\otimes \mathcal{C}^{\infty}(\Omega)\otimes A]^{d}$. We end the proof by routine argument of continuity and density.
	\end{proof}
	
	\begin{corollary}\label{lem24}
		Let 
		\begin{equation*}
			\phi_{\varepsilon}(x,\omega) = \psi_{0}(x,\omega) + \varepsilon_{1} \psi_{1}\left(x,T\left(\frac{x}{\varepsilon_{1}}\right)\omega\right) + \varepsilon_{2} \psi_{2}\left(x,T\left(\frac{x}{\varepsilon_{1}}\right)\omega,\frac{x}{\varepsilon_{2}}\right),
		\end{equation*}
		with $\varepsilon>0$, $(x,\omega) \in Q\times \Omega$, $\psi_{0} \in \mathcal{C}^{\infty}_{0}(Q)\otimes I^{\Phi}_{nv}(\Omega)$, $\psi_{1} \in \mathcal{C}^{\infty}_{0}(Q)\otimes \mathcal{C}^{\infty}(\Omega)$ and  $\psi_{2} \in \mathcal{C}^{\infty}_{0}(Q)\otimes \mathcal{C}^{\infty}(\Omega)\otimes A^{\infty}$. Then,
		\begin{equation*}
			\begin{array}{l}
				\displaystyle 	\lim_{\varepsilon \to 0} \iint_{Q\times\Omega} f^{\varepsilon}\left(\cdot,\cdot, D\phi_{\varepsilon}\right) dxd\mu  \\ 
				= \displaystyle \iint_{Q\times\Omega} \pounds_{A}\int \varrho \circ f\left(\cdot,\cdot, D_{x}\psi_{0}+ D_{\omega}\psi_{1} +\partial \psi_{2}\right) dydxd\mu,
			\end{array}
		\end{equation*}
		where $(\varrho \circ f\left(\cdot,\cdot, D_{x}\psi_{0}+ D_{\omega}\psi_{1} +\partial \psi_{2}\right))(x,\omega ,y)=\varrho (f(x,\omega ,\cdot ,D_{x}\psi_{0}+ D_{\omega}\psi_{1} +\partial \psi_{2}(x,\omega,\cdot )))(y)$
		for $(x,\omega ,y)\in Q\times \Omega \times \mathbb{R}^{d}$.
	\end{corollary}
	
	\begin{proof}
		As $D\psi_{1}\left(x,T\left(\frac{x}{\varepsilon_{1}}\right)\omega\right) = (D_{x}\psi_{1})^{\varepsilon} + \frac{1}{\varepsilon_{1}}(D_{\omega}\psi_{1})^{\varepsilon}$ and $D\psi_{2}\left(x,T\left(\frac{x}{\varepsilon_{1}}\right)\omega,\frac{x}{\varepsilon_{2}}\right) = (D_{x}\psi_{2})^{\varepsilon} + \frac{1}{\varepsilon_{1}}(D_{\omega}\psi_{2})^{\varepsilon} + \frac{1}{\varepsilon_{2}}(D_{y}\psi_{2})^{\varepsilon}$, [see (\ref{lem18})], we have
		\begin{equation}\label{tc1}
			D\phi_{\varepsilon}(x,\omega) = D_{x}\psi_{0}+ \varepsilon_{1} (D_{x}\psi_{1})^{\varepsilon} + (D_{\omega}\psi_{1})^{\varepsilon} + \varepsilon_{2}(D_{x}\psi_{2})^{\varepsilon} + \frac{\varepsilon_{2}}{\varepsilon_{1}} (D_{\omega}\psi_{2})^{\varepsilon} +  (D_{y}\psi_{2})^{\varepsilon}.
		\end{equation}
		Recalling that functions $D_{x}\psi_{0}, D_{x}\psi_{1}$, $ D_{\omega}\psi_{1}$ belongs to $\left[\mathcal{C}^{\infty}(\Omega)\otimes \mathcal{C}(\overline{Q}, \mathbb{R}) \right]^{d}$ and $D_{y}\psi_{2}$ belong to $\left[ \mathcal{C}^{\infty}(\Omega)\otimes\mathcal{C}(\overline{Q} ; A) \right]^{d}$, the function $f(\cdot,\cdot, D_{x}\psi_{0}+ D_{\omega}\psi_{1} + D_{y}\psi_{2}) $ belongs to  $\mathcal{C}(\overline{Q}; L^{\infty}(\Omega ; A))$, so that (\ref{lem17}) implies 
		\begin{eqnarray}\label{lem21}
			f^{\varepsilon}(\cdot,\cdot, D_{x}\psi_{0}+ (D_{\omega}\psi_{1})^{\varepsilon} + (D_{y}\psi_{2})^{\varepsilon}) \longrightarrow \varrho \circ
			f(\cdot , \cdot, D_{x}\psi_{0}+ D_{\omega}\psi_{1} + D_{y}\psi_{2} ) \nonumber \\ 
			\text{\  in }L^{\Phi}(Q\times \Omega )\text{%
				-weak } \digamma\,\Sigma \text{ as }\varepsilon \rightarrow 0.\text{ }
		\end{eqnarray}
		On the other hand, since (\ref{tc1}) holds, it follows (see, \cite[Subsection 4.2]{tchin2}) 
		\begin{equation*}
			|f^{\varepsilon}(\cdot,\cdot, D\phi_{\varepsilon}) - f^{\varepsilon}(\cdot,\cdot, D_{x}\psi_{0}+ (D_{\omega}\psi_{1})^{\varepsilon} + (D_{y}\psi_{2})^{\varepsilon}) | \leq c\varepsilon \quad \textup{in} \; Q\times\Omega, \; \varepsilon >0,
		\end{equation*} 
		where $c = c\left( \|D_{x}\psi_{1}\|_{\infty}, \phi(\|D_{x}\psi_{0}\|_{\infty}), \phi(\|D_{\omega}\psi_{1}\|_{\infty}), \phi(\|D_{y}\psi_{2}\|_{\infty})  \right) >0$. Hence 
		\begin{equation}\label{lem20}
			f^{\varepsilon}(\cdot,\cdot, D\phi_{\varepsilon}) - f^{\varepsilon}(\cdot,\cdot, D_{x}\psi_{0}+ (D_{\omega}\psi_{1})^{\varepsilon} + (D_{y}\psi_{2})^{\varepsilon}) \rightarrow 0 \quad \textup{in} \; L^{1}(Q\times \Omega) \; \textup{as} \; \varepsilon \to 0.
		\end{equation}
		The proof is completed by combining (\ref{lem21})-(\ref{lem20}) with the decomposition 
		\begin{equation*}
			\begin{array}{l}
				\displaystyle 	\iint_{Q\times\Omega}  f^{\varepsilon}(\cdot,\cdot, D\phi_{\varepsilon})dxd\mu - 	\iint_{Q\times\Omega} \pounds_{A}\int  \varrho \circ f\left(\cdot,\cdot, D_{x}\psi_{0}+ D_{\omega}\psi_{1} +\partial \psi_{2}\right) dydxd\mu \\
				= \displaystyle \iint_{Q\times\Omega} \left[  f^{\varepsilon}(\cdot,\cdot, D\phi_{\varepsilon}) -  f^{\varepsilon}(\cdot,\cdot, D_{x}\psi_{0}+ (D_{\omega}\psi_{1})^{\varepsilon} + (D_{y}\psi_{2})^{\varepsilon}) 
				\right] dxd\mu \\
				\;\,	 + \displaystyle \iint_{Q\times\Omega}  f^{\varepsilon}(\cdot,\cdot, D_{x}\psi_{0}+ (D_{\omega}\psi_{1})^{\varepsilon} + (D_{y}\psi_{2})^{\varepsilon}) dxd\mu \\
				\;\,	  - \displaystyle \iint_{Q\times\Omega} \pounds_{A}\int  \varrho \circ f\left(\cdot,\cdot, D_{x}\psi_{0}+ D_{\omega}\psi_{1} +\partial \psi_{2}\right) dydxd\mu. 
			\end{array}
		\end{equation*}
	\end{proof}

	
	Assume that   $\widetilde{\Phi}$ is of class $\Delta_{2}\cap \Delta'$. We put
	\begin{equation}\label{banac1}
		\mathbb{F}_{0}^{1}L^{\Phi} = W^{1}_{0}L^{\Phi}_{D_{x}}(Q; I_{nv}^{\Phi}(\Omega))\times L^{\Phi}(Q; W^{1}_{\#}L^{\Phi}(\Omega))\times L^{\Phi}(Q\times\Omega; W^{1}_{\#}\mathcal{X}^{\Phi}_{A}).
	\end{equation}
	We equip $\mathbb{F}_{0}^{1}L^{\Phi}$ with the norm 
	\begin{equation*}
		\parallel \mathbf{u} \parallel_{\mathbb{F}_{0}^{1}L^{\Phi}} = \parallel Du_{0} \parallel_{L^{\Phi}(Q\times\Omega)^{d}} + \parallel \overline{D}_{\omega}u_{1} \parallel_{L^{\Phi}(Q\times\Omega)^{d}} + \parallel \overline{D}_{y}u_{2} \parallel_{L^{\Phi}(Q\times\Omega; \mathcal{X}^{\Phi}_{A})},
	\end{equation*} 
	where $\mathbf{u} = (u_{0}, u_{1}, u_{2}) \in \mathbb{F}_{0}^{1}L^{\Phi}$.
	With this norm, $\mathbb{F}_{0}^{1}L^{\Phi}$ in (\ref{banac1}) is a Banach space admitting 
	\begin{equation}\label{banac2}
		F^{\infty}_{0} = \left[  \mathcal{C}^{\infty}_{0}(Q)\otimes I_{nv}^{\Phi}(\Omega) \right] \times \left[ \mathcal{C}^{\infty}_{0}(Q)\otimes I_{\Phi}(\mathcal{C}^{\infty}(\Omega)) \right] \times \left[  \mathcal{C}^{\infty}_{0}(Q)\otimes \mathcal{C}^{\infty}(\Omega)\otimes (J_{1}\circ \varrho )(A^{\infty }/%
		\mathbb{C}) \right]	
	\end{equation}
	as a dense subspace, where  $J_{1}$ (resp. $\varrho $, $I_{\Phi}$)
	denotes the canonical mapping of $W^{1}\mathcal{X}_{A}^{\Phi}/\mathbb{C}$ (resp. $%
	\mathfrak{X}_{A}^{\Phi}$, $\mathcal{C}^{\infty }(\Omega )$) into its separated completion $%
	W^{1}_{\#}\mathcal{X}^{\Phi}_{A}$ (resp. $\mathcal{X}_{A}^{\Phi}$, $W^{1}_{\#}L^{\Phi}(\Omega)$). \\
	
	Let now $\mathbf{v} = (v_{0}, v_{1}, v_{2}) \in \mathbb{F}_{0}^{1}L^{\Phi}$, set $\mathbb{D}\mathbf{v} = Dv_{0} + \overline{D}_{\omega}v_{1} + \overline{D}_{y}v_{2} \in L^{\Phi}(Q\times\Omega,\mathcal{X}_{A}^{\Phi})^{d}$.
We define the functional $F$ on $\mathbb{F}_{0}^{1}L^{\Phi}$ by 
	\begin{equation*}
		F(\mathbf{v}) = \iint_{Q\times\Omega} \pounds_{A} \int  \varrho \circ f(\cdot,\cdot, \mathbb{D}\mathbf{v})\, dydxd\mu,
	\end{equation*}
	where the function $\varrho \circ f$ here is defined as in Corollary \ref{lem24}. \\
	The hypotheses \textbf{(H$_{1}$)}-\textbf{(H$_{4}$)} drive to the following lemma.
	\begin{lemma}
		There exists a unique $\mathbf{u} = (u_{0}, u_{1}, u_{2}) \in \mathbb{F}_{0}^{1}L^{\Phi}$ [see (\ref{banac1})] such that 
		\begin{equation}\label{lem28}
			F(\mathbf{u}) = \inf \left\{  F(\mathbf{v})\, : \, \mathbf{v} \in \mathbb{F}_{0}^{1}L^{\Phi} \right\}.
		\end{equation}
	\end{lemma}
	
	\subsection{Regularization} 
	
	As in \cite{tacha3}, we regularize the integrand $f$ in order to get an approximating family of integrands $(f_{n})_{n\in\mathbb{N}^{\ast}}$ having in particular some properties \textbf{(H$_{1}$)}-\textbf{(H$_{4}$)} . Precisely, let $\theta_{n} \in \mathcal{C}^{\infty}_{0}(\mathbb{R}^{d})$ with $0\leq \theta_{n}$, $\textup{supp}\theta_{n} \subset \frac{1}{n}\overline{B_{d}}$ (where $\overline{B_{d}}$ denotes the closure of the open unit ball $B_{d}$ in $\mathbb{R}^{d}$) and $\int \theta_{n}(\eta)d\eta = 1$. Setting 
	\begin{equation*}
		f_{n}(x,\omega,y, \lambda) = \int \theta_{n}(\eta)f(x,\omega,y, \lambda-\eta)d\eta, \quad (x,\omega,y,\lambda) \in \mathbb{R}^{d}\times\Omega\times \mathbb{R}^{d}\times \mathbb{R}^{d}.
	\end{equation*}
	The main properties of this new integrand are the following:
	\begin{itemize}
		\item[$(H_{1})_{n}$] for all $(x,\lambda) \in \mathbb{R}^{d}\times \mathbb{R}^{d}$ and for almost all $(\omega,y)\in\Omega\times\mathbb{R}^{d}$, $f_{n}(x,\cdot,\cdot, \lambda)$ is measurable and $f_{n}(\cdot,\omega,\cdot,\cdot)$ is continuous. Moreover, there exist a continuous positive function $\varpi : \mathbb{R} \to \mathbb{R}_{+}$ with $\varpi(0)=0$, and a function $a \in L^{1}_{loc}(\mathbb{R}^{d}_{y})$ such that 
		\begin{equation*}
			|f_{n}(x,\omega,y,\lambda) - f_{n}(x',\omega,y,\lambda)| \leq \varpi(|x - x'|)[a(y) + f_{n}(x,\omega,y,\lambda)]
		\end{equation*}
		for all $x, x' \in \mathbb{R}^{d}$, $\lambda \in \mathbb{R}^{d}$ and for $d\mu\times dy$-almost all $(\omega,y) \in \Omega\times\mathbb{R}^{d}$ ;
		\item[$(H_{2})_{n}$] $f_{n}(x,\omega,y,\cdot)$ is strictly convex for almost all $\omega \in \Omega$ and for all $x,y \in \mathbb{R}^{d}$  ;
		\item[$(H_{3})_{n}$] There is constant $c_{5} > 0$ such that 
		\begin{equation*}
			f_{n}(x,\omega,y,\lambda) \leq c_{5}(1+ \Phi(|\lambda|))
		\end{equation*}
		for all $x,y,\lambda \in \mathbb{R}^{N}$ and for almost all $\omega \in \Omega$ ;
		\item[$(H_{4})_{n}$] $f_{n}(x,\omega,\cdot, \lambda) \in A$ for all $x,\lambda \in \mathbb{R}^{d}$ and for almost all $\omega \in \Omega$;
		\item[$(H_{5})_{n}$] $\dfrac{\partial f_{n}}{\partial \lambda}(x,\omega,y,\lambda)$ exists for all $x,y,\lambda \in \mathbb{R}^{d}$ and for almost all $\omega\in \Omega$, and there exists a constant $c_{6}=c_{6}(n) >0$ such that 
		\begin{equation*}
			\left| \dfrac{\partial f_{n}}{\partial \lambda}(x,\omega,y,\lambda) \right| \leq c_{6} (1+\phi(|\lambda|)).
		\end{equation*}
	\end{itemize}
	That being the case, we obtain the results in Proposition \ref{lem23} and in Corollary \ref{lem24}, where $f$ is replaced by $f_{n}$, and as in \cite[Subsection 4.3.]{franck} we have the following lemma.
	\begin{lemma}
		For every $\mathbf{v} \in L^{\Phi}(Q\times\Omega, \mathcal{X}_{A}^{\Phi})^{d}$, as $n\to \infty$, one has
		\begin{equation*}
			\varrho \circ f_{n}(\cdot,\cdot,\mathbf{v}) \rightarrow \varrho \circ f(\cdot,\cdot,\mathbf{v}) \quad \textup{in} \; L^{1}(Q\times\Omega; \mathcal{X}_{A}^{\Phi}),
		\end{equation*}
		where $\varrho \circ f_{n}$ and $\varrho \circ f$ are defined as in Corollary \ref{lem24}.
	\end{lemma} 
	
	We are now ready to give one of the most important results of this subsection.
	\begin{proposition}\label{conseq1}
		Let $(\mathbf{v}_{\varepsilon})_{\varepsilon}$ be a sequence in $L^{\Phi}(Q\times\Omega)^{d}$ which weakly $\Sigma$-converges (in each component) to $\mathbf{v} \in L^{\Phi}(Q\times\Omega; \mathcal{X}_{A}^{\Phi})^{d}$. Then, for any integer $n\geq 1$,  we have that there exists a constant $C'$ such that
		\begin{equation*}
			\iint_{Q\times\Omega} \pounds_{A}\int \varrho \circ f_{n}(\cdot,\cdot, \mathbf{v}) dydxd\mu   - \dfrac{C'}{n} \leq \lim_{\varepsilon \to 0} \iint_{Q\times\Omega}  f\left(x, T\left(\frac{x}{\varepsilon_{1}}\right)\omega,\frac{x}{\varepsilon_{2}}, \mathbf{v}_{\varepsilon}(x,\omega)\right)  dxd\mu.
		\end{equation*}
	\end{proposition}
	\begin{proof}
		Arguing as in the proof of \cite[Proposition 21]{franck} with the sequence $(\mathbf{v}_{l})_{l\in \mathbb{N}} \in \left[ \mathcal{C}^{\infty}_{0}(Q)\otimes\mathcal{C}^{\infty}(\Omega)\otimes (J_{1}\circ \varrho)(A^{\infty}) \right]^{d}$ such that $\mathbf{v}_{l} \rightarrow \mathbf{v}$ in  $L^{\Phi}(Q\times\Omega; \mathcal{X}_{A}^{\Phi})^{d}$ as $l\to \infty$. One get the result.
	\end{proof} 
	
	Letting $n\to\infty$ in Proposition \ref{conseq1}, and replacing $\mathbf{v}_{\varepsilon}$ by $Du_{\varepsilon}$, with $Du_{\varepsilon}$ stochastically weakly $\Sigma$-converges (componentwise) to
	$\mathbb{D}\mathbf{u} = Du_{0} + \overline{D}_{\omega}u_{1} + \overline{D}_{y}u_{2} \in L^{\Phi}(Q\times\Omega;\mathcal{X}_{A}^{\Phi})^{d}$, one obtains the following result.
	
	\begin{corollary}\label{lem32}
		Let $(u_{\varepsilon})_{\varepsilon\in E}$ be a sequence in $W^{1}L^{\Phi}_{D_{x}}(Q; L^{\Phi}(\Omega))$. Assume that $(Du_{\varepsilon})_{\varepsilon\in E}$ stochastically weakly $\Sigma$-converges componentwise to $\mathbb{D}\mathbf{u} = Du_{0} + \overline{D}_{\omega}u_{1} + \overline{D}_{y}u_{2} \in L^{\Phi}(Q\times\Omega;\mathcal{X}_{A}^{\Phi})^{d}$, where $\mathbf{u} = (u_{0}, u_{1}, u_{2}) \in \mathbb{F}^{1}_{0}L^{\Phi}$. Then 
		\begin{equation*}
			\iint_{Q\times\Omega} \pounds_{A}\int \varrho \circ f(\cdot,\cdot, \mathbb{D}\mathbf{u}(x,\omega,y))  dydxd\mu  \leq \lim_{\varepsilon \to 0} \iint_{Q\times\Omega}  f^{\varepsilon}\left(\cdot,\cdot, Du_{\varepsilon}\right)  dxd\mu.
		\end{equation*}
	\end{corollary}
	
	\subsection{Main homogenization result}
	
	Our objective in this section is to prove the following result.
	\begin{theorem}\label{lem37}
		For each $\varepsilon > 0$, let $(u_{\varepsilon})_{\varepsilon\in E} \in W^{1}_{0}L^{\Phi}_{D_{x}}(Q; L^{\Phi}(\Omega))$ be the unique solution of (\ref{lem26}). Then, as $\varepsilon \to 0$, 
		\begin{equation}\label{lem30}
			u_{\varepsilon} \rightarrow u_{0} \quad  \;\textup{in} \; L^{\Phi}(Q\times\Omega)-weak \,\digamma\,\Sigma,  
		\end{equation}
		and 
		\begin{equation}\label{lem31}
			Du_{\varepsilon} \rightarrow Du_{0} + \overline{D}_{\omega}u_{1} + \overline{D}_{y}u_{2} \quad  \;\textup{in} \; L^{\Phi}(Q\times\Omega)^{d}-weak \,\digamma\,\Sigma,  
		\end{equation}
		where $\mathbf{u} = (u_{0}, u_{1}, u_{2}) \in \mathbb{F}_{0}^{1}L^{\Phi}$ [see (\ref{banac1})] is the unique solution to the minimization problem (\ref{lem28}). 
	\end{theorem}
	\begin{proof}
		In view of the growth conditions in $(H_{3})$, the sequence $(u_{\varepsilon})_{\varepsilon>0}$ is bounded in $W^{1}_{0}L^{\Phi}_{D_{x}}(Q; L^{\Phi}(\Omega))$ and so the sequence $(f^{\varepsilon}(\cdot,\cdot, Du_{\varepsilon}))_{\varepsilon>0}$ is bounded in $L^{1}(Q\times\Omega)$. Thus, given an arbitrary fundamental sequence $E$, we get by Theorem \ref{t3.1} the existence of a subsequence $E'$ from $E$ and a triplet $\mathbf{u} = (u_{0}, u_{1}, u_{2}) \in \mathbb{F}_{0}^{1}L^{\Phi}$ such that (\ref{lem30})-(\ref{lem31}) hold when $E' \ni \varepsilon \to 0$. The sequence $(F_{\varepsilon}(u_{\varepsilon}))_{\varepsilon>0}$ consisting of real numbers being bounded, since $(u_{\varepsilon})_{\varepsilon>0}$ is bounded in $W^{1}_{0}L^{\Phi}_{D_{x}}(Q; L^{\Phi}(\Omega))$, there exists a subsequence from $E'$ not relabeled such that $\lim_{E' \ni\varepsilon \to 0} F_{\varepsilon}(u_{\epsilon})$ exists. 
		We still have to verify that $\mathbf{u} = (u_{0}, u_{1}, u_{2})$ solves (\ref{lem28}). In fact, if $\mathbf{u}$ solves this problem, then thanks to the uniqueness of the solution of (\ref{lem28}), the whole sequence $(u_{\varepsilon})_{\varepsilon>0}$ will verify (\ref{lem30}) and (\ref{lem31}) when $\varepsilon \to 0$. Thus, our only concern here is to check that $\mathbf{u}$ solves problem (\ref{lem28}). To this end, in view of Corollary \ref{lem32}, we have 
		\begin{equation}\label{lem34}
			\begin{array}{l}
				\displaystyle	\iint_{Q\times\Omega}\pounds_{A}\int \varrho \circ f(\cdot,\cdot, \mathbb{D}\mathbf{u}) dydxd\mu  \\
				\leq \lim_{E'\ni\varepsilon \to 0} \displaystyle \iint_{Q\times\Omega}  f^{\varepsilon}\left(\cdot,\cdot, Du_{\varepsilon}(x,\omega)\right)  dxd\mu.
			\end{array}
		\end{equation}
		On the other hand, let us establish an upper bound for 
		\begin{equation*}
			\iint_{Q\times\Omega}  f^{\varepsilon}\left(\cdot,\cdot, Du_{\varepsilon}(x,\omega)\right)  dxd\mu.
		\end{equation*}
		To do that, let $\phi = \left(\psi_{0}, I_{\Phi}(\psi_{1}), \psi_{2}\right) \in F^{\infty}_{0}$ [see (\ref{banac2})] with $\psi_{0} \in \mathcal{C}_{0}^{\infty}(Q)\otimes I_{nv}^{\Phi}(\Omega)$, $\psi_{1} \in \mathcal{C}_{0}^{\infty}(Q)\otimes \mathcal{C}^{\infty}(\Omega)$ and $\psi_{2} \in \mathcal{C}_{0}^{\infty}(Q)\otimes \mathcal{C}^{\infty}(\Omega)\times (J_{1}\circ\varrho)(A^{\infty}/\mathbb{C})$. 	Define $\phi_{\varepsilon}$ as in Corollary \ref{lem24}. Since $u_{\varepsilon}$ is the minimizer, one has 
		\begin{equation*}
			\iint_{Q\times\Omega}  f^{\varepsilon}\left(\cdot,\cdot, Du_{\varepsilon}(x,\omega)\right)  dxd\mu \leq \iint_{Q\times\Omega}  f^{\varepsilon}\left(\cdot,\cdot, D\phi_{\varepsilon}\right)  dxd\mu.
		\end{equation*}
		Thus, using Corollary \ref{lem24} we get 
		\begin{equation*}
			\begin{array}{l}
				\displaystyle	\lim_{E'\ni\varepsilon \to 0}  \displaystyle \iint_{Q\times\Omega}  f^{\varepsilon}\left(\cdot,\cdot, Du_{\varepsilon}(x,\omega)\right)  dxd\mu  \\
				\leq  \displaystyle \iint_{Q\times\Omega}\pounds_{A}\int  \varrho \circ f(\cdot,\cdot, D\psi_{0} + D_{\omega}\psi_{1} + \partial \psi_{2})  dydxd\mu,
			\end{array}
		\end{equation*}
		for any $\phi \in F^{\infty}_{0}$, and by density, for all $\phi \in \mathbb{F}_{0}^{1}L^{\Phi}$. From which we get 
		\begin{equation}\label{lem33}
			\lim_{E'\ni\varepsilon \to 0}  \iint_{Q\times\Omega}  f^{\varepsilon}\left(\cdot,\cdot, Du_{\varepsilon}(x,\omega)\right)  dxd\mu  \leq  \inf_{\mathbf{v} \in \mathbb{F}_{0}^{1}L^{\Phi}} \iint_{Q\times\Omega} \pounds_{A}\int \varrho \circ f(\cdot,\cdot, \mathbb{D}\mathbf{v})  dydxd\mu.
		\end{equation}
		Inequalities (\ref{lem34}) and (\ref{lem33}) yield 
		\begin{equation*}
			\iint_{Q\times\Omega}\pounds_{A}\int \varrho \circ f(\cdot,\cdot,  \mathbb{D}\mathbf{u}) dydxd\mu  = \inf_{\mathbf{v} \in \mathbb{F}_{0}^{1}L^{\Phi}} \iint_{Q\times\Omega} \pounds_{A}\int \varrho \circ f(\cdot,\cdot, \mathbb{D}\mathbf{v})  dydxd\mu 
		\end{equation*}
		i.e. (\ref{lem28}). The proof is complete.
	\end{proof}	
	
	\section{Concrete homogenization problems for (\ref{lem13})-(\ref{lem26})} \label{labelsect5} 
	
	In this section, we examine some concrete homogenization problem for (\ref{lem13})-(\ref{lem26}) using Theorem \ref{usetheo}.
	
	\subsection{The coupled stochastic-periodic homogenization problem}

	Here, we assume only that $A = \mathcal{C}_{\textup{per}}(Y)$ be the classical $H$-algebra of $Y$-periodic continuous complex functions on $\mathbb{R}_{y}^{d}$ where $Y= (0,1)^{d}$ (see Example \ref{p2.5}).
	For each fixed $(x,\omega
	,\lambda )\in Q\times \Omega \times \mathbb{R}^{d}$,
	the functions $y\mapsto
	f(x,\omega ,y,\lambda )$ in (\ref{lem13}) satisfy:
	\begin{equation*}
		f(x,\omega ,\cdot ,\lambda )\in \mathcal{C}_{\textup{per}}(Y) \text{ for any }%
		(x,\omega ,\lambda )\in \overline{Q}\times \Omega \times \mathbb{R}^{d};
	\end{equation*}
	(we also say that the function $y\mapsto
	f(x,\omega ,y,\lambda )$  is $Y$-periodic). It is well known that $\mathcal{C}_{\textup{per}}(Y)$ satisfies assumptions of Theorem \ref{usetheo} (see \cite{gabri}%
	).  Then the stochastic-deterministic homogenization problem (\ref{lem26}) is equivalent to the stochastic-periodic homogenization problem  which is treated in \cite{tchin2} (with $\varepsilon_{1}=\varepsilon$ and $\varepsilon_{2}=\varepsilon^{2}$).

\begin{remark}
	Indeed, when we consider the particular dynamical system $T(x)$ on $\Omega = \mathbb{T}^{d} \equiv \mathbb{R}^{d}/\mathbb{Z}^{d}$ (the $d$-dimensional torus) defined by $T(x)\omega = x+\omega\;\textup{mod}\;\mathbb{Z}^{d}$, then one can view $\Omega$ as the unit cube in $\mathbb{R}^{d}$ with all the pairs of antipodal faces being identified. The Lebesgue measure on $\mathbb{R}^{d}$ induces the Haar measure on $\mathbb{T}^{d}$ which is invariant with respect to the action of $T(x)$ on $\mathbb{T}^{d}$. Moreover, $T(x)$ is ergodic and in this situation, any function on $\Omega$ may be regarded as a periodic function on $\mathbb{R}^{d}$ whose period in each coordinate is 1, so that in this case our integrand $f$ may be viewed as a periodic function with respect to the variable $\omega$. Thus, the stochastic-deterministic homogenization problem (\ref{lem26}) is equivalent to the reiterated-periodic homogenization problem (with $\varepsilon_{1}=\varepsilon$ and $\varepsilon_{2}=\varepsilon^{2}$),
	\begin{equation*}
		\min \left\{ \int_{Q} f\left(x,\frac{x}{\varepsilon},\frac{x}{\varepsilon^{2}}, Dv(x) \right) dx \; : \; v \in  W^{1}_{0}L^{\Phi}(Q) \right\},
	\end{equation*}
	whose a particular case is treated in \cite{tacha3}. 
\end{remark}

	\subsection{The coupled stochastic-almost periodic homogenization problem}
	
	Here, we assume that 	for each fixed $(x,\omega
	,\lambda )\in Q\times \Omega \times \mathbb{R}^{d}$, the functions $y\mapsto
	f(x,\omega ,y,\lambda )$ in (\ref{lem13}) satisfy:
	\begin{equation*}
		f(x,\omega ,\cdot ,\lambda )\in AP(\mathbb{R}^{d})\text{ for any }%
		(x,\omega ,\lambda )\in \overline{Q}\times \Omega \times \mathbb{R}^{d};
	\end{equation*}
	where here, $AP(\mathbb{R}^{d})$ (see Example \ref{p2.5}) is the algebra of all Bohr almost periodic
	complex functions defined as the algebra of functions on $\mathbb{R}%
	^{d}$ that are uniformly approximated by finite linear combinations of
	functions in the set $\{\gamma _{k}:k\in \mathbb{R}^{d}\}$ with $\gamma
	_{k}(y)=\exp (2i\pi k\cdot y)$ ($y\in \mathbb{R}^{d}$). It is known that $AP(%
	\mathbb{R}^{d})$ satisfies assumptions of Theorem \ref{usetheo} (see \cite{bohr,gabri}%
	). We are led to the homogenization of (\ref{lem26}) with $A=AP(\mathbb{R}%
	^{d}) $.
	
	\subsection{The coupled stochastic-weakly almost periodic homogenization problem}
	
	Here, we assume that 	for each fixed $(x,\omega
	,\lambda )\in Q\times \Omega \times \mathbb{R}^{d}$, the functions $y\mapsto
	f(x,\omega ,y,\lambda )$ in (\ref{lem13}) satisfy:
	\begin{equation*}
		f(x,\omega ,\cdot ,\lambda )\in WAP(\mathbb{R}^{d})\text{ for any }%
		(x,\omega ,\lambda )\in \overline{Q}\times \Omega \times \mathbb{R}^{d};
	\end{equation*}
	where $WAP(\mathbb{R}^{d})$ (see Example \ref{p2.5}) is the algebra of weakly almost periodic
	functions on $\mathbb{R}^{d}$. It is known (see \cite{gabri}) that $%
	WAP(\mathbb{R}^{d})$ satisfies hypotheses of Theorem \ref{usetheo}. This leads to the
	homogenization of (\ref{lem26}) with $A=WAP(\mathbb{R}^{d})$.
	
		\subsection{The coupled stochastic-infinity bounded homogenization problem I}
	
	Here, we assume that 	for each fixed $(x,\omega
	,\lambda )\in Q\times \Omega \times \mathbb{R}^{d}$, the functions $y\mapsto
	f(x,\omega ,y,\lambda )$ in (\ref{lem13}) satisfy:
	\begin{equation*}
		f(x,\omega ,\cdot ,\lambda )\in \mathcal{B}_{\infty}(\mathbb{R}^{d})\text{ for any }%
		(x,\omega ,\lambda )\in \overline{Q}\times \Omega \times \mathbb{R}^{d};
	\end{equation*}
	where $\mathcal{B}_{\infty}(\mathbb{R}^{d})$ (see Example \ref{p2.5}) is the algebra of infinity bounded
	functions on $\mathbb{R}^{d}$. It is known (see \cite{gabri}) that $%
	\mathcal{B}_{\infty}(\mathbb{R}^{d})$ satisfies hypotheses of Theorem \ref{usetheo}. This leads to the
	homogenization of (\ref{lem26}) with $A=\mathcal{B}_{\infty}(\mathbb{R}^{d})$.
	
		\subsection{The coupled stochastic-infinity bounded homogenization problem II}
	
	Here, we assume that 	for each fixed $(x,\omega
	,\lambda )\in Q\times \Omega \times \mathbb{R}^{d}$, the functions $y\mapsto
	f(x,\omega ,y,\lambda )$ in (\ref{lem13}) satisfy:
	\begin{equation*}
		f(x,\omega ,\cdot ,\lambda )\in \mathcal{B}_{\infty,\mathbb{Z}^{d}}(\mathbb{R}^{d})\text{ for any }%
		(x,\omega ,\lambda )\in \overline{Q}\times \Omega \times \mathbb{R}^{d};
	\end{equation*}
	where $\mathcal{B}_{\infty,\mathbb{Z}^{d}}(\mathbb{R}^{d})$ is defined as in Example \ref{p2.5}. It is known (see \cite{gabri}) that $%
	\mathcal{B}_{\infty,\mathbb{Z}^{d}}(\mathbb{R}^{d})$ satisfies hypotheses of Theorem \ref{usetheo}. This leads to the
	homogenization of (\ref{lem26}) with $A=\mathcal{B}_{\infty,\mathbb{Z}^{d}}(\mathbb{R}^{d})$.
	
		\subsection{The coupled stochastic-Fourier Stieltjes homogenization problem}
	
	Here, we assume that 	for each fixed $(x,\omega
	,\lambda )\in Q\times \Omega \times \mathbb{R}^{d}$, the functions $y\mapsto
	f(x,\omega ,y,\lambda )$ in (\ref{lem13}) satisfy:
	\begin{equation*}
		f(x,\omega ,\cdot ,\lambda )\in FS(\mathbb{R}^{d}_{y}) \text{ for any }%
		(x,\omega ,\lambda )\in \overline{Q}\times \Omega \times \mathbb{R}^{d};
	\end{equation*}
	where $FS(\mathbb{R}^{d}_{y})$ is the Fourier Stieltjes algebra  defined as in Example \ref{p2.5}. It is known (see \cite{gabri}) that $%
	FS(\mathbb{R}^{d}_{y})$ satisfies hypotheses of Theorem \ref{usetheo}. This leads to the
	homogenization of (\ref{lem26}) with $A=\mathcal{B}_{\infty,\mathbb{Z}^{d}}(\mathbb{R}^{d})$.
	
	\section*{Acknowledgments}
	
	The authors would like to thank the anonymous referee for his/her pertinent remarks, comments and suggestions.

	\section*{Declarations}
	
	\begin{itemize}
		\item Funding : No funding was received to to assist with the preparation of this manuscript.
		\item Conflict of interest/Competing interests : We have no conflicts of interest to disclose. 
		\item Consent to participate : All authors consented to participate in this work.
		\item Consent for publication : We are enclosing herewith a manuscript entitled ``Stochastic $\Sigma$-convergence in Orlicz setting and Applications" submitted to the journal ``++++++" for possible evaluation. 
		\item Ethics approval : With the submission of this manuscript we would like to undertake that the above mentioned manuscript has not been published elsewhere, accepted for publication elsewhere or under editorial review for publication elsewhere. 
		\item Availability of data and materials : `Not applicable'
		\item Code availability : `Not applicable'
		\item Authors contributions : The authors contributed equally to this work. The first draft of the manuscript was written by \textsc{ Tchinda Takougoum Franck Arnold} and all authors commented on previous versions of the manuscript. All authors read and approved the final manuscript.
	\end{itemize}
	

	\bibliographystyle{abbrv}
	\bibliography{stoc_sigma_biblio1}
	
\end{document}